\documentclass[12 pt]{article}

\usepackage[T1]{fontenc}
\usepackage{ulem}
\usepackage[utf8]{luainputenc}
\usepackage{geometry}
\usepackage[pdftex]{graphicx}
\usepackage{amstext}
\usepackage{amsthm}
\usepackage{amssymb}
\usepackage{amsmath}
\usepackage[T1,T2A]{fontenc}
\usepackage[utf8]{inputenc}
\usepackage[english]{babel}
\usepackage{setspace}

\usepackage{esint}
\usepackage{comment}
\usepackage{babel}
\usepackage{yfonts}
\usepackage{float}
\usepackage{amsfonts}
\usepackage{fullpage} 
\usepackage{parskip} 
\usepackage{tikz} 
\usepackage{indentfirst}
\usepackage{subcaption}
\usepackage[nottoc]{tocbibind}
\usepackage[all]{xy}
\usepackage{extarrows}
\usepackage{overpic}
\usepackage{contour}
\usepackage{mathrsfs}
\usepackage{makecell}
\usepackage{ulem}
\usepackage{nicematrix}
\usepackage{url}

\theoremstyle{definition}
\newtheorem{theorem}{Theorem}[section]
\newtheorem{lemma}[theorem]{Lemma}
\newtheorem{corollary}[theorem]{Corollary}
\newtheorem{proposition}[theorem]{Proposition}
\newtheorem{definition}[theorem]{Definition}
\newtheorem{example}[theorem]{Example}
\newtheorem{remark}[theorem]{Remark}

\newcommand{\rr}{\mathbb{R}}
\newcommand{\rp}{\mathbb{R}P^1}
\newcommand{\cc}{\mathbb{C}}
\newcommand{\cp}{\mathbb{C}P^1}

\providecommand{\keywords}[1]
{
  \small	
  \textbf{\textit{Keywords:}} #1
}

\title{On the Algebraic Classification of Non-singular\\Flexible Kokotsakis Polyhedra}
\author{Yang Liu$^{1}$, Yi Ouyang$^{2}$, Dominik L. Michels$^{1}$  \\
        \small $^{1}$KAUST Computational Sciences Group \\
        \small \{yang.liu.4, dominik.michels\}@kaust.edu.sa \\
        \small $^{2}$USTC School of Mathematical Sciences \\
        \small yiouyang@ustc.edu.cn \\
        }
\date{January 24, 2024}

\begin{document}

\maketitle
\begin{abstract}
\noindent
Across various scientific and engineering domains, a growing interest in flexible and deployable structures is becoming evident. These structures facilitate seamless transitions between distinct states of shape and find broad applicability ranging from robotics and solar cells to meta-materials and architecture. In this contribution, we study a class of mechanisms known as Kokotsakis polyhedra with a quadrangular base. These are $3\times3$ quadrilateral meshes whose faces are rigid bodies and joined by hinges at the common edges. Compared to prior work, the quadrilateral faces do not have to be planar. In general, such meshes are not flexible, and the problem of finding and classifying the flexible ones is old, but until now largely unsolved. It appears that the tangent values of the dihedral angles between different faces are algebraically related through polynomials. Specifically, by fixing one angle as a parameter, the others can be parameterized algebraically and hence belong to an extended rational function field of the parameter. We use this approach to characterize shape restrictions resulting in flexible polyhedra.
\end{abstract}

\keywords{Factorization of resultant, field extension, flexible meshes, kinematic geometry, Kokotsakis polyhedron, spherical linkage.}\\

\section{Introduction}
Flexible geometric structures play a crucial role in various fields, including engineering, architecture, and material science, owing to their adaptability and versatility \cite{hoberman2005,dieleman2020,Callens2018,metamaterials-survey,ciang-2019-cf,Dudte2016,konakovic2016,konakovic2018,jamali2022,xiao2022}. These structures can deform and change shape in response to external forces, enabling them to navigate complex environments or absorb impacts. In engineering, flexible geometric structures are employed in robotics and machinery, enhancing the efficiency and agility of machines. In architecture, they offer innovative solutions for dynamic and responsive structures that can adapt to changing environmental conditions. Additionally, in material science, the development of flexible materials with geometric patterns opens up new possibilities for creating lightweight and durable materials with applications in aerospace, transportation, and beyond. The relevance of flexible geometric structures lies in their ability to provide adaptable solutions to diverse challenges, fostering advancements in technology and design.

Flexible and deployable structures may contain non-rigid parts \cite{flectofin,strelitzia12,Pillwein2020,Pillwein2020a,soriano2019g,panetta2019x,guseinov2017,malomo2018}. However, in this paper, we are only dealing with
mechanisms in the classical sense. They are formed by
rigid bodies connected by hinges. Our interest is in
flexible quadrilateral meshes, whose faces are rigid and joined by hinges at the common edges. It is well known
that such a quadrilateral mesh is a mechanism if and only if
any $3 \times 3$ sub-mesh is flexible. This has been
shown by Schief et al. \cite{Schief2008}, who formulate
the problem for planar faces, but planarity does not enter
their proof. Therefore, the first key step in the determination of all flexible quadrilateral meshes is the {\it classification
of all flexible $3 \times 3$ quadrilateral meshes}, which is the
main contribution of the present paper.

The problem of flexible polyhedra can be traced back to the last century. Very
similar mechanisms to the ones studied in our paper have 
been introduced by Kokotsakis \cite{kokotsakis33}. His
mechanisms possess a central planar, but not necessarily quadrilateral face, surrounded by a belt of planar faces.
The work of Kokotsakis inspired further research on such structures \cite{stachel2010}, in particular with a focus on 
origami \cite{Tac09a,tachi-2011-onedof,tachi-2013-composite}. 

A breakthrough in this area has been made by Izmestiev \cite{izmestiev-2017} who derived a complete classification of flexible $3 \times 3$ quadrilateral meshes with planar faces. It is, however, still unknown
which of these types can be parts of larger flexible $n \times m$ quadrilateral meshes. Larger flexible quadrilateral meshes with planar faces have
been studied by Schief et al. \cite{Schief2008} as discrete
integrable systems and  counterparts to isometric deformations of
conjugate nets on smooth surfaces.
Their results are largely rooted in second-order infinitesimal flexibility.
A wealth of results on first-order infinitesimal flexibility
 is found in the book
by Sauer \cite{sauer:1970}. There, we also find a detailed 
study of special quadrilateral meshes with planar faces that are
mechanisms. They come in two types. One type,
the so-called
Voss meshes, are discrete counterparts to surfaces
with a conjugate net of geodesics, first described
by Voss \cite{voss1888uber} (see also \cite{Schief2008}).
The other type is the
T-nets, first described by Graf and Sauer \cite{sauer-graf} (for a recent description
in English see \cite{t-nets-IASS}), are discrete counterparts to an
affine generalization of moulding surfaces. A very special case of T-nets is the famous Miura origami
structures \cite{miura-1980}. T-nets are also capable
of forming tubular flexible structures and flat-foldable metamaterials \cite{nawratil-t-nets2023}. The isometric deformations of T-nets
and their smooth counterparts have recently been studied by Izmestiev et al. \cite{izmestiev-t-nets}.

The first non-trivial example of a flexible Kokotsakis polyhedron with skew quadrilateral faces was constructed by Nawratil in 2022 \cite{nawratil-2022}. The question of existence has already been posed by Sauer \cite{sauer:1970}, but remained unsolved for a long time. 

In this paper, we are going to systematically study Kokotsakis polyhedra with a quadrangular base and arbitrary,
not necessarily planar faces. We refer to these geometric objects as skew-quadrilateral Kokotsakis meshes, but point out that planarity of faces is
included as a special case. We utilize the classical approach proposed by Bricard \cite{bricard1897}  to reformulate the flexibility problem in Euclidean space to one on the sphere. We can apply powerful methods from algebra to proceed with the polynomial system that arises after turning to new variables, namely tangents of adjacent dihedral half-angles. Algebraically, the condition for continuous flexion implies the existence of a one-parameter real-valued solution set of the polynomial system.

\subsection{Main approach and contributions}

Our main object is the mesh\footnote{In this article 'mesh' is always meant $3\times3$ quadrilateral mesh by default.} described in Fig.~\ref{mesh} left, which contains nine quadrilaterals linked by hinges that allow for linked units to rotate. The whole mesh is generally still rigid. We are dedicated to finding a classification for all flexible ones.  In fact, during flexion, the values of dihedral angles along hinges change dependently on each other. The motion can be described by zeros of polynomials and therefore, each mesh uniquely determines a polynomial ideal
$$
M=(\Tilde{g}^{(1)}(x_1,x_2), \Tilde{g}^{(2)}(x_2,x_3), \Tilde{g}^{(3)}(x_3,x_4), \Tilde{g}^{(4)}(x_4,x_1))\subset\cc[x_1,x_2,x_3,x_4]
$$
coming with a zero set
$$
Z(M):=\{(x_1,x_2,x_3,x_4)\in (\cp)^4: f(x_1,x_2,x_3,x_4)=0, \forall f \in M\}.\footnote{Kindly notice that $(\cp)^n\neq \mathbb{C} P^n$ in general. We consider projective varieties because $x_i$ actually stand for tangents of angles.}
$$
The flexibility of the mesh refers to an infinite $Z(M)$: a continuous trajectory of the motion. Notice that $Z(M)$ can be also obtained by the fiber product
\begin{equation}\label{fiber}
  Z(M)=Z(S_1)\times_{\{x_1, x_3\}} Z(S_2):=\{(x_i)_{i=1}^4: (x_1,x_2,x_3)\in Z(S_1), (x_1,x_4,x_3)\in Z(S_2)\}
\end{equation}
where the ideals $S_1=(\Tilde{g}^{(1)}, \Tilde{g}^{(2)})$ and $S_2=(\Tilde{g}^{(3)}, \Tilde{g}^{(4)})$.
The classification of flexible meshes (i.e. the ideal $M$) proceeds in three main steps:
\begin{itemize}
    \item [$\bullet$] Study the components of $Z(S_i)$ in Zariski topology;
    \item [$\bullet$] Classify $Z(S_i)$ according to its components;
    \item [$\bullet$] Investigate how to choose $S_i$ such that the fiber product in equation (\ref{fiber}) is an infinite set, hence the classification of flexible meshes can be conducted on $M$.
\end{itemize}

We only focus on non-singular meshes in which $Z(S_i)$ has no singular component. That is, regarding $Z(S_i)$ as a bunch of irreducible curves, none of them has a constant entry, e.g. 
$$
Z(x^2-xy)=Z(x)\cup Z(x-y)
$$
is singular since $Z(x)=\{(0,y): y\in \cc\}$ is singular. Non-singular meshes are the majority and in particular, each polynomial $\Tilde{g}^{(i)}$ of non-singular meshes has a very specific form. It is quadratic in both $x_i$ and $x_{i+1}$ and has no factors in $\cc[x_i]$ or $\cc[x_{i+1}]$. This means $Z(S_1)$ is a finite collection of irreducible curves (1-dimensional) and each curve can be \underline{locally} parameterized by one of the three coordinates, say $x_1$, so that $x_2, x_3$ are just algebraic functions of $x_1$ derived from $\Tilde{g}^{(1)}=\Tilde{g}^{(2)}=0$ and hence can be treated as algebraic elements of some extended field over the rational function field $K_i:=\cc(x_i)$. Given $\Tilde{g}^{(i)}$ quadratic, the 2-steps extension
$$
\xymatrix{
    K_1 \ar@{-}[r]^{\Tilde{g}^{(1)}} & K_1(x_2)\ar@{-}[r]^{\Tilde{g}^{(2)}} & K_1(x_2,x_3) & \text{or} & K_3 \ar@{-}[r]^{\Tilde{g}^{(2)}} & K_3(x_2)\ar@{-}[r]^{\Tilde{g}^{(1)}} & K_3(x_1,x_2)
    }
$$
only provides the following possibilities of extension degrees
$$
[K_1(x_2):K_1]\in \{1,2\},\,\,[K_1(x_3):K_1]\in\{1,2,4\},
$$
which leads to 6 situations:
\begin{itemize}
        \item \textbf{purely-rational}: if $[K_1(x_2):K_1]=[K_1(x_3):K_1]=1$, i.e. both $\Tilde{g}^{(i)}$ are reducible;
        \item \textbf{half-quadratic}: if $[K_1(x_2):K_1]=1$ and $[K_1(x_3):K_1]=2$, or, $[K_3(x_2):K_3]=1$ and $[K_3(x_1):K_3]=2$, i.e.
        only $\Tilde{g}^{(1)}$ is reducible or only $\Tilde{g}^{(2)}$ is reducible;
    \end{itemize} 
    Further, when both $\Tilde{g}^{(1)}, \Tilde{g}^{(2)}$ are irreducible,
    \begin{itemize}
        \item \textbf{involutive-rational}: if $[K_1(x_2):K_1]=2$ and $[K_1(x_3):K_1]=1$;
        \item \textbf{purely-quadratic}: if $[K_1(x_2):K_1]=[K_1(x_3):K_1]=2$ and $K_1(x_2)=K_1(x_3)$;
        \item \textbf{involutive-quadratic}: if $[K_1(x_2):K_1]=[K_1(x_3):K_1]=2$ but $K_1(x_2)\neq K_1(x_3)$;
        \item \textbf{quartic}: if $[K_1(x_3):K_1]=4$.
    \end{itemize} 
In the meanwhile, each component of $Z(S_1)$, since known as an irreducible algebraic curve, can be determined by its local behavior only! The above information therefore becomes a global invariant of the component and thus all components fall into one of 6 cases.

\begin{equation}\label{c-table}
    \begin{tabular}{ |c|c|c|c|c|c| }
        \hline
        \thead{purely- \\ rational} & \thead{half- \\ quadratic} & \thead{involutive- \\ rational} & \thead{purely- \\ quadratic} & \thead{involutive- \\ quadratic} & \thead{quartic} \\ 
        \hline
        \textbf{Case 1} & \textbf{Case 2} & \textbf{Case 3} & \textbf{Case 4} & \textbf{Case 5} & \textbf{Case 6}\\
        \hline
    \end{tabular} 
\end{equation}    
According to our theory, $S_i$ will fall into one of 7 classes, based on the components of $Z(S_i)$.
\begin{equation}\label{table}
\begin{NiceTabular}{|c|c|c|c|c|c|c|c|}[hvlines]
    \Block{2-1}{$S_i$} & \Block{2-1}{\thead{purely- \\ rational}} & \Block{2-1}{\thead{half- \\ quadratic}} & \Block{1-3}{equimodular} & & & \Block{2-1}{\thead{involutive- \\ quadratic}} & \Block{2-1}{\thead{quartic}} \\
    & & & \thead{involutive- \\ rational} & \thead{rational- \\ quadratic} & \thead{purely- \\ quadratic} & & \\
    \thead{comp. \\ types} & \textbf{Case 1} & \textbf{Case 2} & \textbf{Case 3} & \textbf{Case 3}, \textbf{4} & \textbf{Case 4} & \textbf{Case 5} & \textbf{Case 6}\\
    \end{NiceTabular}
\end{equation}
The term 'equimodular' was introduced by Izmestiev in \cite{izmestiev-2017}, which we adopted and generalized to non-planar meshes. It covers three classes of ours.
However, the classification of the mesh or $M$ is not simply developed by considering all combinations from table (\ref{table}). This is because the fiber product of $Z(S_1)$ and $Z(S_2)$ is simply a collection of the fiber products of their components, and to make $Z(M)$ an infinite set, some products of the components are finite and hence out of our interest. So the classification of the flexible mesh depends on the components $W_i$ of $Z(S_i)$ that we can choose to make $W_1\times_{\{x_1, x_3\}} W_2$ an infinite set, e.g. 
$$
Z(x^2-y^2)\times_{\{x,y\}}Z(x^2-3xy+2y^2)=(Z(x-y)\cup Z(x+y))\times_{\{x,y\}}(Z(x-y)\cup Z(x-2y))
$$
is infinite only because $Z(x-y)\times_{\{x,y\}}Z(x-y)$ is infinite.
In a nutshell, all flexible meshes are classified by those 'contributing' components of $Z(S_i)$.

\begin{theorem}[Main Theorem]\label{main}
Every $3\times3$ non-singular flexible mesh must belong to one of the following classes:
\begin{equation}\label{classes}
\textbf{simple classes}\left\{\begin{array}{l}
                                    \textbf{PR}: \text{purely-rational;} \\
                                    \textbf{HQ}: \text{half-quadratic;}  \\
                                    \textbf{IR}: \text{involutive-rational;}  \\
                                    \textbf{RQ}: \text{rational-quadratic;}  \\
                                    \textbf{PQ}: \text{purely-quadratic;} \\
                                    \textbf{IQ}: \text{involutive-quadratic;} \\
                                    \textbf{Q}: \text{quartic;} \\
                                \end{array}\right.
\textbf{hybrid classes}\left\{\begin{array}{l}
                                    \textbf{PR + IR};\\
                                    \textbf{HQ + IQ};\\
                                    \textbf{HQ + PQ};\\
                                    \textbf{PQ + IQ}.\\
                                \end{array}\right.
\end{equation}
\end{theorem}
By \textbf{simple classes} we mean $W_1, W_2$ belong to the same \textbf{Case} of table (\ref{c-table}), and by \textbf{hybrid classes} we mean the opposite. To avoid misunderstanding, we will use abbreviations such as \textbf{PR}, \textbf{PQ + IQ} etc. ONLY on meshes.

\subsection{Organization of the paper}

In Section \ref{setup} we introduce non-singular couplings. Roughly speaking, a coupling is just a polynomial ideal $S_1=(\Tilde{g}^{(1)}(x_1,x_2), \Tilde{g}^{(2)}(x_2,x_3))$ and also the building block of the mesh. We show how to convert a mesh to an ideal $M\subset \cc[x_1,x_2,x_3,x_4]$ such that
$$
M=(\Tilde{g}^{(1)}(x_1,x_2), \Tilde{g}^{(2)}(x_2,x_3))+(\Tilde{g}^{(3)}(x_3,x_4), \Tilde{g}^{(4)}(x_4,x_1))
$$
is the sum of two couplings. To study non-singular couplings, we propose a brand new method that is based on the decomposition of the zero set $Z(S_1)$. 

In Section \ref{Component}, the components are classified into 6 cases. We also provide a construction theorem of the simplest class of flexible meshes as a tutorial for our new classification method.

We further split the non-singular couplings into two types which are studied in Section \ref{pseudo} and \ref{General} respectively. This will explain how $S_i$ can be classified by table (\ref{table}).

In Section \ref{Res}, we use previous results to factorize the resultant $\text{Res}(\Tilde{g}^{(1)},\Tilde{g}^{(2)}; x_2)$ to get ready for the classification of flexible meshes.

Finally in Section \ref{back-to-com}, we consider the construction of flexible meshes. Starting from a given half $S_1=(\Tilde{g}^{(1)}, \Tilde{g}^{(2)})$, we find restrictions for the other half $S_2=(\Tilde{g}^{(3)}, \Tilde{g}^{(4)})$ in order to form a flexible mesh, i.e. to ensure that $M=S_1+S_2$ has an infinite zero set $Z(M)$. This will tell why Theorem \ref{main} covers all possibilities.

We delayed most of the proofs to Appendix \ref{techpf} so that the readers can have quick access to the theory. The remaining proofs are important to comprehending some main ideas. One can also download a Maple\texttrademark\,\,script\footnote{\url{https://drive.google.com/file/d/10yn_3bVrJNWD-RsXa-ltJqJvaDwgghD5/view?usp=sharing}} to verify the calculations of the proofs.

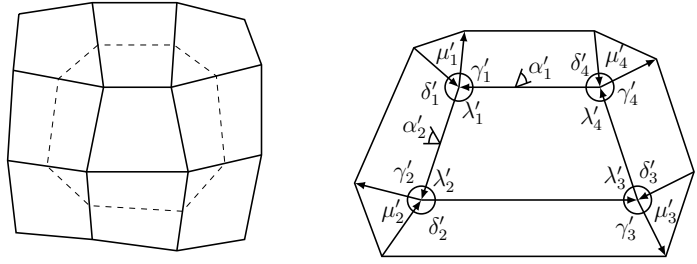
\begin{figure}
    \centering
    \scalebox{0.75}{
 \begin{tikzpicture}[>=latex]

        \coordinate (A1) at (0.2, 4.3);
        \coordinate (A2) at (0.1, 3.3);
        \coordinate (A3) at (0, 1.7);
        \coordinate (A4) at (0.15, 0.42);
        \coordinate (A5) at (1.5, 0.3);
        \coordinate (A6) at (3, 0.1);
        \coordinate (A7) at (4.2, 0.3);
        \coordinate (A8) at (4.5, 1.8);
        \coordinate (A9) at (4.5, 3.4);
        \coordinate (A10) at (4.2, 4.2);
        \coordinate (A11) at (3, 4.5);
        \coordinate (A12) at (1.5, 4.5);

        \coordinate (B1) at (1.7, 3);
        \coordinate (B2) at (1.4, 1.5);
        \coordinate (B3) at (3.2, 1.5);
        \coordinate (B4) at (2.9, 3);

        \coordinate (C1) at (1.6,3.75);
        \coordinate (C2) at (0.9,3.15);
        \coordinate (C3) at (0.7,1.6);
        \coordinate (C4) at (1.45,0.9);
        \coordinate (C5) at (3.1,0.8);
        \coordinate (C6) at (3.85,1.65);
        \coordinate (C7) at (3.7,3.2);
        \coordinate (C8) at (2.95,3.75);

        \draw[thick] (A1) -- (A2) -- (A3) --(A4) --(A5) --(A6) --(A7) --(A8) --(A9) --(A10) --(A11) --(A12) --(A1);
        \draw[thick] (B1) -- (B2) -- (B3) --(B4) --(B1);
        \draw[dashed] (C1) -- (C2) -- (C3) --(C4) --(C5) --(C6) --(C7) --(C8) --(C1);
        
        \draw[thick] (B1) -- (A12);
        \draw[thick] (B1) -- (A2);
        
        \draw[thick] (B2) -- (A3);
        \draw[thick] (B2) -- (A5);
        
        \draw[thick] (B3) -- (A6);
        \draw[thick] (B3) -- (A8);

        \draw[thick] (B4) -- (A9);
        \draw[thick] (B4) -- (A11);

        \coordinate (a3) at (11.166667, 1);
        \coordinate (a2) at (7.333333, 1);
        \coordinate (a1) at (8, 3);
        \coordinate (a4) at (10.5, 3);
        \coordinate (b3) at (12.166667, 1.5);
        \coordinate (c3) at (11.666667, 0);
        \coordinate (b2) at (6.633333, 0);
        \coordinate (c2) at (6.15, 1.3);
        \coordinate (b1) at (7.2, 3.7);
        \coordinate (c1) at (8.1, 4);
        \coordinate (b4) at (10.4, 4);
        \coordinate (c4) at (11.5, 3.5);

        \draw[thick] (b3) -- (c3) -- (b2) -- (c2) -- (b1) -- (c1) -- (b4) -- (c4) -- cycle;
        
        \draw[thick,<-] (a3) -- (a2);
        \draw[thick,<-] (a2) -- (a1);
        \draw[thick,<-] (a1) -- (a4);
        \draw[thick,<-] (a4) -- (a3);
        
        \draw[thick,<-] (a3) -- (b3);
        \draw[thick,->] (a3) -- (c3);
        
        \draw[thick,<-] (a2) -- (b2);
        \draw[thick,->] (a2) -- (c2);
        
        \draw[thick,<-] (a1) -- (b1);
        \draw[thick,->] (a1) -- (c1);

        \draw[thick,<-] (a4) -- (b4);
        \draw[thick,->] (a4) -- (c4);

        \draw[thick] (a3) circle (7pt);
        \draw (a3)+(-0.4, 0.4) node {$\lambda_3'$};
        \draw (a3)+(0.2, 0.5) node {$\delta_3'$};
        \draw (a3)+(0.5, -0.2) node {$\mu_3'$};
        \draw (a3)+(-0.2, -0.5) node {$\gamma_3'$};

        \draw[thick] (a2) circle (7pt);
        \draw (a2) + (0.4, 0.35) node {$\lambda_2'$};
        \draw (a2) + (-0.3, 0.5) node {$\gamma_2'$};
        \draw (a2) + (-0.5, -0.2) node {$\mu_2'$};
        \draw (a2) + (0.3, -0.5) node {$\delta_2'$};

        \draw[thick] (a1) circle (7pt);
        \draw (a1) + (0.25, -0.45) node {$\lambda_1'$};
        \draw (a1) + (0.4, 0.3) node {$\gamma_1'$};
        \draw (a1) + (-0.2, 0.6) node {$\mu_1'$};
        \draw (a1) + (-0.5, -0.1) node {$\delta_1'$};
        
        \draw[thick] (a4) circle (7pt);
        \draw (a4) + (-0.15, -0.55) node {$\lambda_4'$};
        \draw (a4) + (0.3, 0.55) node {$\mu_4'$};
        \draw (a4) + (-0.35, 0.4) node {$\delta_4'$};
        \draw (a4) + (0.5, -0.1) node {$\gamma_4'$};
        
         \coordinate (alpha1) at (9, 3);
        \draw[thick, rotate around = {-20:(alpha1)}] (alpha1) -- + (0, 0.35);
        \draw[thick, rotate around = {20:(alpha1)}] (alpha1) -- + (0.3, 0);
        \draw[thick] (alpha1) + (0.2, .07) arc (0:60:6pt);
        \draw (alpha1) + (.45, .35) node {$\alpha_1'$};
        
        \coordinate (psi3) at (7.66667, 2);
        \draw[thick, , rotate around = {30:(psi3)}] (psi3) -- + (0, 0.35);
        \draw[thick, , rotate around = {0:(psi3)}] (psi3) -- + (-0.3, 0);
        \draw[thick] (psi3) + (-0.2, 0) arc (180:141:10pt);
        \draw (psi3) + (-0.45, 0.3) node {$\alpha_2'$};
        
    \end{tikzpicture}
}    
    \caption{Left: Sketch of a $3\times3$ quadrilateral mesh, or equivalently, Kokotsakis polyhedron with a quadrangular base. The flexibility only relies on a smaller area marked in dashed lines; Right: A flexible mesh with fixed angles $\lambda_i', \gamma_i', \mu_i', \delta_i'$, which are well-defined by vector products and $\arccos$ function. The flexible angles $\alpha_i'$ and its complement (not shown) $\alpha_i=\pi-\alpha_i'$ are dihedral angles between the central planar panel and surrounding planar panels, which are rigorously defined in Appendix \ref{dfnang}.}
    \label{mesh}
\end{figure}

\section{Basic setup}\label{setup}

\subsection{From $\rr^3$ to sphere}

In this subsection, we introduce the classic approach which has been used extensively among peers.

\subsubsection{Bricard equation}

\begin{figure}
    \scalebox{0.75}{
    \begin{tikzpicture}[>=latex]           %
        \coordinate (A3) at (11.5, 4.5);
        \coordinate (A2) at (6.5, 4.5);
        \coordinate (A1) at (7.5, 7.5);
        \coordinate (A4) at (10.5, 7.5);
        \coordinate (B3) at (12.5, 5);
        \coordinate (C3) at (12, 3.5);
        \coordinate (B2) at (5.8, 3.5);
        \coordinate (C2) at (5.65, 4.8);
        \coordinate (B1) at (6.7, 8.2);
        \coordinate (C1) at (7.6, 8.5);
        \coordinate (B4) at (10.4, 8.5);
        \coordinate (C4) at (11.5, 8);

        \draw[thick,dashed] (A2) -- (C3) -- (B3) -- (A4) -- (C1) -- (B1) -- cycle;
        \draw[thick] (A1) -- (C2) -- (B2) -- (A3) -- (C4) -- (B4) -- cycle;
        \draw[thick] (B3) -- (C3) -- (B2) -- (C2) -- (B1) -- (C1) -- (B4) -- (C4) -- cycle;
        
        \draw[thick,<-] (A3) -- (A2);
        \draw[thick,<-] (A2) -- (A1);
        \draw[thick,<-] (A1) -- (A4);
        \draw[thick,<-] (A4) -- (A3);
        \draw[thick] (A1) -- (A3);
        \draw[thick,dashed] (A4) -- (A2);
        
        \draw[thick,<-] (A3) -- (B3);
        \draw[thick,->] (A3) -- (C3);
        
        \draw[thick,<-] (A2) -- (B2);
        \draw[thick,->] (A2) -- (C2);
        
        \draw[thick,<-] (A1) -- (B1);
        \draw[thick,->] (A1) -- (C1);

        \draw[thick,<-] (A4) -- (B4);
        \draw[thick,->] (A4) -- (C4);

        \draw[thick,dashed] (A3) circle (7pt);
        \draw (A3) + (-0.4, 0.6) node {$\lambda_3'$};
        \draw (A3) + (0.3, 0.6) node {$\delta_3'$};
        \draw (A3) + (0.5, -0.2) node {$\mu_3'$};
        \draw (A3) + (-0.2, -0.5) node {$\gamma_3'$};

        \draw[thick] (A2) circle (7pt);
        \draw (A2) + (0.4, 0.35) node {$\lambda_2'$};
        \draw (A2) + (-0.3, 0.5) node {$\gamma_2'$};
        \draw (A2) + (-0.5, -0.2) node {$\mu_2'$};
        \draw (A2) + (0.3, -0.5) node {$\delta_2'$};

        \draw[thick,dashed] (A1) circle (7pt);
        \draw (A1) + (0.2, -0.45) node {$\lambda_1'$};
        \draw (A1) + (0.35, 0.4) node {$\gamma_1'$};
        \draw (A1) + (-0.2, 0.6) node {$\mu_1'$};
        \draw (A1) + (-0.5, -0.1) node {$\delta_1'$};
        
        \draw[thick] (A4) circle (7pt);
        \draw (A4) + (-0.1, -0.6) node {$\lambda_4'$};
        \draw (A4) + (0.3, 0.55) node {$\mu_4'$};
        \draw (A4) + (-0.45, 0.5) node {$\delta_4'$};
        \draw (A4) + (0.6, -0.2) node {$\gamma_4'$};

        \draw[thick, dashed] (6.97, 5.9) -- (6.6, 6);
        \draw[thick] (6.97, 5.9) -- (6.9, 6.2);
        \draw[thick] (6.72, 5.97) arc(170:132:10pt);
        \draw (7.2, 5.9) node {$\zeta_1$};

        \draw[thick] (7.17, 6.5) -- (7.5, 6.7);
        \draw[thick, dashed, rotate around ={-25:(7.17, 6.5)} ] (7.17, 6.5) -- (7.5, 6.5);
        \draw[thick] (7.4, 6.4) arc (-20:20:10pt);
        \draw (7.7, 6.5) node {$\tau_1$};

        \coordinate (x3) at (18, 4);
        \coordinate (x2) at (13, 4);
        \coordinate (x1) at (14, 7);
        \coordinate (x4) at (17, 7);
        \coordinate (y3) at (19, 4.9);
        \coordinate (z3) at (18.1, 4.7);
        \coordinate (y2) at (12.3, 3);
        \coordinate (z2) at (12.15, 4.3);
        \coordinate (y1) at (13.2, 7.7);
        \coordinate (z1) at (14.1, 8);
        \coordinate (y4) at (16.9, 8);
        \coordinate (z4) at (18, 7.5);

        \draw[thick] (x3) -- (x2) -- (x1) -- (x4) -- cycle;
        \draw[thick,dashed] (y3) -- (x4);
        \draw[thick,dashed] (z3) -- (x2);
        \draw[thick,dashed] (y1) -- (x2);
        \draw[thick,dashed] (z1) -- (x4);
        \draw[thick,dashed] (x2) -- (x4);
        
        \draw[thick] (y3) -- (z3);
        \draw[thick] (y1) -- (z1);
        \draw[thick] (x3) -- (y3);
        \draw[thick] (x3) -- (z3);
    
        \draw[thick] (x1) -- (y1);
        \draw[thick] (x1) -- (z1);

        \coordinate (phi1) at (17.5, 4);
        \draw[thick, rotate around = {30:(phi1)}] (phi1) -- + (-0.3,0);
        \draw[thick] (phi1) -- + (0, 0.3);
        \draw[thick] (phi1) + (-0.15, -0.07) arc (210:95:5pt);
        \draw (phi1) + (-0.4,0.25) node {$\alpha_3'$};

        \coordinate (phi2) at (17.6, 5.2);
        \draw[thick, rotate around = {60:(phi1)}] (phi2) -- + (0.4, 0);
        \draw[thick] (phi2) -- + (0.4, 0);
        \draw[thick] (phi2)+(0.3,0) arc (0:51:10pt);
        \draw (phi2) + (0.5,0.3) node {$\beta_3'$};

        \coordinate (phi4) at (14.5, 7);
        \draw[thick, rotate around = {20:(phi4)}] (phi4) -- + (0.4, 0);
        \draw[thick, rotate around = {70:(phi4)}] (phi4) -- + (0.5, 0);
        \draw[thick] (phi4) -- + (0.5, 0);
        \draw[thick] (phi4)+(0.3,0.1) arc (0:100:5pt);
        \draw (phi4) + (0.4,0.45) node {$\alpha_1'$};
        
        \draw[thick,dashed] (x1) circle (7pt);

        \coordinate (phi3) at (13.77, 6.3);
        \draw[thick, rotate around = {-70:(phi3)}] (phi3) -- + (-0.4,0);
        \draw[thick] (phi3) -- + (-0.35, 0);
        \draw[thick] (phi3) + (-0.25, 0) arc (180:125:10pt);
        \draw (phi3) + (-0.4,0.4) node {$\beta_1'$};

        \coordinate (a3) at (25, 4);
        \coordinate (a2) at (20, 4);
        \coordinate (a1) at (21, 7);
        \coordinate (a4) at (24, 7);
        \coordinate (b3) at (25, 4.5);
        \coordinate (c3) at (25.5, 3);
        \coordinate (b2) at (21, 4.7);
        \coordinate (c2) at (19.4, 5.2);
        \coordinate (b1) at (20.2, 7.7);
        \coordinate (c1) at (21.1, 8);
        \coordinate (b4) at (23.9, 8);
        \coordinate (c4) at (25, 7.5);

        \draw[thick] (a3) -- (a2) -- (a1) -- (a4) -- cycle;
        
        \draw[thick] (a1) -- (a3);
        \draw[thick] (a2) -- (b2);
        \draw[thick] (a2) -- (c2);
        \draw[thick] (a4) -- (b4);
        \draw[thick] (a4) -- (c4);
        \draw[thick] (a3) -- (c4); 
        \draw[thick] (a3) -- (b2);
        \draw[thick] (c2) -- (a1);
        \draw[thick] (b4) -- (a1);
        \draw[thick] (b2) -- (c2);
        \draw[thick] (b4) -- (c4);

        \coordinate (ps2) at (20.7, 4);
        \draw[thick, , rotate around = {-40:(ps2)}] (ps2) -- + (0, 0.35);
        \draw[thick, rotate around = {-40:(ps2)}] (ps2) -- + (0.35, 0);
        \draw[thick] (ps2) + (0.2, -0.15) arc (-40:50:7pt);
        \draw (ps2) + (0.55, 0.25) node {$\beta_2'$};

        \coordinate (ps3) at (20.4, 5.2);
        \draw[thick, , rotate around = {30:(ps3)}] (ps3) -- + (0, 0.35);
        \draw[thick, , rotate around = {0:(ps3)}] (ps3) -- + (-0.3, 0);
        \draw[thick] (ps3) + (-0.2, 0) arc (180:141:10pt);
        \draw (ps3) + (-0.45, 0.3) node {$\alpha_2'$};

        \coordinate (ps1) at (24.2, 6.4);
        \draw[thick, rotate around = {-50:(ps1)}] (ps1) -- + (0, 0.35);
        \draw[thick, rotate around = {-30:(ps1)}] (ps1) -- + (0.3, 0);
        \draw[thick] (ps1) + (0.2, -0.1) arc (-30:36:7pt);
        \draw (ps1) + (.55, .1) node {$\alpha_4'$};

        \coordinate (ps4) at (23, 7);
        \draw[thick, rotate around = {-20:(ps4)}] (ps4) -- + (0, 0.35);
        \draw[thick, rotate around = {20:(ps4)}] (ps4) -- + (0.3, 0);
        \draw[thick] (ps4) + (0.2, .07) arc (0:60:6pt);
        \draw (ps4) + (.45, .35) node {$\beta_4'$};

        \draw[thick] (a4) circle (7pt);
     \end{tikzpicture}
     }
    \caption{A non-planar flexible mesh (left) and its decomposition (middle and right). $(\lambda_i', \gamma_i', \mu_i', \delta_i')$ are angles between corresponding oriented edges; $\tau_i, \zeta_i$ are fixed dihedral angles along hinges of the central tetrahedron and the attached tetrahedron respectively; $\alpha_i', \beta_i'$ are flexible dihedral angles (in planar case $\alpha_{i+1}'=\beta_i'$). All angles can be defined in a similar way (see Appendix \ref{dfnang}).}
    \label{sqmesh}
\end{figure}
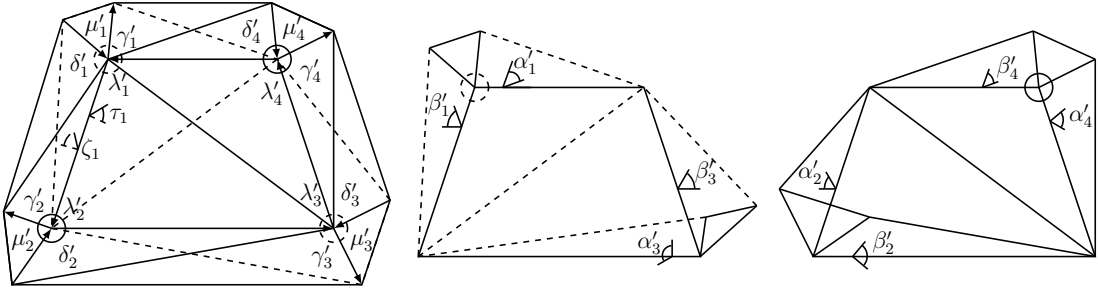

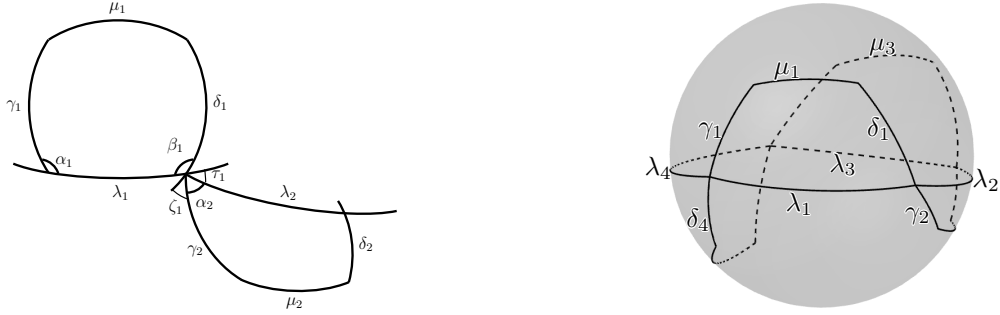
\begin{figure}[t]
    \begin{minipage}{0.5\textwidth}
          \scalebox{0.55}{  
      \begin{tikzpicture}[>=latex]
     \draw[ultra thick] (0.5, 4) arc[start angle=210, end angle = 330, x radius = 3cm, y radius = 0.7cm]  node[below, midway] {$\lambda_1$};
        \draw[ultra thick] (1.35, 7) arc [start angle=140, end angle=220,x radius = 2cm, y radius = 2.5cm] node[left, midway] {$\gamma_1$};
        \draw[ultra thick] (4.7, 7) arc [start angle=40, end angle=-55,x radius = 2cm, y radius = 2.5cm] node[right, pos = 0.4] {$\delta_1$};
        \draw[ultra thick] (4.7,7) arc [start angle=50, end angle=130,x radius =2.6cm, y radius = 2cm] node[above, midway] {$\mu_1$};
        \draw[ultra thick] (4.85, 4.1) arc[start angle = 95, end angle = 165,radius=0.5cm] node[above left,pos=0.1] {$\beta_1$};
        \draw[ultra thick] (1.2, 4.1) arc[start angle = 80, end angle = 18,radius=0.5cm] node[right, pos = 0.3] {$\alpha_1$};

        \draw[ultra thick, rotate around={-10:(4.7,3.75)}] (4.65, 3.75) arc [start angle=210, end angle = 330, x radius = 3cm, y radius = 0.7cm] node[above, midway] {$\lambda_2$};
        \draw[ultra thick, rotate around = {19:(6,1.2)}] (6,1.2) arc [start angle=-134, end angle=-205,x radius = 2cm, y radius = 2.5cm] node[left, pos = 0.35] {$\gamma_2$};
        \draw[ultra thick, rotate around = {0:(8.35,3.1)}] (8.35,3.1) arc [start angle=40, end angle=-20,x radius = 1.5cm, y radius = 2cm] node[right, pos = 0.55] {$\delta_2$};
        \draw[ultra thick, rotate around = {-2:(6,-0.8)}] (5.92,1.2) arc [start angle=230, end angle=312,x radius = 2cm, y radius = 1cm] node[below, midway] {$\mu_2$};
        \draw[thick] (5.15, 3.55) arc[start angle = 0, end angle = 9,radius=2cm] node[right, pos = 0.5] {$\tau_1$};
        \draw[thick] (4.7, 3.15) arc[start angle = -110, end angle = -135, radius=1cm] node[below, pos = 0.7] {$\zeta_1$};
        \draw[ultra thick] (4.7, 3.3) arc[start angle = 270, end angle =325 ,radius=0.5cm] node[below right, pos=0.1] {$\alpha_2$};
    \end{tikzpicture}
}
    \end{minipage}
    \begin{minipage}{0.8\textwidth}
    \begin{overpic}[width=0.4\textwidth]{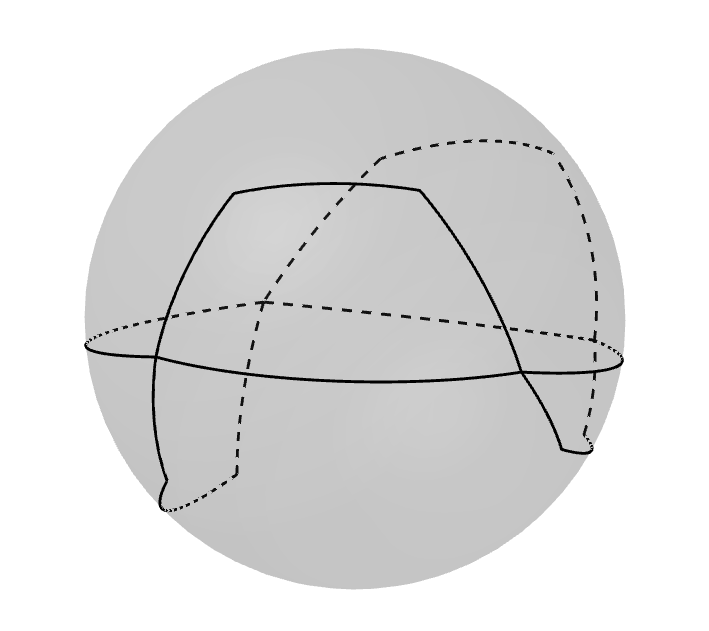}
    \put(41,30){\footnotesize\contour{white}{$\lambda_1$}}
    \put(38,64.5){\footnotesize\contour{white}{$\mu_1$}}
    \put(19,49.5){\footnotesize\contour{white}{$\gamma_1$}}    
    \put(61.5,49){\footnotesize\contour{white}{$\delta_1$}}
    \put(62,70.5){\footnotesize\contour{white}{$\mu_3$}}
    \put(71,28){\footnotesize\contour{white}{$\gamma_2$}}
    \put(16,26){\footnotesize\contour{white}{$\delta_4$}}
    \put(88,36){\footnotesize\contour{white}{$\lambda_2$}}
    \put(52,40){\footnotesize\contour{white}{$\lambda_3$}}
    \put(5.5,39){\footnotesize\contour{white}{$\lambda_4$}}
    \end{overpic}
    \end{minipage}
    \caption{Left: Partial and whole spherical linkage of a mesh in Fig.~\ref{sqmesh}, $(\lambda_i, \gamma_i, \mu_i, \delta_i,\alpha_i,\beta_i)$ and $(\lambda_i', \gamma_i', \mu_i', \delta_i',\alpha_i',\beta_i')$ are complementary to $\pi$ respectively, the gap between $\beta_1$ and $\alpha_2$ is caused by $\tau_1$ and $\zeta_1$.}
    \label{squads}
\end{figure}

From Fig.~\ref{mesh} left, it is obvious that the flexibility of a mesh is only determined by the structure of a neighborhood of the central quadrilateral. So without loss of generality, we can trim the mesh such that its corner panels are just triangles. Fig.~\ref{mesh} right is an illustration of a flexible mesh with 9 planar quadrilaterals where $(\lambda_i', \gamma_i', \mu_i', \delta_i')$ are fixed angles determined by rigid quadrilaterals and $\alpha_i'$ are flexible (dihedral) angles between the quadrilaterals while the mesh is deforming. Similarly, if the mesh contains 9 non-planar quadrilaterals, each quadrilateral can be treated as a rigid tetrahedron (consider a diagonally folded quadrilateral vs. a tetrahedron with a missing edge). Hence Fig.~\ref{mesh} right can be generalized to Fig.~\ref{sqmesh} left with a $3\times 3$ tetrahedron mesh.

The classical approach to analyze the flexibility is to consider a spherical quadrilateral at each vertex of the central face, which is defined by oriented edges: In Fig.~\ref{sqmesh}, by collecting all vectors at the origin, the directions (of vectors) intersect the unit sphere and determine the \textbf{spherical linkage} of the mesh, see Fig.~\ref{squads} right, which contains 4 spherical quads.\footnote{To distinguish the planar quadrilateral, we will call spherical quadrilateral just quad.} Every 'edge' of each quad is an arc of a great circle. Particularly, each quad of the lingkage, say $(\lambda_1, \gamma_1, \mu_1, \delta_1)$, is associated with the corner surrounded by $(\lambda_1', \gamma_1', \mu_1', \delta_1')$ in Fig.~\ref{sqmesh}. It is easy to see that $(\lambda_i, \gamma_i, \mu_i, \delta_i)$ and $(\lambda_i', \gamma_i', \mu_i', \delta_i')$ are complementary to $\pi$ respectively. Angles $\tau_i,\zeta_i$ are (interior) dihedral angles in the corresponding tetrahedrons. All these angles will be rigorously defined in Appendix \ref{dfnang}. In summary, the mesh vs. its spherical linkage, angles between edges are transformed into arcs of the quads, angles between planes are transformed into angles between arcs, and flexible hinges on adjacent quadrilaterals are transformed into flexible joints between adjacent arcs. Hence the flexibility of the mesh is equivalent to the flexibility of the spherical linkage with the deformation rules:
\begin{itemize}
    \item [$\bullet$] All arc lengths of each quad are fixed;
    \item [$\bullet$] All angles between $\lambda_i$ and $\lambda_{i+1}$ are fixed;
    \item [$\bullet$] All angles between $\delta_i$ and $\gamma_{i+1}$ are fixed.
\end{itemize}

According to Fig.~\ref{squads} left, by setting
$$
x_i=\tan\left(\frac{\alpha_i}{2}\right),\,\, y_i=\tan\left(\frac{\beta_i}{2}\right),\,\,F_i=\tan\left(\frac{\zeta_{i}+\tau_{i}}{2}\right)
$$ 
we have an invertible relation between $y_i$ and $x_{i+1}$: $y_i=\frac{F_i+x_{i+1}}{1-F_ix_{i+1}}$, or
$$
H^{(i)}(y_i,x_{i+1})=y_i(1-F_ix_{i+1})-(F_i+x_{i+1})=0
$$
where $i=1, 2, 3, 4$ and $x_5=x_1$. Since we are dealing with angles, it might occur $\zeta_{i}+\tau_{i}=\pi$ and hence $y_i=-\frac{1}{x_{i+1}}$. In such case, we regard $F_i=\infty$ and set 
$$
H^{(i)}(y_i,x_{i+1})=y_i x_{i+1} + 1=0.
$$
When $\lambda_i, \gamma_i, \mu_i, \delta_i\in(0,\pi)$, the arc lengths and angles of the quad in Fig.~\ref{squads} left satisfy a polynomial equation during its deformation. That means when the quad changes its shape while all arcs are attaching the sphere and remaining the same lengths, the valves of the interior angles $\alpha_i, \beta_i$ must obey the so-called Bricard equation \cite{bricard1897} (the proof can be also found in \cite{stachel2010}).

\begin{equation}\label{bricard}
\left\{\begin{array}{l}
    A_ix_i^2 y_i^2  + B_i x_i^2 + C_iy_i^2 + D_ix_iy_i + E_i = 0, \text{ or} \\
    g^{(i)}(x_i,y_i)=a_ix_i^2 y_i^2  + b_i x_i^2 + c_i y_i^2 + x_iy_i + e_i=0 \\
    \end{array}\right. \text{ where }
    \left\{\begin{array}{l}
    A_i=\cos(\lambda_i+ \gamma_i + \delta_i ) - \cos(\mu_i), \\
    B_i=\cos(\lambda_i+ \gamma_i - \delta_i ) - \cos(\mu_i), \\
    C_i=\cos(\lambda_i- \gamma_i + \delta_i ) - \cos(\mu_i), \\
    D_i=4\sin(\gamma_i)\sin(\delta_i), \\
    E_i=\cos(\lambda_i- \gamma_i - \delta_i ) - \cos(\mu_i), \\
    (a_i, b_i, c_i, e_i)=\left(\frac{A_i}{D_i},\frac{B_i}{D_i},\frac{C_i}{D_i},\frac{E_i}{D_i}\right).\\
    \end{array}\right.
\end{equation}

Now all flexible angles $\alpha_i, \beta_i$ (as known as $x_i,y_i$) are related by eight polynomial equations $\{H^{(i)}=g^{(i)}=0\}_{i=1}^4$. However, $y_i$ becomes redundant after the elimination
$$
\Tilde{g}^{(i)}(x_i,x_{i+1})=\text{Res}(g^{(i)},H^{(i)}; y_i)=0.
$$

\begin{remark}\label{recover}
    The coefficients $(a_i, b_i, c_i, e_i)$ could be mapped back to the angles. Since
    $$
    b_i+c_i-a_i-e_i=\cos(\lambda_i),\,\,b_i-c_i-a_i+e_i=\sin(\lambda_i) \cot(\gamma_i),\,\,c_i+e_i-a_i-b_i=\sin(\lambda_i) \cot(\delta_i),
    $$
    we can get $\lambda_i, \gamma_i, \delta_i$ immediately and $\mu_i$ can be recovered from the definition of $A_i$. In general, there is no guarantee that the angles exist for arbitrary coefficients. Thus, we will develop our theory disregarding reality and embeddability into $\rr^3$.
\end{remark}

We would like to immediately show a physical model of a non-planar flexible mesh.
\begin{example}\label{physical}
  Consider the following system of angles
  $$
    \left\{\begin{array}{l}
    (a_1, b_1, c_1, e_1,F_1)=\left(\frac{3}{2},1,n_0,-\frac{2}{3},0\right), \\
    (a_2, b_2, c_2, e_2,F_2)=\left(\frac{3}{2},n_0,1,-\frac{2}{3},0\right), \\
    (a_3, b_3, c_3, e_3,F_3)=\left(\frac{1}{2},-\frac{1}{2},\frac{(n_0+1)^2}{3},-\frac{(n_0+1)^2}{3},1\right), \\
    (a_4, b_4, c_4, e_4,F_4)=\left(\frac{3}{8n_0^2+12n_0+12},\frac{(n_0+1)^2}{4n_0^2+6n_0+6},-\frac{3}{8n_0^2+12n_0+12},-\frac{(n_0+1)^2}{4n_0^2+6n_0+6},0\right), \\
    \end{array}\right.
  $$
  where $n_0$ is the only root of $(n+1)^4+n=0$ lying in $(-1,0)$. A numerical solution could be
  $$
    \left\{\begin{array}{l}
    (\lambda_1, \gamma_1, \mu_1, \delta_1,\tau_1,\zeta_1)=(1.679854, 2.301666, 1.973198, 2.860453, 1.558808, -1.558808), \\
    (\lambda_2, \gamma_2, \mu_2, \delta_2,\tau_2,\zeta_2)=(1.679854, 2.860453, 1.973198, 2.301666, 0.694319, -0.694319), \\
    (\lambda_3, \gamma_3, \mu_3, \delta_3,\tau_3,\zeta_3)=(2.278478, 2.628901, 1.570796, 1.570796, 1.164528, 0.406268), \\
    (\lambda_4, \gamma_4, \mu_4, \delta_4,\tau_4,\zeta_4)=(2.003527, 1.570796, 1.570796, 2.335389, 0.907881, -0.907881), \\
    \end{array}\right.
  $$
  where the values of $\mu_3,\delta_3,(\tau_3+\zeta_3),\gamma_4,\mu_4$ are actually $\frac{\pi}{2}$. Visualization is in Fig.~\ref{Visual} or through the link.\footnote{\url{https://www.geogebra.org/calculator/xaurstfu}} \footnote{Please note that a numerical model is not exactly flexible. We had to set an extra free parameter $\zeta_3$ in the simulation. However, its value is almost fixed during the animation.}One can verify later that the above mesh is indeed flexible (Theorem \ref{flexequi}) and belongs to our most complicated \textbf{PQ + IQ} class (Section \ref{back-to-com}).
\end{example}

\begin{figure}[t]
    \centering
    \begin{overpic}[width=0.8\textwidth]{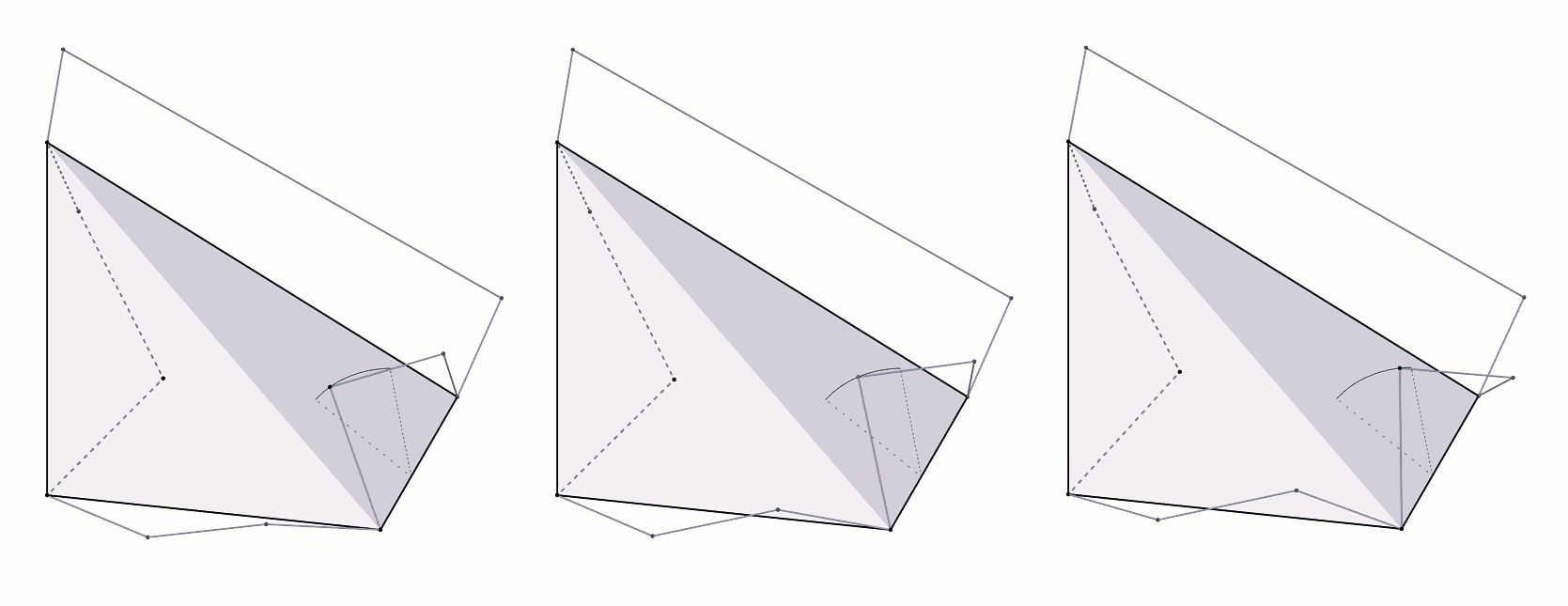}
    \put(23,2){\footnotesize\contour{white}{$V_1$}}
    \put(1,4.5){\footnotesize\contour{white}{$V_2$}}
    \put(0,29){\footnotesize\contour{white}{$V_3$}}
    \put(29,11){\footnotesize\contour{white}{$V_4$}}    
    \put(18.5,12){\footnotesize\contour{white}{$\rho$}}
    \put(19.5,15){\footnotesize\contour{white}{$X_1$}}
    \put(16,3){\footnotesize\contour{white}{$Y_1$}}
    \put(8,2){\footnotesize\contour{white}{$X_2$}}
    \put(10,12){\footnotesize\contour{white}{$Y_2$}}
    \put(6,24){\footnotesize\contour{white}{$X_3$}}
    \put(1,35){\footnotesize\contour{white}{$Y_3$}}
    \put(30.5,21){\footnotesize\contour{white}{$X_4$}}
    \put(27.5,17){\footnotesize\contour{white}{$Y_4$}}
    
    \put(56,2){\footnotesize\contour{white}{$V_1$}}
    \put(34,4.5){\footnotesize\contour{white}{$V_2$}}
    \put(32,29){\footnotesize\contour{white}{$V_3$}}
    \put(62,11){\footnotesize\contour{white}{$V_4$}}    
    \put(51,12){\footnotesize\contour{white}{$\rho$}}
    \put(53,15.7){\footnotesize\contour{white}{$X_1$}}
    \put(49,7){\footnotesize\contour{white}{$Y_1$}}
    \put(41,2){\footnotesize\contour{white}{$X_2$}}
    \put(43,12){\footnotesize\contour{white}{$Y_2$}}
    \put(38,25){\footnotesize\contour{white}{$X_3$}}
    \put(33,35){\footnotesize\contour{white}{$Y_3$}}
    \put(63.5,21){\footnotesize\contour{white}{$X_4$}}
    \put(60.5,16.5){\footnotesize\contour{white}{$Y_4$}}

    \put(89,2){\footnotesize\contour{white}{$V_1$}}
    \put(67,4.5){\footnotesize\contour{white}{$V_2$}}
    \put(64.6,29){\footnotesize\contour{white}{$V_3$}}
    \put(94,11){\footnotesize\contour{white}{$V_4$}}    
    \put(83.5,12){\footnotesize\contour{white}{$\rho$}}
    \put(87.5,16){\footnotesize\contour{white}{$X_1$}}
    \put(82,8.5){\footnotesize\contour{white}{$Y_1$}}
    \put(73,3){\footnotesize\contour{white}{$X_2$}}
    \put(76,13){\footnotesize\contour{white}{$Y_2$}}
    \put(70.3,25){\footnotesize\contour{white}{$X_3$}}
    \put(65.6,35){\footnotesize\contour{white}{$Y_3$}}
    \put(97,20.5){\footnotesize\contour{white}{$X_4$}}
    \put(96.5,15){\footnotesize\contour{white}{$Y_4$}}
\end{overpic}
    \caption{A flexible mesh with central panel $V_1V_2V_3V_4$ and attached panels $V_iY_iX_{i+1}V_{i+1}$. While undetermined, one can always choose a properly folded corner panel, containing the polyline $X_iV_iY_i$, to avoid self-intersection. This image is generated by rotating $X_1$ along arc $\rho$, in which dashed lines mean the corresponding panel is beneath the central one. }
    \label{Visual}
\end{figure}
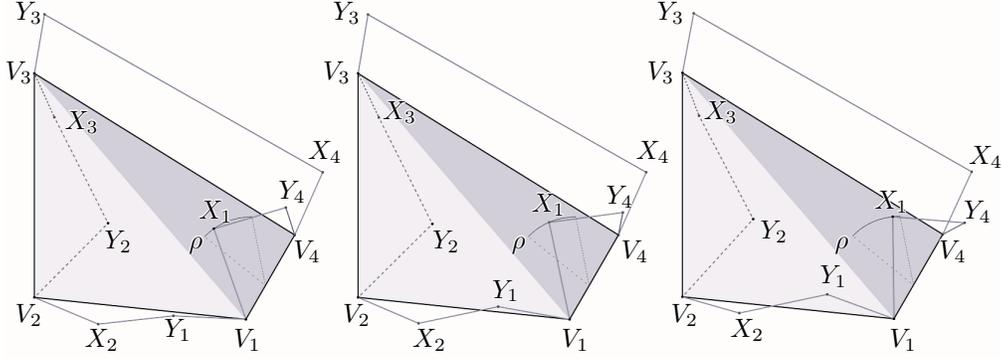

\subsubsection{Shape restrictions of the quads}

Several types of spherical quads play an important role in flexibility problems. Here we follow the definitions from \cite{izmestiev-2017}.

\begin{definition}\label{(anti)iso}
    Quad $Q_i$ with arc lengths $(\lambda_i, \delta_i, \mu_i,\gamma_i)$ is called \textbf{(anti)isogonal} or an \textbf{(anti)isogram} if it satisfies one of the corresponding conditions: 
    \begin{itemize}
        \item [(i)] $\lambda_i= \mu_i, \gamma_i=\delta_i$ (isogram);\quad (ii) $\lambda_i+\mu_i=\gamma_i+\delta_i=\pi$ (antiisogram).
    \end{itemize} 
\end{definition}

Geometrically speaking, an isogram is a quad whose opposite arc lengths are equal, and an antiisogram is a quad whose opposite arc lengths are complementary to $\pi$. When similar conditions hold for adjacent arc lengths, we have

\begin{definition}\label{(anti)del}
    Quad $Q_i$ with arc lengths $(\lambda_i, \delta_i, \mu_i,\gamma_i)$ is called an \textbf{(anti)deltoid} if it satisfies one of the corresponding conditions:
    \begin{itemize}
        \item [(iii)] $\lambda_i= \gamma_i, \mu_i=\delta_i$ (deltoid);\quad (iv) $\lambda_i+\gamma_i=\mu_i+\delta_i=\pi$ (antideltoid);
        \item [(v)] $\lambda_i= \delta_i, \gamma_i=\mu_i$ (deltoid);\quad (vi) $\lambda_i+\delta_i=\gamma_i+\mu_i=\pi$ (antideltoid).
   \end{itemize} 
\end{definition}

Obviously, a single quad $Q_i$ is in general flexible. Its shape is able to continuously change while all arc lengths $\lambda_i, \gamma_i, \mu_i, \delta_i$ are fixed. Here we do not consider the case where 0 or $\pi$ appears to be one of the lengths, which means the corresponding quadrilateral in Fig.~\ref{sqmesh} degenerates into a triangle. It is therefore reasonable to assume that:
\begin{itemize}
        \item [$\bullet$] General assumption: $\boxed{\lambda_i, \delta_i, \mu_i,\gamma_i \in (0,\pi), \forall i=1,2,3,4}$
\end{itemize}
since they are defined by the $\arccos$ function, so Bricard equation (\ref{bricard}) is applicable to all spherical linkages.

\subsection{From sphere to polynomial ideals}

We have seen that each quad uniquely determines a polynomial (\ref{bricard}) which reflects the pattern of its deformation on the sphere. We want to describe the constraints on $(\lambda_i, \delta_i, \mu_i,\gamma_i)$ in terms of $(a_i, b_i, c_i, e_i)$ since we are not interested in polynomials with arbitrary coefficients but with a geometric background.

\subsubsection{Couplings and matchings}

Since all of $\lambda_i, \gamma_i, \mu_i, \delta_i$ lie in $(0,\pi)$, we have

\begin{equation}\label{n0d}
    \sin(\lambda_i)\sin(\delta_i)\sin(\gamma_i)\sin(\mu_i)> 0.
\end{equation}

Simple calculations can show the following.

\begin{lemma}\label{i-vi}
    The conditions for (anti)isograms and (anti)deltoids are equivalent to:
    \begin{itemize}
        \item [(i)] $\Leftrightarrow b_i=c_i=0$;\quad (ii) $\Leftrightarrow a_i=e_i=0$;\quad (iii) $\Leftrightarrow c_i=e_i=0$;
        \item [(iv)] $\Leftrightarrow a_i=b_i=0$;\quad (v) $\Leftrightarrow b_i=e_i=0$;\quad (vi) $\Leftrightarrow a_i=c_i=0$.
    \end{itemize}
\end{lemma}

Clearly, a given mesh uniquely determines an ideal $M=(\Tilde{g}^{(1)},\Tilde{g}^{(2)},\Tilde{g}^{(3)},\Tilde{g}^{(4)})$ and when the mesh deforms, the 'trajectory' lies in $Z(M)\subset (\cp)^4$.

Now we are ready to convert all geometry concepts to their algebraic versions. The following notations are the most important ones hence they must be specifically defined and stick to their meanings throughout the article.

\begin{definition}\label{dfnpoly}
    $g^{(i)}\in \rr[x_i,y_i]$ is a polynomial in the form 
    $$
    g^{(i)}(x_i,y_i):=a_ix_i^2 y_i^2  + b_i x_i^2 + c_i y_i^2 + x_iy_i + e_i 
    $$
    where the coefficients satisfy\footnote{Recall inequality (\ref{n0d}), $(1-4a_ie_i-4b_ic_i)^2> 64a_ib_ic_ie_i \Leftrightarrow (\sin(\lambda_i)\sin(\delta_i)\sin(\gamma_i)\sin(\mu_i))^2>0$, $(a_i-b_i)(e_i-c_i)< \frac{1}{4}$ and $(a_i-c_i)(e_i-b_i)< \frac{1}{4}$ are both referred to $(\sin(\lambda_i))^2 > 0$.}
    $$
    (1-4a_ie_i-4b_ic_i)^2> 64a_ib_ic_ie_i,\,\, (a_i-b_i)(e_i-c_i)< \frac{1}{4},\,\, (a_i-c_i)(e_i-b_i)< \frac{1}{4}.
    $$
    $H^{(i)}\in \rr[y_i,x_{i+1}]$ is a polynomial determined by a constant $F_i\in \rp$ such that
    $$
    H^{(i)}(y_i,x_{i+1}):=\left\{
    \begin{array}{l}
        y_i(1-F_ix_{i+1})-(F_i+x_{i+1}) \text{ if } F_i\in \rr, \\
        y_ix_{i+1}+1 \text{ if } F_i=\infty. \\
    \end{array}
    \right.
    $$
    Finally, $\Tilde{g}^{(i)}\in \rr[x_i,x_{i+1}]$ and $R_i\in \rr[x_i,y_{i+1}]$ are the polynomials
    $$
    \Tilde{g}^{(i)}(x_i,x_{i+1}):=\text{Res}(g^{(i)},H^{(i)}; y_i),\,\, R_i(x_i,y_{i+1}):=\text{Res}(\Tilde{g}^{(i)},g^{(i+1)}; x_{i+1})
    $$
    where $\text{Res}(\cdot,\cdot;\ast)$ stands for the resultant respect to $'\ast'$.

    The index is always considered modulo 4, e.g. $x_5=x_1$, $H^{(8)}=H^{(4)}\in\rr[y_4,x_1]$.
\end{definition}

\begin{definition}\label{dfncom}
    For a coupling $S$, the irreducible components of $Z(S)$ in Zariski topology are called \textbf{components} of $Z(S)$ (or $S$). The components are unique up to permutations.
\end{definition}

\begin{definition}\label{coupling}
    An ideal in the form
    $$
    S=(\Tilde{g}^{(i)}, g^{(i+1)}, H^{(i)}, H^{(i+1)}) \subset \cc[x_i,x_{i+1},x_{i+2},y_i,y_{i+1}]
    $$
    is called a \textbf{coupling}. An ideal in the form 
    $$
    M=(\Tilde{g}^{(1)}(x_1,x_2), \Tilde{g}^{(2)}(x_2,x_3), \Tilde{g}^{(3)}(x_3,x_4), \Tilde{g}^{(4)}(x_4,x_1))\subset \cc[x_1,x_2,x_3,x_4]
    $$
    is called a \textbf{matching} if the zero set
    $$
    Z(M):=\{(x_1,x_2,x_3,x_4)\in (\cp)^4: f(x_1,x_2,x_3,x_4)=0, \forall f \in M\}
    $$
    is an infinite set.\footnote{We are going to apply the notation $Z(\cdot)$ on different ideals, the dimension depends on the number of variables contained by the ideal, e.g. $Z(S)\subset (\cp)^5$, $Z(g^{(i)})\subset (\cp)^2$, etc...}
\end{definition}

Through Bricard equation (\ref{bricard}), every mesh uniquely determines an ideal with the same format as a matching. However, a matching not only represents a mesh but also a flexible one. Geometrically speaking, a coupling $S=(\Tilde{g}^{(1)}, g^{(2)}, H^{(1)}, H^{(2)})$ is a pair of quads $(Q_1, F_1, Q_2, F_2)$ inserted by two constant gaps $(\zeta_{i}+\tau_{i})=2\arctan(F_i)$ at the opposite angles of $\beta_i$ for $i=1,2$ respectively, as shown in Fig.~\ref{squads} left, which is always flexible, i.e. $Z(S)$ is an infinite set. A matching is a proper pair of couplings $(Q_1,F_1,Q_2,F_2)$ and $(Q_3,F_3,Q_4,F_4)$ such that the spherical linkage $(Q_i,F_i)_{i=1}^{4}$ (Fig.~\ref{squads} right) is flexible, i.e. $Z(\Tilde{g}^{(1)}, \Tilde{g}^{(2)}, \Tilde{g}^{(3)}, \Tilde{g}^{(4)})$ is an infinite set.

\subsubsection{Non-singular properties}\label{nonsglpro}

\begin{definition}\label{dfnsingular}
    $g^{(i)}$ (or $\Tilde{g}^{(i)}$) is said to be \textbf{singular} if $a_ie_i=b_ic_i=0$, it is \textbf{non-singular} otherwise. Similarly, a coupling (or matching) is said to be \textbf{singular} if it contains a singular $g^{(i)}$ (or $\Tilde{g}^{(i)}$), it is \textbf{non-singular} otherwise.
\end{definition}

In this article, we are only interested in non-singular couplings and matching. It is not hard to check that when $g^{(i)}$ is non-singular, it is quadratic in both $x_i, y_i$ and has no factors in $\cc[x_i]$ or $\cc[y_i]$. In addition, we have the following.

\begin{lemma}\label{rational}
    A non-singular polynomial $g^{(i)}$ is reducible if and only if $a_i=e_i=0$ or $b_i=c_i=0$. Moreover, when $g^{(i)}$ is reducible, the factorization must be in the form  
    $$
    g^{(i)}=\left\{
    \begin{array}{l}
        c_i(kx_i-y_i)(k'x_i-y_i) \text{ if } a_i=e_i=0, \\
        a_i(x_iy_i-k)(x_iy_i-k') \text{ if } b_i=c_i=0, \\
    \end{array}
    \right.
    $$ 
where $kk'\neq 0$.
\end{lemma}

Combine Lemma \ref{i-vi} and \ref{rational} we immediately have

\begin{corollary}\label{redu=iso}
    Given a quad $Q_i=(\lambda_i, \gamma_i, \mu_i, \delta_i)$ which determines a polynomial $g^{(i)}$ through equation (\ref{bricard}). Then $Q_i$ is an (anti)deltoid if and only if $g^{(i)}$ is singular. Further, suppose $g^{(i)}$ is non-singular, $g^{(i)}$ is reducible if and only if $Q_i$ is (anti)isogonal if and only if $\{a_i,b_i,c_i,e_i\}$ has more than one 0.
\end{corollary}

\begin{proposition}\label{2g_i}
    A non-singular polynomial $\Tilde{g}^{(i)}$ shares the same reducibility with $g^{(i)}$. In particular, 
    \begin{itemize}
        \item [\textbullet] $\Tilde{g}^{(i)}$ is quadratic in both $x_i, x_{i+1}$ and has no factors in $\cc[x_i]$ or $\cc[x_{i+1}]$;
        \item [\textbullet] $f^{(i)}$ is a factor of $g^{(i)}$ if and only if $\Tilde{f}^{(i)}=\text{Res}(f^{(i)},H^{(i)}; y_i)$ is a factor of $\Tilde{g}^{(i)}$.
    \end{itemize}
\end{proposition}

\begin{proof}
    Since $H^{(i)}=0$ provides an rational relation between $y_i$ and $x_{i+1}$, one should realize the factorization of $g^{(i)}$ induces the factorization of $\Tilde{g}^{(i)}$ and vice versa. In fact, it is easy to check that $\exists c\neq 0$ such that $g^{(i)}=c \cdot\text{Res}(\Tilde{g}^{(i)},H^{(i)}; x_{i+1})$.
\end{proof}

\begin{corollary}\label{r_ij}
    Given $\Tilde{g}^{(i)}, g^{(i+1)}$ non-singular and at least one of them is reducible, for any irreducible factors $\Tilde{f}^{(i)},f^{(i+1)}$ of $\Tilde{g}^{(i)}, g^{(i+1)}$ respectively, $\text{Res}(\Tilde{f}^{(i)},f^{(i+1)};x_{i+1})$ is irreducible.
\end{corollary}

\begin{corollary}\label{noncons}
    Given $\Tilde{g}^{(i)}, g^{(i+1)}$ non-singular, $R_i(x_i,y_{i+1})$ has no factor in $\cc[x_i]$ or $\cc[y_{i+1}]$.
\end{corollary}

\begin{proposition}\label{eqcoup}
    A non-singular coupling $S$ has the following equivalent representations:
    $$
    S=(\Tilde{g}^{(i)}, g^{(i+1)}, H^{(i)}, H^{(i+1)})=(g^{(i)}, g^{(i+1)}, H^{(i)}, H^{(i+1)})=(\Tilde{g}^{(i)}, \Tilde{g}^{(i+1)}, H^{(i)}, H^{(i+1)}).
    $$
    The components have the following equivalent representations:
\begin{itemize}
        \item [$\bullet$] If at least one of $\Tilde{g}^{(i)}, g^{(i+1)}$ is reducible,
            \begin{equation}\label{rr}
            Z(S)=\bigcup_{j,k}W_{jk}=\bigcup_{j,k}Z(\Tilde{f}^{(i)}_j, f^{(i+1)}_k, H^{(i)}, H^{(i+1)})=\bigcup_{j,k}Z(\Tilde{f}^{(i)}_j, f^{(i+1)}_k, r_{jk}, H^{(i)}, H^{(i+1)})    
            \end{equation}
            where $\Tilde{f}^{(i)}_j, f^{(i+1)}_k$ travel all distinct irreducible factors of $\Tilde{g}^{(i)}, g^{(i+1)}$ respectively, $r_{jk}:=\text{Res}(\Tilde{f}^{(i)}_j,$ 
            $f^{(i+1)}_k;x_{i+1})$ is irreducible and $r_{jk} | R_i(x_i,y_{i+1})$;
        \item [$\bullet$] If none of $\Tilde{g}^{(i)}, g^{(i+1)}$ is reducible,
            \begin{equation}\label{irr}
                Z(S)=\bigcup_{k}W_{k}=\bigcup_{k}Z(\Tilde{g}^{(i)}, g^{(i+1)}, r_k, H^{(i)}, H^{(i+1)})=\bigcup_{k}Z(g^{(i)}, g^{(i+1)}, r_k, H^{(i)}, H^{(i+1)})    
            \end{equation}
            where $r_k$ travels all distinct irreducible factors of $R_i(x_i,y_{i+1})$ (notice that $R_i\in S$).
\end{itemize}
    Moreover, a matching $M=(\Tilde{g}^{(1)}, \Tilde{g}^{(2)}, \Tilde{g}^{(3)}, \Tilde{g}^{(4)})$ can be regarded as the sum of two couplings since $Z(S)\cong Z(\Tilde{g}^{(i)}, \Tilde{g}^{(i+1)})$.
\end{proposition}

\begin{proof}
    The proof of Proposition \ref{2g_i} shows $(g^{(i)},H^{(i)})$ and $(\Tilde{g}^{(i)},H^{(i)})$ are identical ideals. Moreover, it is easy to check that $Z(g^{(i)})\cong Z(g^{(i)}, H^{(i)})=Z(\Tilde{g}^{(i)}, H^{(i)})\cong Z(\Tilde{g}^{(i)})$ hence a matching can be regarded as the sum of two couplings. As for the components, suppose at least one of $\Tilde{g}^{(i)}, g^{(i+1)}$ is reducible. In equation (\ref{rr}) we have $W_{jk}\cong Z(\Tilde{f}^{(i)}_j, f^{(i+1)}_k, r_{jk})\subset (\cp)^3$, the projections of which in $(\cp)^2$ are irreducible curves $\Tilde{f}^{(i)}_j=0, f^{(i+1)}_k=0, r_{jk}=0$ (see Corollary \ref{r_ij}), thus $W_{jk}$ is an (irreducible) component. Finally, notice that $\forall x_1, y_2 \in \cc$,
    $$
    (r_{jk}=0) \Rightarrow (\exists x_2, \Tilde{f}^{(i)}_j=f^{(i+1)}_k=0) \Rightarrow (\exists x_2, \Tilde{g}^{(i)}=g^{(i+1)}=0) \Rightarrow (R_i=0).
    $$
    Due to irreducibility, $r_{jk}|R_i=0$ must hold. The same logic applies to $W_k$ of equation (\ref{irr}).
\end{proof}

Equivalent representations make it more convenient to analyze the differential conditions on $Z(S)$.

\begin{proposition}\label{mani}
    Given a non-singular coupling $S=(\Tilde{g}^{(1)}, g^{(2)}, H^{(1)}, H^{(2)})$. For every component $W$ of $Z(S)$, there is a finite set $W_0$ such that $\forall p \in W-W_0$ and in a neighborhood of $p$, $W$ is a curve and also a function of any of the coordinates, i.e. $W\subset (\cp)^5$ is 1-dimensional and can be locally parameterized by any variable of $\{x_1,x_2,x_3,y_1,y_2\}$. 
\end{proposition}

\subsection{Main problem (reformulated)}\label{fundamental-lemma}

It is necessary to reclaim the main problem of this article through our new definitions.

The ultimate goal is to classify all flexible meshes without (anti)deltoids, known as non-singular matchings. Since every matching is just two properly paired couplings, we first need to study the couplings. One will see in Section \ref{Component} that for any non-singular coupling, every component of its zero set must lie in 1 of 6 cases. In Section \ref{pseudo} and \ref{General}, we will demonstrate that all non-singular couplings, according to their components, fall into 7 classes. In the meanwhile, some examples are given to help the readers understand how to construct a matching by two couplings. Later in Section \ref{Res}, we will describe the classification of non-singular couplings in a different but equivalent way to serve the construction of non-singular matchings. Finally in Section \ref{back-to-com}, we will discover more restrictions on two given couplings so that they can form a matching. Therefore the final classification of flexible meshes is done.

We remind the readers that we admit the flexibility of a mesh even if the corresponding matching has an infinite zero set in $\cc$ but not necessarily in $\rr$.
\section{Classification of the components and couplings}\label{Component}

\subsection{Extension diagrams}

A brand new algebraic perspective to study the components of non-singular couplings starts here. According to Proposition \ref{eqcoup}, each component $W$ of coupling $(\Tilde{g}^{(1)}, g^{(2)}, H^{(1)}, H^{(2)})$ is obtained by solving a system with irreducible polynomials 

\begin{equation}\label{compo}
    W=\{(x_1,y_1,x_2,y_2,x_3)\in (\cp)^5 : \Tilde{f}^{(1)}=f^{(2)}=r=H^{(1)}=H^{(2)}=0\}
\end{equation}
where $\Tilde{f}^{(1)},f^{(2)}, r$ are irreducible factors of $\Tilde{g}^{(1)}, g^{(2)}, R_1$ respectively. If we fix $x_1$ as a parameter of system (\ref{compo}) and adapt the notations from Proposition \ref{mani}, locally on $W-W_0$, the rest variables are functions of $x_1$ hence can be regarded as algebraic elements of some extended fields over $K_1$ where
$$
K_i:=\cc(x_i).
$$
Similarly, if we switch the parameter to $x_2$, (locally on $W-W_0$) $x_1$ and $y_2$ can be regarded as algebraic elements on $K_2$. A formal definition goes as follows.

\begin{definition}\label{dfndia-l}
    Given a non-singular coupling $(\Tilde{g}^{(1)}, g^{(2)}, H^{(1)}, H^{(2)})$ and its component in equation (\ref{compo}), at some $p\in W$ there exist locally defined algebraic functions $y_i=Y_{i}(x_1), x_{i+1}=X_{i+1}(x_1)$ for $i=1,2$ such that
    $$
    \Tilde{f}^{(1)}(x_1,X_{2,p})=f^{(2)}(X_{2,p},Y_{2,p})=r(x_1,Y_{2,p})=H^{(1)}(Y_{1,p},X_{2,p})=H^{(2)}(Y_{2,p},X_{3,p})\equiv0.
    $$
    The \textbf{extension diagram of $W$ with respect to $x_1$ at $p$} is given by

    $$
    \scalebox{0.75}{\xymatrix{
        & K_1(X_{2,p},Y_{2,p}) & \\
        K_1(Y_{1,p})=K_1(X_{2,p}) \ar@{-}[ur]^{f^{(2)}} & & K_1(Y_{2,p})=K_1(X_{3,p})\ar@{-}[ul]\\
        & K_1\ar@{-}[ul]^{\Tilde{f}^{(1)}} \ar@{-}[ur]_{r}& \\
    }}   
    $$
\end{definition}

Similarly, we can define the extension diagrams with respect to $x_2, x_3$, etc...

Although the functions are locally defined, they are still enough to determine the component since irreducible algebraic curves can be uniquely determined by their local behaviors. So if we do not distinguish isomorphic field extensions, the extension diagram in Definition \ref{dfndia-l} is globally well-defined (on $W-W_0$) and hence uniquely characterizes $W$. In this regard, we prefer to slightly abuse the notations and simplify Definition \ref{dfndia-l} to:

\begin{definition}\label{dfndia}
    Given a non-singular coupling $(\Tilde{g}^{(1)}, g^{(2)}, H^{(1)}, H^{(2)})$ and its component in equation (\ref{compo}). The \textbf{extension diagram of $W$} (with respect to $x_1$) is given by
    \begin{equation}\label{diagram}
        \scalebox{0.75}{\xymatrix{
            & K_1(x_2,y_2) & \\
            K_1(y_1)=K_1(x_2) \ar@{-}[ur]^{f^{(2)}} & & K_1(y_2)=K_1(x_3)\ar@{-}[ul]\\
            & K_1\ar@{-}[ul]^{\Tilde{f}^{(1)}} \ar@{-}[ur]_{r}& \\
        }}   
    \end{equation}
    In particular, if we want to emphasize a trivial extension, e.g. $K_1(x_2)=K_1$, we will change the line type to
    $$
    \scalebox{0.75}{\xymatrix{K_1(x_2)& K_1\ar@{=}[l]\\}}
    $$
\end{definition}

This simplification provides us with a huge convenience when changing the parameters since locally on $W-W_0$, one can use any $x_i$ or $y_j$ to express the rest variables. And since $H^{(i)}=0$ determines a rational relation between $y_i$ and $x_{i+1}$, we formally have
$$
K_{i+1}=\cc(y_i).
$$

\begin{theorem}\label{flexequi}
    Given an ideal $M=(\Tilde{g}^{(1)}, \Tilde{g}^{(2)}, \Tilde{g}^{(3)}, \Tilde{g}^{(4)})\subset \cc[x_1,x_2,x_3,x_4]$ with all $\Tilde{g}^{(i)}$ non-singular, $M$ is a matching if and only if 
    $$
    \gcd(\text{Res}(\Tilde{g}^{(1)},\Tilde{g}^{(2)};x_2),\text{Res}(\Tilde{g}^{(3)},\Tilde{g}^{(4)};x_4))\neq 1
    $$
    if and only if for couplings $S_1=(\Tilde{g}^{(1)}, \Tilde{g}^{(2)}, H^{(1)}, H^{(2)})$ and $S_2=(\Tilde{g}^{(3)}, \Tilde{g}^{(4)}, H^{(3)}, H^{(4)})$, there exist components $W_1, W_2$ of $S_1, S_2$ respectively such that in the extension diagrams of $W_1$ and $W_2$, the minimal polynomials of $x_3$ on $K_1$ are identical.
\end{theorem}

\begin{remark}
    Each class in Theorem \ref{main} actually stands for a particular case of
    $$
    \gcd(\text{Res}(\Tilde{g}^{(1)},\Tilde{g}^{(2)};x_2),\text{Res}(\Tilde{g}^{(3)},\Tilde{g}^{(4)};x_4)),
    $$
    which means the classification of flexible meshes is essentially the classification of matchings by the greatest common divisor of the resultants.
\end{remark}

Following Definition \ref{dfndia}, we further define
$$
[K_i(x_j):K_i]:=\text{extension degree of } K_i(x_j)/K_i,\quad [K_i(y_j):K_i]:=\text{extension degree of } K_i(y_j)/K_i.
$$
Clearly, $[K_i(x_{j+1}):K_i]=[K_i(y_j):K_i]$ by previous discussion. For each component,
$$
[K_1(x_2):K_1]=[K_1(y_1):K_1] \in \{1,2\},\,\, [K_1(x_2,y_2):K_1] \in \{1,2,4\}
$$
due to the quadratic nature of $g^{(i)}$ or $\Tilde{g}^{(i)}$, which also implies
$$
[K_i(x_{i+1}):K_i]=2 \Leftrightarrow [K_{i+1}(x_i):K_{i+1}]=2 \Leftrightarrow \Tilde{g}^{(i)} \text{ is irreducible.}
$$
And for $K_1(y_2)$, as a subfield of $K_1(x_2,y_2)$, we must have
$$
[K_1(y_2):K_1] \in \{1,2,4\}.
$$
The above argument together with Definition \ref{dfndia} separates the components of non-singular couplings into 6 cases, all of which are equivalent to those we mentioned above table (\ref{c-table}).
\begin{definition}\label{casedefn}
Given $W$ a component of a non-singular coupling. According to its extension diagram, we say $W$ is of 
    \begin{itemize}
        \item [\textbf{Case 1}:] \textbf{purely-rational} if $[K_1(x_2):K_1]=[K_1(y_2):K_1]=1$;
        \item [\textbf{Case 2}:] \textbf{half-quadratic} if $[K_1(x_2):K_1]=1$ and $[K_1(y_2):K_1]=2$, or if $[K_3(x_2):K_3]=1$ and $[K_3(x_1):K_3]=2$;
        \item [\textbf{Case 3}:] \textbf{involutive-rational} if $[K_1(y_2):K_1]=1$ and $[K_1(x_2):K_1]=2$;
        \item [\textbf{Case 4}:] \textbf{purely-quadratic} if $[K_1(x_2,y_2):K_1]=[K_1(y_2):K_1]=[K_2(y_2):K_2]=[K_1(x_2):K_1]=2$;
        \item [\textbf{Case 5}:] \textbf{involutive-quadratic} if $[K_1(y_2):K_1]=[K_1(x_2):K_1]=2$ and $[K_1(x_2,y_2):K_1]=4$;
        \item [\textbf{Case 6}:] \textbf{quartic} if $[K_1(y_2):K_1]=4$.
    \end{itemize} 
\end{definition}

This definition shows that, depending on the reducibility of $g^{(i)}$ and $R_i$, the components from Proposition \ref{eqcoup} are not just represented in different ways but are essentially different in their own algebraic natures. Table (\ref{table}) tells that a given non-singular coupling cannot have components from different cases except for rational-quadratic class. This means only the components of \textbf{Case 3} and \textbf{Case 4} can coexist in one coupling during the decomposition. 

\begin{definition} \label{dfn-ps}
    Non-singular couplings $S$ are characterized as table (\ref{table}). More precisely, we say $S$ is
    \begin{itemize}
        \item [$\bullet$] \textbf{purely-rational} if $S$ only admits components of \textbf{Case 1};
        \item [$\bullet$] \textbf{half-quadratic} if $S$ only admits components of \textbf{Case 2};   
        \item [$\bullet$] \textbf{equimodular} if $S$ only admits components of \textbf{Case 3} or \textbf{Case 4}, moreover, $S$ is
        \begin{itemize}
            \item [$\bullet$] \textbf{involutive-rational} if $S$ only admits components of \textbf{Case 3};
            \item [$\bullet$] \textbf{rational-quadratic} if $S$ only admits components of \textbf{Case 3} and \textbf{Case 4};
            \item [$\bullet$] \textbf{purely-quadratic} if $S$ only admits components of \textbf{Case 4}; 
        \end{itemize}
        \item [$\bullet$] \textbf{involutive-quadratic} if $S$ only admits components of \textbf{Case 5};
        \item [$\bullet$] \textbf{quartic} if $S$ only admits components of \textbf{Case 6}.
    \end{itemize} 
\end{definition}
\underline{The definition only leads to a classification once we prove its completeness}, which will be a direct consequence of Theorem \ref{main-ps} and \ref{main-g}.

\subsection{First application}\label{PRproof}

Now we are ready to deal with the classification problem. As a starter, we tackle the simplest class of flexible meshes whose spherical linkages only contain (anti)isograms. It requires no advanced methods but is a good chance to show how the new algebraic definitions work in practice. This kind of mesh has a very unique algebraic structure and is called \textbf{PQ} (purely-rational) class. An equivalent definition of such a class can be found later in Definition \ref{11class}.

\begin{theorem}\label{PR}
    Given a mesh with notations in Fig.~\ref{sqmesh} and Fig.~\ref{squads} such that $\forall i\in \{1,2,3,4\}$, $(\lambda_i, \delta_i, \mu_i,\gamma_i)$ form an (anti)isogram but not (anti)deltoid. Set
    $$
    k_i=\frac{\sin(\gamma_i-\mu_i)}{\sin(\gamma_i)\pm \sin(\mu_i)}\in \rr-\{0\},\,\,F_i=\tan\left(\frac{\zeta_{i}+\tau_{i}}{2}\right) \in \rr\cup\{\infty\},
    $$
    $$
    N_i=\left\{\begin{array}{l}
            \left(\begin{array}{cc}
            -F_i & k_i \\
            1 & k_iF_i \\
            \end{array}\right) \text{ if } \lambda_i=\mu_i, \delta_i=\gamma_i, F_i\in \rr,\\
            \left(\begin{array}{cc}
            -1 & 0 \\
            0 & k_i \\
            \end{array}\right) \text{ if } \lambda_i=\mu_i, \delta_i=\gamma_i, F_i=\infty,\\            
            \left(\begin{array}{cc}
            k_i & -F_i \\
            k_iF_i & 1 \\
            \end{array}\right) \text{ if } \lambda_i+\mu_i=\delta_i+\gamma_i=\pi, F_i\in \rr,\\
            \left(\begin{array}{cc}
            0 & -1 \\
            k_i & 0 \\
            \end{array}\right) \text{ if } \lambda_i+\mu_i=\delta_i+\gamma_i=\pi, F_i=\infty.\\
    \end{array}\right.
    $$
    Consider all combinations of signs $'\pm'$ appear in $k_i$, the mesh is flexible if and only if there exists a combination such that $N_4N_3N_2N_1$ is a scalar matrix.
\end{theorem}
We refer to \cite{nawratil-2022} for visualizations of such flexible meshes.
\begin{proof}
    We only prove the case when $\lambda_i=\mu_i, \delta_i=\gamma_i, F_i\in\rr$ for all $1\leq i \leq 4$ since the other cases are quite similar. First, it is easy to verify that $k_i$ is well-defined and nonzero since none of the conditions in Definition \ref{(anti)del} is satisfied. On the other hand, a given mesh uniquely determines a spherical linkage $(Q_i,F_i)_{i=1}^4$ (e.g. Fig.~\ref{squads} right) and each $Q_i$ determines a $g^{(i)}$, according to Lemma \ref{i-vi} and \ref{rational} we have
    $$
    g^{(i)}=a_i(x_iy_i-k_i)(x_iy_i-k_i')=a_ix_i^2y_i^2+x_iy_i+e_i
    $$
    where $\{k_i,k_i'\}=\left\{\frac{-1\pm\sqrt{1-4a_ie_i}}{2a_i}\right\}=\left\{\frac{\sin(\gamma_i-\mu_i)}{\sin(\gamma_i)\pm \sin(\mu_i)}\right\}$. Consider the coupling $S_1=(\Tilde{g}^{(1)}, g^{(2)}, H^{(1)}, H^{(2)})$, by Proposition \ref{eqcoup} each component of $S_1$ is in the form
    $$
    W_1=Z(\Tilde{f}^{(1)}, x_2y_2-k_2, H^{(1)}, H^{(2)})               
    $$
    where $\Tilde{f}^{(1)}$ is an irreducible factor of $\Tilde{g}^{(1)}$, which can be replaced by the factor of $g^{(1)}$ due to Proposition \ref{2g_i}, so
    \begin{equation}\label{1dm}
        W_1=Z(x_1y_1-k_1, x_2y_2-k_2, y_1-F_1y_1x_2-F_1-x_2, y_2-F_2y_2x_3-F_2-x_3),
    \end{equation}
    hence in the extension diagram, for $i=1,2$ we have
    $$
    x_{i+1}=\frac{y_i-F_i}{F_iy_i+1},\,\, y_i=\frac{0x_i+k_i}{x_i+0}.
    $$
    If we formally express the above relations in a matrix form
    $$
    N=\left (\begin{array}{cc}
    a & b \\
    c & d \\
    \end{array}\right):\,\, \cp \rightarrow \cp;\,\, x \mapsto N(x):=\frac{ax+b}{cx+d},\,\,ad-bc\neq 0,
    $$
    this is the so-called M\"{o}bius transformation on $\cp$ which is compatible with matrix multiplications. e.g.
    $$
    x_{i+1}=\left (\begin{array}{cc}
    1 & -F_i \\
    F_i & 1 \\
    \end{array}\right)(y_i)=\left (\begin{array}{cc}
    1 & -F_i \\
    F_i & 1 \\
    \end{array}\right)\left (\begin{array}{cc}
    0 & k_i \\
    1 & 0 \\
    \end{array}\right)(x_i)=
    \left(\begin{array}{cc}
    -F_i & k_i \\
    1 & k_iF_i \\
    \end{array}\right)(x_i).
    $$
    Let $N_i=\left(\begin{array}{cc}
    -F_i & k_i \\
    1 & k_iF_i \\
    \end{array}\right)$ we have $x_3=N_2N_1(x_1)$. Likewise, in a component $W_2$ of coupling $S_2=(\Tilde{g}^{(3)}, g^{(4)}, H^{(3)}, H^{(4)})$ we have $x_1=N_4N_3(x_3)$. Since the mesh should be flexible, the corresponding matching $M=(\Tilde{g}^{(1)},\Tilde{g}^{(2)},\Tilde{g}^{(3)},\Tilde{g}^{(4)})$ must have an infinite zero set $Z(M)\cong Z(S_1)\times_{\{x_1,x_3\}}Z(S_2)$ (Proposition \ref{eqcoup}) thus there exist components $W_1, W_2$ of $S_1, S_2$ respectively such that $W_1\times_{\{x_1,x_3\}}W_2$ is an infinite set. From equation (\ref{1dm}) we know $W_i$ is a 1-dimensional curve that can be parameterized by $x_1$. Therefore $\{x_3=N_2N_1(x_1),x_1=N_4N_3(x_3)\}$ must have infinitely many solutions so that $W_1\times_{\{x_1,x_3\}}W_2$ can be infinite. In other words,
    $$
    x_1=N_4N_3N_2N_1(x_1)=\left (\begin{array}{cc}
    n & 0 \\
    0 & n \\
    \end{array}\right)(x_1)=x_1.
    $$
\end{proof}

\begin{figure}
\hfill
\begin{overpic}[width=0.8\textwidth]{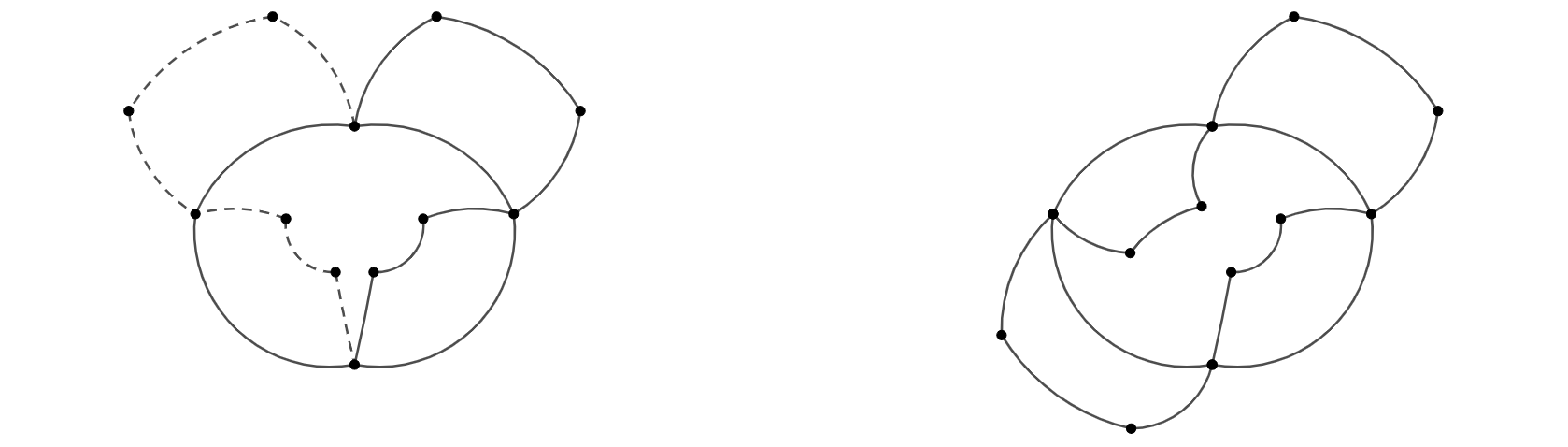}
    \put(27,9){\footnotesize\contour{white}{$Q_1$}}
    \put(29,22){\footnotesize\contour{white}{$Q_2$}}
    \put(16,9.5){\footnotesize\contour{white}{$Q_1'$}}
    \put(14,22){\footnotesize\contour{white}{$Q_2'$}}
    \put(30.5,6){\footnotesize\contour{white}{$\lambda_1$}}
    \put(26.5,17.5){\footnotesize\contour{white}{$\lambda_2$}}
    \put(12.5,5.5){\footnotesize\contour{white}{$\lambda_1'$}}
    \put(16.5,17.5){\footnotesize\contour{white}{$\lambda_2'$}}
    \put(82,9.5){\footnotesize\contour{white}{$Q_1$}}
    \put(84,22){\footnotesize\contour{white}{$Q_2$}}
    \put(85.5,6){\footnotesize\contour{white}{$\lambda_1$}}
    \put(81.5,17.5){\footnotesize\contour{white}{$\lambda_2$}}
    \put(69,9){\footnotesize\contour{white}{$\lambda_4=\lambda_1'$}}
    \put(66,21.5){\footnotesize\contour{white}{$\lambda_3=\lambda_2'$}}
    \put(67.5,5.5){\footnotesize\contour{white}{$Q_4$}}
    \put(71,16){\footnotesize\contour{white}{$Q_3$}}
\end{overpic}
  \hfill{}
    \caption{This is a (conceptual) spherical image, every curve lies on a great circle. Left: A coupling $(Q_1,F_1,Q_2,0)$ and its spherical reflection $(Q_2',-F_1,Q_1',0)$. Right: According to Fig.~\ref{squads}, the positions of adjacent quads should be upside down, so in order to get a spherical linkage, we need to flip $Q_2'$ respect to $\lambda_2'$ and $Q_1'$ respect to $\lambda_1'$.} 
    \label{Ex1}
\end{figure}
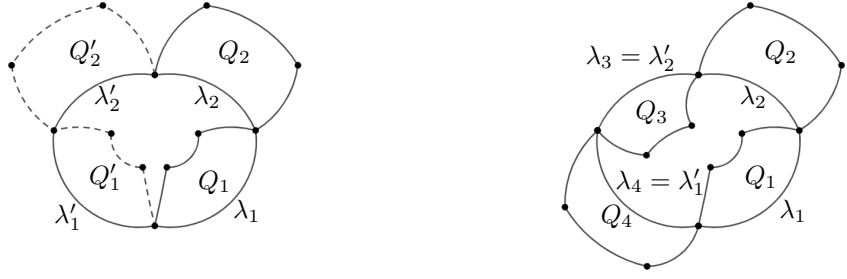

Here is another example that shows the existence of the flexible mesh whose spherical linkage contains two arbitrarily coupled quads in Fig.~\ref{squads} left, i.e. the matching that contains a coupling of an arbitrary class.

\begin{example}[Reflection]\label{skeleton}
    The main idea is extending a coupling of an arbitrary class to a matching. Notice that in table (\ref{table}), the class of the coupling is defined by its components, and the cases of components, given in Definition \ref{casedefn}, are independent of $H^{(2)}$, which means the class of a coupling $S=(\Tilde{g}^{(1)}, g^{(2)}, H^{(1)}, H^{(2)})=(g^{(1)}, g^{(2)}, H^{(1)}, H^{(2)})$ only depends on its first three polynomials. Now suppose we have a coupling in which $g^{(1)}, H^{(1)}, g^{(2)}$ are arbitrary. Set the coefficients of $g^{(3)}, g^{(4)}$ and $F_2, F_3, F_4$ as follows
    $$
    (a_3,b_3,c_3,e_3)=(a_2,c_2,b_2,e_2),\,\,(a_4,b_4,c_4,e_4)=(a_1,c_1,b_1,e_1),\,\,F_2=F_4=0,\,\,F_3=-F_1.
    $$
    In this regard $Z(\Tilde{g}^{(3)},\Tilde{g}^{(4)})$ is just a 'copy' of $Z(\Tilde{g}^{(1)},\Tilde{g}^{(2)})$ thus, by Proposition \ref{eqcoup}, the fiber product $Z(S_1)\times_{\{x_1,x_3\}} Z(S_2)\cong Z(\Tilde{g}^{(1)},\Tilde{g}^{(2)},\Tilde{g}^{(3)},\Tilde{g}^{(4)})$ is always an infinite set. Recall the paragraph above Section \ref{nonsglpro}, the geometric illustration of this example is simply the (spherical) reflection of a given coupling $(Q_1,F_1,Q_2,F_2=0)$, see Fig.~\ref{Ex1}. This trick actually works in all situations (including (anti)deltoids).
\end{example}

\section{Pseudo-planar type without (anti)deltoids}\label{pseudo}

\subsection{Section outline and classification theorem}

We start with the simplest meshes in which all $F_i\in\{0,\infty\}$. In such a type we have 
$$
\left\{\begin{array}{l}
   \Tilde{g}^{(i)}(x_i,x_{i+1})=g^{(i)}(x_i,x_{i+1})=a_ix_i^2 x_{i+1}^2  + b_i x_i^2 + c_ix_{i+1}^2 + x_ix_{i+1} + e_i \text{ if $F_i=0$,} \\
   \Tilde{g}^{(i)}(x_i,x_{i+1})=-b_ix_i^2 x_{i+1}^2  - a_i x_i^2 - e_ix_{i+1}^2 + x_ix_{i+1} - c_i \text{ if $F_i=\infty$.} \\
\end{array}\right.
$$
All equations are exactly in the same format as in the planar type, the only difference might be switching the coefficients and that is why we call it \textbf{pseudo-planar type}.\footnote{In Fig.~\ref{squads} left, planar type requires $\zeta_1=\tau_1=0$ but pseudo-planar type only requires $\sin(\zeta_1+\tau_1)=0$.} It is easy to check that the coefficients, when $F_i=\infty$, also fit the inequalities of Definition \ref{dfnpoly} and the non-sigularity of $\Tilde{g}^{(i)}$ remains. In this regard, we only need to develop the theory for $F_i=0$ and directly apply it to $F_i=\infty$ after switching the coefficients. Hence in this section, we regard $y_i=x_{i+1}$ and simply write a coupling as $S=(g^{(1)}(x_1,x_2), g^{(2)}(x_2,x_3))$. 

\begin{theorem}\label{main-ps}
    Given a non-singular coupling $S=(g^{(1)}, g^{(2)})$. According to Definition \ref{dfn-ps}, each class of $S$ is equivalent to one of the corresponding conditions. i.e. $S$ is
    \begin{itemize}
        \item [$\bullet$] \textbf{purely-rational} if both $g^{(1)}, g^{(2)}$ are reducible;
        \item [$\bullet$] \textbf{half-quadratic} if only one of $g^{(1)}, g^{(2)}$ is reducible;   
        \end{itemize} 
    Further, when both $g^{(1)}, g^{(2)}$ are irreducible, $S$ is
    \begin{itemize}
        \item [$\bullet$] \textbf{involutive-rational} if
        \begin{equation}\label{eq-c3}
            \left\{\frac{a_1}{a_2}=\frac{b_1}{c_2}=\frac{b_2}{c_1}=\frac{e_2}{e_1}\right\} \text{ and } \left\{\frac{a_1}{b_2}=\frac{b_1}{e_2}=\frac{a_2}{c_1}=\frac{c_2}{e_1}\right\};\footnote{We will frequently use such kind of proportional chains throughout the paper to represent the linear dependence between two sets of coefficients, so 0 is allowed among the denominators.}
        \end{equation}
        \item [$\bullet$] \textbf{rational-quadratic} if only one of systems (\ref{eq-c3}) holds;
        \item [$\bullet$] \textbf{purely-quadratic} if one of the following systems holds
        \begin{equation}\label{eq-c4}
            \left\{\begin{array}{l}
                 \left\{\frac{a_1 c_1}{a_2 b_2}=\frac{1- 4 a_1 e_1- 4 b_1 c_1}{1 - 4 a_2 e_2- 4 b_2 c_2}=\frac{b_1 e_1}{c_2 e_2}\neq 1\right\} \text{ or}  \\
                 \left\{\frac{4a_1 c_1}{1- 4 a_2 e_2- 4 b_2 c_2}=\frac{1- 4 a_1 e_1- 4 b_1 c_1}{4c_2 e_2}, b_1 e_1 = a_2 b_2 =0\right\} \text{ or} \\
                 \left\{\frac{4b_1e_1}{1- 4 a_2 e_2- 4 b_2 c_2}=\frac{1- 4 a_1 e_1- 4 b_1 c_1}{4a_2b_2}, a_1 c_1 = c_2 e_2 =0\right\};
            \end{array}
            \right.
        \end{equation}
        \item [$\bullet$] \textbf{involutive-quadratic} if
        \begin{equation}\label{eq-c5}
            \left\{\frac{a_1}{b_1}=\frac{c_1}{e_1}=\frac{a_2}{c_2}=\frac{b_2}{e_2}, \frac{a_1}{a_2}\neq \frac{b_2}{c_1}\right\};
        \end{equation}
        \item [$\bullet$] \textbf{quartic} if otherwise.
    \end{itemize} 
\end{theorem}

The proof has been split into cases which will be detailed studied.

\subsection{Case-by-case study}

\subsubsection{Purely-rational}

\begin{proof}(Theorem \ref{main-ps})
    Since all $g^{(i)}$ are reducible, by Proposition \ref{eqcoup}, each component has the form
    $$
    W=Z(f^{(1)}, f^{(2)}, r_{12}, H^{(1)}, H^{(2)})               
    $$
    where $f^{(1)}, f^{(2)}$ are irreducible factors of $g^{(1)},g^{(2)}$ respectively and $r_{12}=\text{Res}(f^{(1)},f^{(2)};x_2)$. Further, by Lemma \ref{rational} all extensions are trivial in the extension diagram of $W$, which implies $[K_1(x_3):K_1]=[K_1(x_2):K_1]=1$. Finally, according to Definition \ref{dfn-ps} we know $S$ is purely-rational.
\end{proof}

\subsubsection{Half-quadratic}

\begin{proof}(Theorem \ref{main-ps})
    Consider the extension diagram of the component of $S$. If $g^{(1)}$ is reducible but $g^{(2)}$ is not, we have the left-hand side of the following diagrams
    $$
    \scalebox{0.75}{\xymatrix{
                & K_1(x_2,x_3) & & & & K_3(x_1,x_2) &\\
                K_1(x_2) \ar@{-}[ur]^{g^{(2)}} & & K_1(x_3)\ar@{=}[ul] & & K_3(x_2) \ar@{-}[ur]^{g^{(1)}} & & K_3(x_1)\ar@{=}[ul]\\
                & K_1\ar@{=}[ul]^{f^{(1)}} \ar@{-}[ur] & & & & K_3\ar@{=}[ul]^{f^{(2)}} \ar@{-}[ur] & \\
            }}   
    $$
    where $f^{(1)}$ is an irreducible factor of $g^{(1)}$; or if $g^{(2)}$ is reducible but $g^{(1)}$ is not, change the parameter to $x_3$ we have the right-hand side of the above diagrams, where $f^{(2)}$ is an irreducible factor of $g^{(2)}$. Each of the diagrams implies a component of \textbf{Case 2} which are the only ones possessed by half-quadratic couplings according to Definition \ref{dfn-ps}.
\end{proof}

\begin{example}\label{HQ}
    Let us construct a matching in which both couplings $(g^{(1)}, g^{(2)}), (g^{(3)}, g^{(4)})$ respectively admit components $W_1, W_2$ of \textbf{Case 2}. First, according to Lemma \ref{rational} we may set $a_1=e_1=0$ and $b_3=c_3=0$, then set $g^{(2)}$ irreducible. To build a matching we need $W_1\times_{\{x_1, x_3\}} W_2$ to be an infinite set. On the other hand, in the extension diagrams of $W_1, W_2$ we have $x_2 = k x_1$ and $x_4=\frac{k'}{x_3}$. Plugging them into $g^{(2)}$ and $g^{(4)}$ ($g^{(4)}$ is undetermined yet) respectively should give two multiples of the same minimal polynomial of $x_3$ on $K_1$ according to Theorem \ref{flexequi}:
    $$
    \left\{\begin{array}{l}
        (k^2 a_2 x_1^2 + c_2 ) x_3^2 + k x_1 x_3 + (k^2 b_2 x_1^2+ e_2),  \\
        (c_4 x_1^2 + e_4 ) x_3^2 + k' x_1 x_3 + (k'^2 a_4 x_1^2 + k'^2 b_4), \\
    \end{array}\right.
    $$
    which should be equal up to a constant, i.e. the coefficients of $g^{(4)}$ are determined by
    $$
    \frac{k^2 a_2}{c_4}=\frac{c_2}{e_4}=\frac{k}{k'}=\frac{k^2 b_2}{k'^2 a_4}=\frac{e_2}{k'^2 b_4}.
    $$
\end{example}

\subsubsection{Involutive-rational}

\begin{lemma}\label{n-is}
    Given a non-singular coupling $S=(g^{(1)}, g^{(2)})$ where all $g^{(i)}$ are irreducible, set $R(x_1,x_3)=\text{Res}(g^{(1)},g^{(2)};x_2)$ we have
        $$
        \begin{array}{ll}
           (1)  & (x_3-kx_1)^2|R(x_1,x_3)\Leftrightarrow(x_3-kx_1)|R(x_1,x_3)\Leftrightarrow\frac{a_1}{a_2}=\frac{b_1}{c_2}=\frac{b_2}{c_1}=\frac{e_2}{e_1}=k; \\
           (2)  & (x_1x_3-k')^2|R(x_1,x_3)\Leftrightarrow(x_1x_3-k')|R(x_1,x_3)\Leftrightarrow\frac{a_1}{b_2}=\frac{b_1}{e_2}=\frac{a_2}{c_1}=\frac{c_2}{e_1}=\frac{1}{k'}.\\
        \end{array}
        $$
\end{lemma}

\begin{corollary}\label{n-is-c}
    Given a non-singular coupling $S=(g^{(1)}, g^{(2)})$ where all $g^{(i)}$ are irreducible, $S$ admits a component of \textbf{Case 3} if and only if one of systems (\ref{eq-c3}) holds if and only if
    \begin{equation}\label{===1}
        \frac{a_1 c_1}{a_2 b_2}=\frac{1- 4 a_1 e_1- 4 b_1 c_1}{1 - 4 a_2 e_2- 4 b_2 c_2}=\frac{b_1 e_1}{c_2 e_2}=1.
    \end{equation}
\end{corollary}

\begin{proof}
    We first prove the equivalence of the last two statements. It is easy to check that either of systems (\ref{eq-c3}) implies equation (\ref{===1}). Conversely, since $S$ is non-singular, we may assume $a_2e_2\neq 0$ (or $b_2c_2\neq 0$ otherwise). By substitution $\left\{b_2=\frac{a_1 c_1}{a_2}, c_2=\frac{b_1 e_1}{e_2}\right\}$, equation (\ref{===1}) implies (one of) systems (\ref{eq-c3}), i.e.
    $$
    \{a_1 c_1=a_2 b_2, b_1 e_1=c_2 e_2, (a_2e_2-a_1e_1)(a_2e_2-b_1c_1)=0\}.
    $$    
    Now suppose $S$ admits a component $W$ of \textbf{Case 3}. In the extension diagram of $W$, $g^{(1)}(x_1,t)\in K_1[t]$ determines the minimal polynomial of $x_2$ on $K_1$ since $g^{(1)}$ is irreducible. By Definition \ref{casedefn}, $x_3$ lies in $K_1$ hence $x_3=f/g$ where $f,g\in \cc[x_1]$ and $\gcd(f,g)=1$. $g^{(2)}(t,x_3)$, as well, determines the minimal polynomial of $x_2$ on $K_1$ so the ratios of the corresponding coefficients in $g^{(1)}(x_1,t)$ and $g^{(2)}(t,x_3)$ should coincide, i.e.
    \begin{equation}\label{eq-coeff}
        \frac{a_1 x_1^2 + c_1}{a_2 x_3^2 + b_2}=\frac{x_1}{x_3}=\frac{b_1 x_1^2 + e_1}{c_2 x_3^2 + e_2} \Leftrightarrow \frac{a_1 x_1^2 + c_1}{a_2 f^2 + b_2 g^2}=\frac{x_1}{f g}=\frac{b_1 x_1^2 + e_1}{c_2 f^2 + e_2 g^2}.
    \end{equation}
    According to Corollary \ref{redu=iso}, for $i=1, 2$, $\{a_i,...,e_i\}$ contains at most one zero, hence $\gcd(a_1 x_1^2 + c_1,x_1,b_1 x_1^2 + e_1)=1=\gcd(a_2 f^2 + b_2 g^2,f g,c_2 f^2 + e_2 g^2)$. This implies the second ratio chain in equation (\ref{eq-coeff}) is a constant in $\cc$ (i.e. $f g/x_1 \in \cc$) so it is either $x_3=k x_1$ or $x_3=k'/x_1$ for some nonzero constant $k$ or $k'$. Plugging them into equation (\ref{eq-coeff}) we obtain 
    $$
    \frac{a_1}{a_2}=\frac{b_1}{c_2}=\frac{b_2}{c_1}=\frac{e_2}{e_1}=k \text{ or } \frac{a_1}{b_2}=\frac{b_1}{e_2}=\frac{a_2}{c_1}=\frac{c_2}{e_1}=\frac{1}{k'}.
    $$
    The converse is directly from Lemma \ref{n-is}.
\end{proof}

\begin{proof}(Theorem \ref{main-ps})
    Given $g^{(i)}$ non-singular and 
    $$
    \left\{\frac{a_1}{a_2}=\frac{b_1}{c_2}=\frac{b_2}{c_1}=\frac{e_2}{e_1}=k, \frac{a_1}{b_2}=\frac{b_1}{e_2}=\frac{a_2}{c_1}=\frac{c_2}{e_1}=\frac{1}{k'}\right\}
    $$
    we know $kk'\neq 0$ hence $\gcd(x_3-kx_1,x_1x_3-k')=1$. Apply Lemma \ref{n-is} we have $(x_3-kx_1)^2(x_1x_3-k')^2|R(x_1,x_3)=\text{Res}(g^{(1)},g^{(2)};x_2)$. Notice that $R(x_1,x_3)$ is originally a 4 by 4 determinant with the degree of $x_1$ at most 4, so $R(x_1,x_3)$ only has two different factors $(x_3-kx_1)$ and $(x_1x_3-k')$. Consequently, the extension diagram of each component always implies $x_3\in K_1$ (i.e. \textbf{Case 3}) hence $S$ is involutive-rational.
\end{proof}

\begin{example}\label{IR}
    Let us consider an ideal $M=(g^{(1)},g^{(2)},g^{(3)},g^{(4)})$ with all $g^{(i)}$ irreducible and the coefficients satisfy
    $$
    \left\{\frac{a_1}{b_2}=\frac{b_1}{e_2}=\frac{a_2}{c_1}=\frac{c_2}{e_1}=\frac{1}{k'}=\frac{a_4}{c_3}=\frac{c_4}{e_3}=\frac{a_3}{b_4}=\frac{b_3}{e_4}\right\}
    $$
    where $k'\neq 0$ is a given parameter. Clearly, according to Lemma \ref{n-is}, both couplings $(g^{(1)},g^{(2)})$ and $(g^{(3)},g^{(4)})$ admit components of \textbf{Case 3} with the extension diagrams in which $x_3x_1-k'=0$. Thus $M$ is a matching according to Theorem \ref{flexequi}.
\end{example}

\subsubsection{Rational-quadratic}

The next lemma characterizes equimodular couplings in two aspects. \underline{Proposition \ref{mani} is crucial for its proof!}
\begin{lemma}\label{samedsc}
    Given a coupling $S=(g^{(1)}, g^{(2)})$ where all $g^{(i)}$ are irreducible. $S$ is equimodular if and only if in every extension diagram of its components, one of the following equivalent conditions holds:
    \begin{itemize}
        \item [(1)] $x_3 \in K_1(x_2)$;\quad (2) $x_3 \in K_2(x_1)$;\quad (3) $x_1 \in K_2(x_3)$;\quad (4) $x_1 \in K_3(x_2)$.
    \end{itemize} 
    On the other hand, $S$ is equimodular if and only if one of the following systems holds.
    \begin{equation}\label{eq-c3or4}
            \left\{\begin{array}{l}
                 \left\{\frac{a_1 c_1}{a_2 b_2}=\frac{1- 4 a_1 e_1- 4 b_1 c_1}{1 - 4 a_2 e_2- 4 b_2 c_2}=\frac{b_1 e_1}{c_2 e_2}\right\} \text{ or}  \\
                 \left\{\frac{4a_1 c_1}{1- 4 a_2 e_2- 4 b_2 c_2}=\frac{1- 4 a_1 e_1- 4 b_1 c_1}{4c_2 e_2}, b_1 e_1 = a_2 b_2 =0\right\} \text{ or} \\
                 \left\{\frac{4b_1e_1}{1- 4 a_2 e_2- 4 b_2 c_2}=\frac{1- 4 a_1 e_1- 4 b_1 c_1}{4a_2b_2}, a_1 c_1 = c_2 e_2 =0\right\}.
            \end{array}
            \right.
        \end{equation}
\end{lemma}

\begin{proof}
    In the extension diagram, one should realize that $K_i(x_{i+1}), K_{i+1}(x_i)$ are just two representations of the same quotient field of the domain $\cc[x_i,x_{i+1}]/(g^{(i)})$, given $g^{(i)}$ is prime. This immediately shows $(1)\Leftrightarrow(2)$ and $(3)\Leftrightarrow(4)$. Now suppose $x_3 \in K_2(x_1)$ so there exist $f_1, g_1 \in K_2$ such that $x_3=f_1 + g_1 x_1$. Due to the irreducibility of $g^{(2)}$, $x_3 \notin K_2$ hence $g_1\neq0$, which implies $x_1=(x_3-f_1)/g_1\in K_2(x_3)$. This trick can show $(2)\Leftrightarrow(3)$. Next, if we treat $g^{(1)}(t,x_2)$ and $g^{(2)}(x_2,t)$ as polynomials in $K_2[t]$, condition (2) means their discriminants are equal up to a square in $K_2$. Let $\Delta_1$ and $\Delta_2$ be the corresponding discriminants of $g^{(1)}$ and $g^{(2)}$ respectively. In particular, $\Delta_1$ is not a square since $g^{(1)}$ is irreducible. Hence
    $$
    \Delta_1=-4 a_1 c_1 x_2^4 + (1 - 4 a_1 e_1- 4 b_1 c_1) x_2^2 - 4 b_1 e_1=\left\{
    \begin{array}{c}
        -4 a_1 c_1(x_2^2-\xi)(x_2^2-\eta) \text{ if } a_1 c_1\neq 0, \\
        (1 - 4 a_1 e_1- 4 b_1 c_1)(x_2^2-\rho) \text{ if } a_1 c_1=0 \\
    \end{array}
    \right.
    $$
    where $\xi,\eta,\rho \in \cc$ and $\xi \neq \eta$. So the only square factor that $\Delta_1$ might have is $x_2^2$ (i.e. $b_1 e_1=0$). All in all, systems (\ref{eq-c3or4}) are just saying $\frac{\Delta_1}{\Delta_2}$ is a square in $K_2$ so formally $K_2(x_1)=K_2(\sqrt{\Delta_1})=K_2(\sqrt{\Delta_2})=K_2(x_3)$ hence systems (\ref{eq-c3or4}) $\Leftrightarrow$ condition (2). Finally, when all $g^{(i)}$ are irreducible, by Definition \ref{casedefn} it is easy to see that components of \textbf{Case 3} and \textbf{Case 4} can be characterized by their common property $K_1(x_2,x_3)=K_1(x_2)$, i.e. condition (1) $\Leftrightarrow S$ is equimodular.
\end{proof}

\begin{proof}(Theorem \ref{main-ps})
    Since we have already proved Theorem \ref{main-ps} for involutive-rational class, together with Corollary \ref{n-is-c} we can conclude that, if $S$ admits components of \textbf{Case 3} and \textbf{Case 4}, then only one of systems (\ref{eq-c3}) holds. Conversely, suppose $\left\{\frac{a_1}{a_2}=\frac{b_1}{c_2}=\frac{b_2}{c_1}=\frac{e_2}{e_1}(=k)\right\}$ holds but $\left\{\frac{a_1}{b_2}=\frac{b_1}{e_2}=\frac{a_2}{c_1}=\frac{c_2}{e_1}\right\}$ fails. According to Corollary \ref{n-is-c} we know $S$ admits a component of \textbf{Case 3}. Symbolic computation can show
    $$
    (a_2,b_2,c_2,e_2)=\left(\frac{a_1}{k},kc_1,\frac{b_1}{k},ke_1\right) \Rightarrow \frac{\text{Res}(g^{(1)},g^{(2)};x_2)}{(x_3-kx_1)^2}=\left(\frac{a_1b_1}{k^2}x_1^2x_3^2+...+c_1e_1\right).
    $$
    We claim that $r=\frac{a_1b_1}{k^2}x_1^2x_3^2+...+c_1e_1$ is irreducible and quadratic in $x_3$. Otherwise, according to Corollary \ref{noncons}, there must be an irreducible factor $r'$ of $r$ which is linear in $x_3$ hence we have a component $Z(g^{(1)}, g^{(2)}, r', H^{(1)}, H^{(2)})$ of \textbf{Case 3} by Definition \ref{casedefn}. Consider the proof of Corollary \ref{n-is-c} we know $r'=x_3-k'x_1$ or $x_1x_3-k'$ for some $k'\neq 0$. By Lemma \ref{n-is}, it must be $r'=x_3-k'x_1=x_3-kx_1$ since we cannot have $\left\{\frac{a_1}{b_2}=\frac{b_1}{e_2}=\frac{a_2}{c_1}=\frac{c_2}{e_1}\right\}$. In the meanwhile, the highest degrees of $x_1, x_3$ in $r$ are at most 2. This leads to $a_1b_1=c_1e_1=0$ given $r'|r$ hence $g^{(1)}$ is either singular or reducible by Lemma \ref{rational}, a contradiction! So finally we have another component
    $$
    (g^{(1)}, g^{(2)}, r, H^{(1)}, H^{(2)})
    $$
    with the extension diagram in which $[K_1(x_3):K_1]=2$ (i.e. $x_3\notin K_1$). Notice that 
    $$
    \left\{\frac{a_1}{a_2}=\frac{b_1}{c_2}=\frac{b_2}{c_1}=\frac{e_2}{e_1}\right\} \Rightarrow \left\{\frac{a_1 c_1}{a_2 b_2}=\frac{1- 4 a_1 e_1- 4 b_1 c_1}{1 - 4 a_2 e_2- 4 b_2 c_2}=\frac{b_1 e_1}{c_2 e_2}\right\}.
    $$
    Hence by Lemma \ref{samedsc} we know $W$ is of \textbf{Case 4}. The proof is similar when $\left\{\frac{a_1}{b_2}=\frac{b_1}{e_2}=\frac{a_2}{c_1}=\frac{c_2}{e_1}\right\}$ holds but $\left\{\frac{a_1}{a_2}=\frac{b_1}{c_2}=\frac{b_2}{c_1}=\frac{e_2}{e_1}\right\}$ fails.
\end{proof}

\subsubsection{Purely-quadratic}

The proof of Theorem \ref{main-ps} is a direct consequence of Corollary \ref{n-is-c} and Lemma \ref{samedsc}.

This class is currently "floating in the air" since we only proved its existence. In Theorem \ref{main-ps} when the condition of purely-quadratic couplings holds, we are expecting a component of \textbf{Case 4} with an explicit representation from equation (\ref{irr})
$$
W=Z(g^{(1)}, g^{(2)}, r_k, H^{(1)}, H^{(2)})
$$
where $r_k(x_1,x_3)$ is an irreducible factor of $\text{Res}(g^{(1)},g^{(2)};x_2)$. In fact, for every component of \textbf{Case 4}, we have $r_k$ in the form
$$
r_k(x_1,x_3)=(\alpha_{22} x_1^2 + \alpha_{02} ) x_3^2 + \alpha_{11}x_1 x_3 + (\alpha_{20} x_1^2+ \alpha_{00}).
$$
The proof can be found later in Corollary \ref{even-term} where all $\alpha_{ij}$ are given explicitly.

\subsubsection{Involutive-quadratic}

The proof of Theorem \ref{main-ps} is a consequence of the following lemma.

\begin{lemma}\label{samemini}
    Given a non-singular coupling $S=(g^{(1)},g^{(2)})$, then $S$ only admits components of \textbf{Case 5} if and only if system (\ref{eq-c5}) holds. Moreover, the factorization of the resultant is
    $$
    \text{Res}(g^{(1)},g^{(2)};x_2)=\frac{b_1}{a_1}(a_2 x_1 x_3^2 - (a_1 x_1^2 + c_1) x_3 + b_2 x_1)^2.
    $$
\end{lemma}

\begin{remark}\label{ae=bc}
    Under the condition of system (\ref{eq-c5}), we have $a_ie_i=b_ic_i$ for both $i=1,2$. In a real flexible linkage, this implies the so-called \textbf{orthodiagonal} property which is $\cos(\lambda_i)\cos(\mu_i)=\cos(\gamma_i)\cos(\delta_i)$.
\end{remark}

\subsubsection{Quartic}

\begin{proof}(Theorem \ref{main-ps})
    Suppose $S$ does not satisfy the requirements of other classes. In particular, both $g^{(1)}, g^{(2)}$ are irreducible. Firstly, $S$ does not have components of \textbf{Case 1} or \textbf{Case 2} which requires reducibility on $g^{(i)}$. Secondly, $S$ does not have components of \textbf{Case 3} or \textbf{Case 4} by Corollary \ref{n-is-c} and Lemma \ref{samedsc}. Finally, $S$ does not have components of \textbf{Case 5} by Lemma \ref{samemini}. Thus $S$ only admits component of \textbf{Case 6}.
\end{proof}

\section{General type without (anti)deltoids}\label{General}

\subsection{Section outline and classification theorem}

In this section, we study the couplings from matchings of general type in which not all $F_i\in \{0, \infty \}$. So, up to a rotation, we may always assume $F_1\notin \{0, \infty \}$. We follow our approach from pseudo-planar type and make a case-by-case study. It is interesting to mention that condition $F_1\notin \{0, \infty \}$ implies even more restrictions on coefficients compared to pseudo-planar type. One will see that involutive-rational, rational-quadratic, and purely-quadratic couplings of general type are completely useless for the construction of flexible meshes. So, to avoid unnecessary discussion, we will treat them all together as equimodular class (see Definition \ref{dfn-ps}).

\begin{theorem}\label{main-g}
    Given a non-singular coupling $S=(\Tilde{g}^{(1)}, g^{(2)}, H^{(1)}, H^{(2)})$ where $F_1\notin\{0,\infty\}$. According to Definition \ref{dfn-ps}, each class of $S$ is equivalent to one of the corresponding conditions. i.e. $S$ is
    \begin{itemize}
        \item [$\bullet$] \textbf{purely-rational} if both $\Tilde{g}^{(1)}, g^{(2)}$ are reducible;
        \item [$\bullet$] \textbf{half-quadratic} if only one of $\Tilde{g}^{(1)}, g^{(2)}$ is reducible;   
        \end{itemize} 
    Further, when both $\Tilde{g}^{(1)}, g^{(2)}$ are irreducible, $S$ is
    \begin{itemize}
        \item [$\bullet$] \textbf{equimodular} if 
        \end{itemize}
        \begin{equation}\label{eq-c4-g}
        \left\{\frac{a_1}{e_1}=\frac{b_1}{c_1}, \frac{a_2}{e_2}=\frac{c_2}{b_2}, F_1=\pm 1, \frac{1-4a_1e_1-4b_1c_1-8a_1c_1}{16a_2b_2}=\frac{16a_1c_1}{1-4a_2e_2-4b_2c_2-8a_2b_2}\right\};   
        \end{equation}
        \begin{itemize}
        \item [$\bullet$] \textbf{involutive-quadratic} if
        \begin{equation}\label{eq-c5-g}
            \left\{\begin{array}{l}
                 \left\{\frac{a_1}{b_1}=\frac{c_1}{e_1}=\frac{a_2}{c_2}=\frac{b_2}{e_2}=-1, F_1\neq\pm 1\right\} \text{ or}  \\
                 \left\{\frac{a_1}{b_1}=\frac{c_1}{e_1}=\frac{a_2}{c_2}=\frac{b_2}{e_2}=-1, F_1=\pm 1, 256a_1c_1a_2b_2\neq 1\right\};
            \end{array}
            \right.
        \end{equation}
        \item [$\bullet$] \textbf{quartic} if otherwise.
    \end{itemize}
\end{theorem}

The proof has been split into cases which will be detailed studied.

\subsection{Case-by-case study}

\subsubsection{Purely-rational}

The proof of Theorem \ref{main-g} is left to the reader. In fact, the proof for purely-rational class and half-quadratic class can be directly copied from their counterparts in pseudo-planar type, Section \ref{pseudo}. This is because in the corresponding extension diagram, $[K_i(y_j):K_i]=[K_i(x_{j+1}):K_i]$ for all $i,j$. Hence by Definition \ref{casedefn}, the component remains in the same \textbf{Case} no matter $F_1\in \{0, \infty\}$ or not.

\subsubsection{Half-quadratic}

The proof of Theorem \ref{main-g} was mentioned above. Nevertheless, given a couping $(\Tilde{g}^{(1)}, g^{(2)}, H^{(1)}, H^{(2)})$ which admits a component $W$ of \textbf{Case 2}, we know it is either $\Tilde{g}^{(1)}$ is reducible or $g^{(2)}$ is reducible (but not both). So when $\Tilde{g}^{(1)}$ is reducible, in the extension diagram of $W$ we have $x_2\in K_1$ and the minimal polynomial of $y_2$ on $K_1$ can be obtained from $g^{(2)}$ through substitution of $x_2=x_2(x_1)$; similarly, when $g^{(2)}$ is reducible we have $x_2\in K_3=\cc(y_2)$ and the minimal polynomial of $y_2$ on $K_1$ can be obtained from $\Tilde{g}^{(1)}$ through substitution $x_2=x_2(y_2)$.

\subsubsection{Equimodular}

As in pseudo-planar type, the problem can be reduced to identical discriminants up to a square in $K_2$. To start up, fix an extension diagram and change the base field to $K_2=\cc(x_2)=\cc(y_1)$ so $x_1, y_2$ are algebraic elements on $K_2$. From $g^{(1)}, \Tilde{g}^{(1)}$, and $g^{(2)}$ respectively, we have discriminants

\begin{equation}\label{d2}
    \left\{\begin{array}{l}
        \Delta_1'=-4 a_1 c_1 y_1^4 + (1 - 4 a_1 e_1- 4 b_1 c_1) y_1^2 - 4 b_1 e_1,\\
        \Delta_1=-4 a_1 c_1 (x_2+F_1)^4 + (1 - 4 a_1 e_1- 4 b_1 c_1) (x_2+F_1)^2(1-F_1x_2)^2 - 4 b_1 e_1 (1-F_1x_2)^4,\\  
        \Delta_2=-4 a_2 b_2 x_2^4 + (1 - 4 a_2 e_2- 4 b_2 c_2) x_2^2 - 4 c_2 e_2.\\
    \end{array}\right.
\end{equation}
Knowing that $H^{(1)}=0$, $\Delta_1$ is actually obtained from $\Delta_1'$ by substitution $y_1=\frac{x_2+F_1}{1-F_1x_2}$.

We skip the proof of the following lemma since it is a complete analog of Lemma \ref{samedsc} from pseudo-planar type, however, it applies to both types.

\begin{lemma}\label{samedsc-g}
    Given a non-singular coupling $S=(\Tilde{g}^{(1)}, g^{(2)}, H^{(1)}, H^{(2)})$ where $\Tilde{g}^{(1)}, g^{(2)}$ are irreducible. $S$ is equimodular if and only if in every extension diagram of its components, one of the following equivalent conditions holds:
    \begin{itemize}
        \item [(1)] $y_2 \in K_1(x_2)$; (2) $y_2 \in K_2(x_1)$; (3) $x_1 \in K_2(y_2)$; (4) $x_1 \in K_3(x_2)$.
    \end{itemize} 
    On the other hand, regarding $\Tilde{g}^{(1)}(t,x_2), g^{(2)}(x_2,t)$ as polynomials in $K_2[t]$ with discriminants $\Delta_1, \Delta_2$ respectively, $S$ is equimodular if and only if $\exists f\in K_2$ such that $f^2=\frac{\Delta_1}{\Delta_2}$.
\end{lemma}

The proof of Theorem \ref{main-g} is a consequence of the following lemma.

\begin{lemma}\label{samedsc-g'}
    Given a non-singular coupling $S=(\Tilde{g}^{(1)}, g^{(2)}, H^{(1)}, H^{(2)})$ where $\Tilde{g}^{(1)}, g^{(2)}$ are irreducible and $F_1\notin \{0, \infty\}$. $S$ is equimodular if and only if system (\ref{eq-c4-g}) holds.
\end{lemma}

When $\Tilde{g}^{(1)}$ is irreducible, it determines the minimal polynomial of $x_2$ on $K_1$ in every component. It is convenient to use the expression
$$
\Tilde{g}^{(1)}=h_2(x_1) x_2^2 + h_1(x_1) x_2 + h_0(x_1)
$$
where
\begin{equation}\label{h012}
    \left\{\begin{array}{l}
    h_2(x_1)=(b_1 F_1^2 +a_1) x_1^2 - F_1 x_1 + e_1 F_1^2 +c_1, \\
    h_1(x_1)=2 F_1 (a_1 - b_1) x_1^2 + (1 - F_1^2) x_1 + 2 F_1 (c_1 - e_1), \\
    h_0(x_1)=(a_1 F_1^2 +b_1) x_1^2 + F_1 x_1 + c_1 F_1^2 +e_1. \\
\end{array}\right.
\end{equation}
Now let us suppose the component is of \textbf{Case 3} with the extension diagram in which $y_2\in K_1$, we have

\begin{lemma}\label{n-is-g}
    If a non-singular coupling $(\Tilde{g}^{(1)}, g^{(2)}, H^{(1)}, H^{(2)})$ admits a component of \textbf{Case 1} or \textbf{Case 3}, then the relation between $x_1$ and $y_2$ in the extension diagram must be in the form $y_2=\frac{px_1+q}{rx_1+s}, ps-qr\neq 0$.
\end{lemma}

\begin{definition}\label{realcom}
    For a component $W$ of a coupling, we say $W$ is \textbf{real} if $W\cap (\rp)^5$ is an infinite set.
\end{definition}

A surprising fact about equimodular couplings of general type is that they have no real components.

\begin{proposition}\label{real-sol}
    An equimodular coupling $S=(\Tilde{g}^{(1)}, g^{(2)}, H^{(1)}, H^{(2)})$ admits a real component only if $F_1\in \{0, \infty\}$ and
    \begin{equation}\label{+ratio}
    \left\{\begin{array}{l}
        \frac{a_1 c_1}{a_2 b_2}=\frac{1- 4 a_1 e_1- 4 b_1 c_1}{1 - 4 a_2 e_2- 4 b_2 c_2}=\frac{b_1 e_1}{c_2 e_2}>0 \text{ if } F_1=0,\\
        \frac{b_1 e_1}{a_2 b_2}=\frac{1- 4 a_1 e_1- 4 b_1 c_1}{1 - 4 a_2 e_2- 4 b_2 c_2}=\frac{a_1 c_1}{c_2 e_2}>0 \text{ if } F_1=\infty.\\         
        \end{array}\right.
    \end{equation}
\end{proposition}

The above proposition tells us, in reality, any two spherical quads as coupled in Fig.~\ref{squads} left can never determine an equimodular coupling of general type.

\subsubsection{Involutive-quadratic}

The proof of Theorem \ref{main-g} is a consequence of a stronger version of Lemma \ref{samemini}.

\begin{lemma}\label{samemini-g}
    Given a non-singular coupling $S=(\Tilde{g}^{(1)}, g^{(2)}, H^{(1)}, H^{(2)})$ where $F_1\notin \{0,\infty \}$, then $S$ only admits components of \textbf{Case 5} if and only if one of systems (\ref{eq-c5-g}) holds. Moreover, the factorization of the resultant is
    $$
    \text{Res}(\Tilde{g}^{(1)},g^{(2)};x_2)=-(a_2 h_1(x_1) y_2^2 - h_2(x_1) y_2 + b_2 h_1(x_1))^2
    $$
    where $h_i(x_1)$ are given in equation (\ref{h012}).
\end{lemma}

\begin{remark}
    The conditions in Lemma \ref{samemini-g} impose restrictions on the shape of the quads, i.e.  
    $$
    \frac{a_1}{b_1}=\frac{c_1}{e_1}=\frac{a_2}{c_2}=\frac{b_2}{e_2}=-1 \Leftrightarrow \delta_1=\mu_1=\gamma_2=\mu_2=\frac{\pi}{2}.
    $$
\end{remark}

\subsubsection{Quartic}

Considering $\Tilde{g}^{(1)},g^{(2)}$ are irreducible, the proof of Theorem \ref{main-g} is a consequence of Lemma \ref{samedsc-g'} and \ref{samemini-g}.

\section{Resultant analysis}\label{Res}

\subsection{Section outline and factorization theorem}

For a given coupling, the following theorem describes its class in a different way. This new perspective is very important for the construction of matchings (i.e. flexible meshes) using Theorem \ref{flexequi}.

\begin{theorem}\label{factor}
    Depending on the class of a non-singular coupling $S=(\Tilde{g}^{(1)}, g^{(2)}, H^{(1)}, H^{(2)})$, the factorization of $R_1:=\text{Res}(\Tilde{g}^{(1)},g^{(2)}; x_2)$ is as follows:
    \begin{itemize}
        \item [$\bullet$] \textbf{purely-rational}: $R_1(x_1,y_2)=\prod_{i=1}^4(r_i x_1 y_2 - p_i x_1 + s_i y_2 - q_i)$;
        \item [$\bullet$] \textbf{half-quadratic}: $R_1(x_1,y_2)=r(x_1,y_2)\cdot r'(x_1,y_2)$ where $r, r'$ are irreducible and quadratic in both $x_1$ and $y_2$;     
        \item [$\bullet$] \textbf{involutive-rational}: $R_1(x_1,y_2)=\prod_{i=1}^2(r_i x_1 y_2 - p_i x_1 + s_i y_2 - q_i)^2$;
        \item [$\bullet$] \textbf{rational-quadratic}: $R_1(x_1,y_2)=(r_1 x_1 y_2 - p_1 x_1 + s_1 y_2 - q_1)^2\cdot r(x_1,y_2)$ where $r$ is irreducible and quadratic in both $x_1$ and $y_2$;
        \item [$\bullet$] \textbf{purely-quadratic}: $R_1(x_1,y_2)=r(x_1,y_2)\cdot r'(x_1,y_2)$ where $r, r'$ are irreducible and quadratic in both $x_1$ and $y_2$;
        \item [$\bullet$] \textbf{involutive-quadratic}: $R_1(x_1,y_2)=\frac{b_1}{a_1}(a_2 h_1(x_1) y_2^2 - h_2(x_1) y_2 + b_2 h_1(x_1))^2$ where $h_i(x_1)$ are the coefficients of $x_2^i$ in $\Tilde{g}^{(1)}$ for $i=0,1,2$ respectively;
        \item [$\bullet$] \textbf{quartic}: $R_1(x_1,y_2)$ is irreducible.
    \end{itemize}
\end{theorem}

The proof will be given in the next section where $R_1$ is factorized explicitly.

\subsection{Symmetric Extension Theorem}\label{fac}

\begin{lemma}\label{symrational}
    In the extension diagram of a component of a non-singular coupling, consider switching the parameter between $x_1$ and $x_3$, $x_3$ can be rationally expressed by $x_1$ if and only if $x_1$ can be rationally expressed by $x_3$.
\end{lemma}
\begin{proof}
    One direction has already been proved in Lemma \ref{n-is-g}. For the other direction, carefully check the proof of Lemma \ref{n-is-g}, one will find that if we alternate $(\Tilde{g}^{(1)},\Tilde{g}^{(2)})$ to $(\Tilde{g}^{(2)},\Tilde{g}^{(1)})$ and consider the problem in $K_3$, all arguments are still valid. Thus $x_3$ is rational in $x_1$ if and only if $x_1$ is rational in $x_3$.
\end{proof}

\begin{theorem}[Symmetric Extension]\label{symext}
    In the extension diagram of a component of a non-singular coupling $(\Tilde{g}^{(1)}, g^{(2)}, H^{(1)}, H^{(2)})$, $[K_i(x_j):K_i]=[K_j(x_i):K_j]$ for all $1\leq i, j \leq 3$. In addition, the extension diagrams with respect to $x_1$ and $x_3$ always belong to the same \textbf{Case} of Definition \ref{casedefn}.
\end{theorem}

The last statement is a bit subtle and can be regarded as a generalization of Lemma \ref{symrational}. In other words, no matter how we interchange the parameters between $x_1$ and $x_3$, so the extension diagram might be upside down, the extension degrees always fit the same \textbf{Case} of Definition \ref{casedefn}.


\begin{corollary}\label{samedg}
    For a non-singular coupling $(\Tilde{g}^{(1)}, g^{(2)}, H^{(1)}, H^{(2)})$, every factor of $R_1(x_1,y_2)$ (or $\text{Res}(\Tilde{g}^{(1)},\Tilde{g}^{(2)}; x_2)$) has the same degree in $x_1$ and $y_2$ (or $x_1$ and $x_3$).
\end{corollary}

\begin{proposition}\label{4dg}
    For a non-singular coupling $(\Tilde{g}^{(1)}, g^{(2)}, H^{(1)}, H^{(2)})$, $R_1(x_1,y_2)$ is always in 4th degree of $x_1$ and $y_2$.
\end{proposition}
Recall Proposition \ref{eqcoup}, each component is related to a factor $r$ of $R_1$. The degree of $y_2$ in $r$ reflects the extension degree $[K_1(y_2):K_1]$ in Definition \ref{casedefn}. According to the component types that a coupling contains, the factorization of $R_1$ therefore becomes very predictable thanks to Corollary \ref{samedg} and Proposition \ref{4dg}.
\begin{proof}(Theorem \ref{factor})
    Firstly, according to table (\ref{table}) and Corollary \ref{samedg}, the component types explain the degree of $x_1$ and $y_2$ in each irreducible factor of $R_1$. In the light of  Proposition \ref{4dg}, the irreducibility of the factors given in Theorem \ref{factor} becomes obvious. The only thing left is to verify if $R_1$ can be factorized as the theorem claimed. For efficiency, we only provide proper reparametrizations so that the symbolic computation can factorize $R_1$ explicitly.
\begin{itemize}
        \item [$\bullet$] \textbf{purely-rational} or \textbf{half-quadratic}: Keep the original expression of $g^{(i)}$ if irreducible, otherwise use the expressions in Lemma \ref{rational}.
        \item [$\bullet$] \textbf{involutive-rational} or \textbf{rational-quadratic}: Set $F_1=0$, check the proof of Theorem \ref{main-ps} for these two classes and replace $x_3$ by $y_2$. The proof is similar for $F_1=\infty$. It is unnecessary to consider the general type $F_1\notin \{0,\infty\}$ due to Proposition \ref{real-sol}. However, the factorizations stated in Theorem \ref{factor} still hold for general type (see Appendix \ref{techpf}).   
        \item [$\bullet$] \textbf{purely-quadratic}: Given table (\ref{table}) and Proposition \ref{4dg}, it is clear that $R_1=r r'$ where $r, r'$ are irreducible and quadratic in both $x_1$ and $y_2$. However, an explicit factorization is only useful with the restrictions in Proposition \ref{real-sol}. Without loss of generality, we may assume $F_1=0$ and    
        \begin{equation}\label{repara}
        \frac{a_1 c_1}{a_2 b_2}=\frac{1- 4 a_1 e_1- 4 b_1 c_1}{1 - 4 a_2 e_2- 4 b_2 c_2}=\frac{b_1 e_1}{c_2 e_2}=m^2>0 \Rightarrow \left\{\begin{array}{l}
        a_1=\frac{b_1b_2(4b_2c_2m^2-4b_1c_1-m^2+1)}{4e_2(b_2c_2m^2-b_1c_1)}, \\
        e_1=\frac{c_2e_2m^2}{b_1},   \\
        a_2=\frac{b_1c_1(4b_2c_2m^2-4b_1c_1-m^2+1)}{4e_2m^2(b_2c_2m^2-b_1c_1)}. \\
        \end{array}\right.
        \end{equation}
        According to Corollary \ref{redu=iso}, $\{a_i, b_i, c_i, e_i\}$ includes at most one zero for $i=1, 2$. In addition, given $\frac{a_1 c_1}{a_2 b_2}=\frac{b_1 e_1}{c_2 e_2}>0$, the zeros, if exist, must gather in one fraction. Hence it is harmless to assume $b_1c_1e_1b_2c_2e_2\neq 0$ so $a_1, e_1, a_2$ can be parametrized as in equation (\ref{repara}). Nevertheless, we should mention $b_2c_2m^2-b_1c_1\neq 0$. Otherwise
        $$
        m^2=\frac{b_1c_1}{b_2c_2}\Rightarrow \left\{\frac{a_1}{a_2}=\frac{b_1}{c_2},\frac{c_1}{b_2}=\frac{e_1}{e_2} \right\}\Rightarrow m^2=\frac{a_1 e_1}{a_2 e_2}=\frac{b_1 c_1}{b_2 c_2}=\frac{1- 4 a_1 e_1- 4 b_1 c_1}{1 - 4 a_2 e_2- 4 b_2 c_2}=1,
        $$
        contradicts to systems (\ref{eq-c4}).
        \item [$\bullet$] \textbf{involutive-quadratic}: Check the counterparts of the proof of Theorem \ref{main-ps} and \ref{main-g}.
        \item [$\bullet$] \textbf{quartic}: Trivial.
    \end{itemize} 
\end{proof}

\begin{corollary}\label{even-term}
     Suppose coupling $S$ is rational-quadratic or purely-quadratic. If $S$ admits a real component (Definition \ref{realcom}) of \textbf{Case 4}, the quadratic factors $r, r'$ of $R_1$ in Theorem \ref{factor} are in the form 
    $$
    (\alpha_{22} x_1^2 + \alpha_{02} ) y_2^2 + \alpha_{11}x_1 y_2 + (\alpha_{20} x_1^2+ \alpha_{00})
    $$
    where $\alpha_{ij}$ are given in table (\ref{aij}).
\end{corollary}

\begin{proof}
    Without loss of generality, we may assume $F_1=0$ due to Proposition \ref{real-sol}. Using the parametrizations as suggested in the proof of Theorem \ref{factor}, symbolic computation immediately shows table (\ref{aij}) for rational-quadratic and purely-quadratic couplings respectively.
\end{proof}

\begin{equation}\label{aij}
    \begin{tabular}{ |c|c|c|c| }
    \hline
        & rational-quadratic & rational-quadratic & purely-quadratic \\
    \hline
     & $\frac{a_1}{a_2}=\frac{b_1}{c_2}=\frac{b_2}{c_1}=\frac{e_2}{e_1}=k$ & $\frac{a_1}{b_2}=\frac{b_1}{e_2}=\frac{a_2}{c_1}=\frac{c_2}{e_1}=\frac{1}{k}$ & $\frac{a_1 c_1}{a_2 b_2}=\frac{1- 4 a_1 e_1- 4 b_1 c_1}{1 - 4 a_2 e_2- 4 b_2 c_2}=\frac{b_1 e_1}{c_2 e_2}=m^2$\\
    \hline
      $\alpha_{22}$ & $\frac{a_1b_1}{k^2}$ & $\frac{(a_1e_1 - b_1c_1)^2}{k^2}$ & $\frac{b_1(1-m^2-4b_1c_1+4b_2c_2m^2)}{4e_2m^2}$\\
    \hline
      $\alpha_{20}$ & $(a_1e_1 - b_1c_1)^2$ & $a_1b_1$ & $\frac{b_1b_2(1\pm m)^2}{4(b_2c_2m^2-b_1c_1)}$\\
    \hline
      $\alpha_{02}$ & $\frac{(a_1e_1 - b_1c_1)^2}{k^2}$ & $\frac{c_1e_1}{k^2}$ & $\frac{c_1c_2(1\pm m)^2}{4(b_2c_2m^2-b_1c_1)}$\\
    \hline
      $\alpha_{11}$ & $\frac{2(a_1e_1 - b_1c_1)^2 - a_1e_1 - b_1c_1}{k}$ & $\frac{2(a_1e_1 - b_1c_1)^2 - a_1e_1 - b_1c_1}{k}$ & $\frac{(1\pm m)(b_1c_1 \pm b_2c_2m^3)-4(b_2c_2m^2-b_1c_1)^2}{2m(b_2c_2m^2-b_1c_1)}$\\
    \hline
      $\alpha_{00}$ & $c_1e_1$ & $(a_1e_1 - b_1c_1)^2$ & $\frac{e_2(b_2c_2m^2-b_1c_1)}{b_1}$\\
    \hline
    \end{tabular}
\end{equation}

\section{Back to the mesh}\label{back-to-com}

\subsection{Section outline and the classification of flexible meshes}

As one already knew, each flexible mesh uniquely determines a matching $M=(\Tilde{g}^{(1)},\Tilde{g}^{(2)},\Tilde{g}^{(3)},\Tilde{g}^{(4)})$ hence the mesh can be classified by its corresponding matching. According to Proposition \ref{eqcoup}, $M$ can be regarded as the sum of two couplings
$$
S_1=(\Tilde{g}^{(1)},\Tilde{g}^{(2)},H^{(1)},H^{(2)}),\,\, S_2=(\Tilde{g}^{(3)},\Tilde{g}^{(4)},H^{(3)},H^{(4)}).
$$
And by Theorem \ref{flexequi}, we must have $\gcd(\text{Res}(\Tilde{g}^{(1)},\Tilde{g}^{(2)};x_2),\text{Res}(\Tilde{g}^{(3)},\Tilde{g}^{(4)};x_4))\neq 1$. Hence a very straightforward classification can be given by the factors of the gcd. To achieve this, Theorem \ref{factor} is a powerful tool. Notice that in the extension diagram of any component, the minimal polynomial of $x_3$ on $K_1$ can be induced by the minimal polynomial of $y_2$ on $K_1$ through substitution $y_2=\frac{x_3+F_2}{1-F_2x_3}$. Consequently, the factorization of $\text{Res}(\Tilde{g}^{(1)},\Tilde{g}^{(2)}; x_2)$ can be induced by the factorization of $R_1(x_1,y_2)$ through the same substitution. Thus by Theorem \ref{factor} and Corollary \ref{samedg}, the classification of flexible meshes can be naturally given as below.

\begin{definition}\label{11class}
    Given a flexible mesh that determines a non-singular matching 
    $$
    M:=(\Tilde{g}^{(1)},\Tilde{g}^{(2)},\Tilde{g}^{(3)},\Tilde{g}^{(4)}),
    $$
    the corresponding couplings are
    $$
    S_1:=(\Tilde{g}^{(1)},\Tilde{g}^{(2)},H^{(1)},H^{(2)}),\,\, S_2:=(\Tilde{g}^{(3)},\Tilde{g}^{(4)},H^{(3)},H^{(4)}).
    $$
    Set
    $$
    \Tilde{R}^{(1)}(x_1,x_3):=\text{Res}(\Tilde{g}^{(1)},\Tilde{g}^{(2)};x_2),\,\, \Tilde{R}^{(2)}(x_1,x_3):=\text{Res}(\Tilde{g}^{(3)},\Tilde{g}^{(4)};x_4),\,\, d_M:=\gcd(\Tilde{R}^{(1)},\Tilde{R}^{(2)}).
    $$
    We say $M$ (or the mesh) is 
    \begin{itemize}
        \item [$\bullet$] \textbf{PR} if both $S_1, S_2$ are purely-rational so $d_M$ has linear factors only;
        \item [$\bullet$] \textbf{HQ} if both $S_1, S_2$ are half-quadratic so $d_M$ has quadratic factors only;
        \item [$\bullet$] \textbf{IR} if for $i=1, 2$, $S_i$ is either involutive-rational or rational-quadratic and $d_M$ has linear factors only;
        \item [$\bullet$] \textbf{RQ} if both $S_1, S_2$ are rational-quadratic and $d_M$ has linear and quadratic factors;
        \item [$\bullet$] \textbf{PQ} if for $i=1, 2$, $S_i$ is either rational-quadratic or purely-quadratic and $d_M$ has quadratic factors only;
        \item [$\bullet$] \textbf{IQ} if both $S_1, S_2$ are involutive-quadratic so $d_M$ has quadratic factors only;
        \item [$\bullet$] \textbf{Q} if both $S_1, S_2$ are quartic so $d_M=\Tilde{R}^{(1)}$ is irreducible;
        \item [$\bullet$] \textbf{PR + IR} if one of $S_i$ is purely-rational and the other is either involutive-rational or rational-quadratic so $d_M$ has linear factors only;
        \item [$\bullet$] \textbf{HQ + IQ} if one of $S_i$ is half-quadratic and the other is involutive-quadratic so $d_M$ has quadratic factors only;
        \item [$\bullet$] \textbf{HQ + PQ} if one of $S_i$ is half-quadratic and the other is either rational-quadratic or purely-quadratic so $d_M$ has quadratic factors only;
        \item [$\bullet$] \textbf{PQ + IQ} if one of $S_i$ is involutive-quadratic and the other is either rational-quadratic or purely-quadratic so $d_M$ has quadratic factors only.
    \end{itemize} 
\end{definition}

The above definition just proved Theorem \ref{main}, and in the rest of this section, $M, S_1, S_2, \Tilde{R}^{(1)},\Tilde{R}^{(2)}$, and $d_M$ always stay their meanings from Definition \ref{11class}.

\subsection{General construction of matchings}\label{Generalap}

The construction of matchings is equivalent to solving a polynomial system on
\begin{equation}\label{as-f4}
    \{a_3, b_3, c_3, e_3, a_4, b_4, c_4, e_4, F_2, F_3, F_4\}.
\end{equation}
For example, we arbitrarily set $\{\Tilde{g}^{(1)}, g^{(2)}\}$, as known as
$$
\{a_1, b_1, c_1, e_1, a_2, b_2, c_2, e_2, F_1\},
$$
then we have an explicit factorization of $R_1(x_1,y_2)$ by Theorem \ref{factor}. Therefore we have the factorization of $\Tilde{R}^{(1)}$ with $F_2$ as a parameter in it. On the other hand, to construct a specific matching, we need to settle down the class of $S_2$ so that $S_1$ and $S_2$ can form a valid combination in Definition \ref{11class}. And once we know the class of $S_2$, more restrictions from Theorem \ref{main-ps} or \ref{main-g} will be added on
$$
\{a_3, b_3, c_3, e_3, a_4, b_4, c_4, e_4, F_3, F_4\}.
$$
In the meanwhile, Theorem \ref{factor} also allows us to factorize $\Tilde{R}^{(2)}$ symbolically. Finally, given $d_M\neq 1$, there must be two irreducible factors of $\Tilde{R}^{(1)}, \Tilde{R}^{(2)}$ respectively such that their coefficients of $x_1^ix_3^j$ are proportional hence we get the desired polynomial system on (\ref{as-f4}).

A matching $M$ always comes with an infinite zero set $Z(M) \cong Z(S_1)\times_{\{x_1, x_3\}} Z(S_2)$ and the essential reason for $Z(M)$ being infinite is that there are components $W_1, W_2$ of $Z(S_1), Z(S_2)$ respectively such that $W_1\times_{\{x_1, x_3\}} W_2$ is infinite. It is easy to see that the first seven classes from \textbf{PR} to \textbf{Q} in Definition \ref{11class} require $W_1, W_2$ belonging to the same \textbf{Case} of Definition \ref{casedefn}, hence we call them \textbf{simple classes}. On the contrary, the rest from \textbf{PR + IR} to \textbf{PQ + IQ} are called \textbf{hybrid classes} since $W_1, W_2$ are required to be of different \textbf{Cases}. In this regard,

\begin{definition}
    Given a matching $M$ with corresponding couplings $S_1, S_2$. A component $W_i$ of $S_i$ is said to be \textbf{valid} if for $j\neq i$, there exists a component $W_j$ of $S_j$ such that $W_i\times_{\{x_1, x_3\}} W_j$ is an infinite set.
\end{definition}

\subsection{Simple classes}

\subsubsection{PR: $S_1, S_2$ only admit valid components of \textbf{Case 1}.}

In such a class both $S_1, S_2$ are purely-rational so all $\Tilde{g}^{(i)}$ and $g^{(i)}$ are reducible. The name 'purely-rational' is based on the fact that in the extension diagram of every component of $S_i$, all field extensions are trivial. Geometrically speaking, all forming quads in the spherical linkage of a \textbf{PR} mesh are (anti)isogonal (Corollary \ref{redu=iso}). This means all \textbf{PR} meshes are just those we described in Theorem \ref{PR}. For illustration, we provide an algebraic construction when all quads are antiisogonal. In other words, all \textbf{PR} matchings $M$ in which $a_i=e_i=0$, $\forall i\in\{1,2,3,4\}$ (Lemma \ref{i-vi}), can be constructed as below.

\begin{example}\label{construction-PR}
    Here we construct matchings $(\Tilde{g}^{(1)},\Tilde{g}^{(2)},\Tilde{g}^{(3)},\Tilde{g}^{(4)})$ in which $a_i=e_i=0$ for all $i\in\{1,2,3,4\}$. First, set
    $$
    \{b_1\neq 0,c_1\neq 0,b_2\neq 0,c_2\neq 0,b_3\neq 0,c_3\neq 0,c_4\neq 0,F_1,F_2\}
    $$
    as free parameters and 
    $$
    \{b_4\neq 0,F_3,F_4\}
    $$
    undetermined. Following the proof of Theorem \ref{PR} we have 
    $$
    N_i=\left(\begin{array}{cc}
        k_i & -F_i \\
        k_iF_i & 1 \\
        \end{array}\right)
        \text{ or } N_i=\left(\begin{array}{cc}
        0 & -1 \\
        k_i & 0 \\
        \end{array}\right)
    $$
    where $k_i=\frac{-1\pm\sqrt{1-4b_ic_i}}{2c_i}\neq 0$. Since any scalar matrix is the identity of M\"{o}bius transformation,
    $$
    \left (\begin{array}{cc}
    n & 0 \\
    0 & n \\
    \end{array}\right)(x)=\frac{nx+0}{0x+n}=x, \forall n\neq 0,
    $$
    we may regard $N_i\in PGL(2,\rr)$ hence the inverse of a matrix, despite a scalar, is simply its adjoint matrix. Simple observation shows
    $$
    \begin{array}{l}
    N_i\in T:=\left\{ 
    \left (\begin{array}{cc}
    a & b \\
    c & d \\
    \end{array}\right)
    \in \text{PGL}(2,\rr): ab+cd=0
    \right\}, \\
    N_i^{-1}\in T':=\left\{ 
    \left (\begin{array}{cc}
    a & b \\
    c & d \\
    \end{array}\right)
    \in \text{PGL}(2,\rr): ac+bd=0
    \right\}, \\
    \end{array}
    $$
    which implies $N_3N_2N_1\in T'$ since $N_4N_3N_2N_1$ is a scalar. Notice that the map
    $$
    f:\,\,\rp \rightarrow \rp;\,\, x \mapsto \frac{2x}{1-x^2}
    $$
    is onto given $f(\tan(\alpha))=\tan(2\alpha)$. Hence $\exists F_3 \in \rp$ such that
    $$
    \frac{2F_3}{1-F_3^2}=\frac{2(ac+bd)}{c^2+d^2-a^2-b^2} \text{ where }
    \left (\begin{array}{cc}
    a & b \\
    c & d \\
    \end{array}\right):=
    \left (\begin{array}{cc}
    k_3 & 0 \\
    0 & 1 \\
    \end{array}\right)N_2N_1.
    $$
    This is equivalent to saying that $\exists F_3 \in \rp$ such that
    $$
    \left (\begin{array}{cc}
    a' & b' \\
    c' & d' \\
    \end{array}\right)
    :=N_3N_2N_1=
    \left (\begin{array}{cc}
    1 & -F_3 \\
    F_3 & 1 \\
    \end{array}\right)
    \left (\begin{array}{cc}
    a & b \\
    c & d \\
    \end{array}\right)
    \in T', \text{ i.e. } a'c'+b'd'=0,
    $$
    so we must have $a'\neq 0$ or $b'\neq 0$. Thus in $PGL(2,\rr)$ we have
    $$
    N_4=\left (\begin{array}{cc}
    d' & -b' \\
    -c' & a' \\
    \end{array}\right)=\left (\begin{array}{cc}
    \frac{d'}{a'} & -\frac{b'}{a'} \\
    -\frac{c'}{a'} & 1 \\
    \end{array}\right) \text{ or }
    N_4=\left (\begin{array}{cc}
    0 & -1 \\
    -\frac{c'}{b'} & 0 \\
    \end{array}\right).
    $$
    i.e. $(k_4,F_4)=(\frac{d'}{a'},\frac{b'}{a'})$ or $(-\frac{c'}{b'},\infty)$. Finally, although $c_4$ was mentioned to be free, we need to make sure $c_4k_4\neq-1$ so that $b_4=-c_4k_4^2-k_4\neq 0$ (recall $k_i=\frac{-1\pm\sqrt{1-4b_ic_i}}{2c_i}$).
\end{example}

\begin{remark}
    For each $i\in\{1,2,3,4\}$, a \textbf{PR} mesh requires either $a_i=e_i=0$ or $b_i=c_i=0$. However, one should realize that the above example is typical. The approach can be adapted to any other case after a minor change.
\end{remark}
 
\subsubsection{HQ: $S_1, S_2$ only admit valid components of \textbf{Case 2}.}

In such a class both $S_1, S_2$ are half-quadratic so each of them only contains one reducible $\Tilde{g}^{(i)}$. The extension diagram of a valid component of $S_1$ leads to either $x_2\in K_1$ or $x_2\in K_3$ but not both. That is why we call it 'half-quadratic'. Please note that we only need to consider the matchings in which $\Tilde{g}^{(1)}, \Tilde{g}^{(3)}$ are reducible or $\Tilde{g}^{(2)}, \Tilde{g}^{(4)}$ are reducible. This is because when reducible polynomials are adjacent, say $\Tilde{g}^{(1)}, \Tilde{g}^{(4)}$, then write the matching as $(\Tilde{g}^{(4)},\Tilde{g}^{(1)},\Tilde{g}^{(2)},\Tilde{g}^{(3)})$, it goes to \textbf{PR + IR} class.

Example \ref{HQ} gave a way to construct pseudo-planar \textbf{HQ} matchings in which $\Tilde{g}^{(1)}, \Tilde{g}^{(3)}$ are reducible. Clearly, after a rotation, it also works for those in which $\Tilde{g}^{(2)}, \Tilde{g}^{(4)}$ are reducible.

\subsubsection{IR: $S_1, S_2$ only admit valid components of \textbf{Case 3}.}

In such a class all $\Tilde{g}^{(i)}$ are irreducible and $d_M$ has linear factors only. This means in the extension diagram of any valid component we have $x_3\in K_1$. So $\Tilde{g}^{(1)}(x_1,t),\Tilde{g}^{(2)}(t,x_3)\in K_1[t]$ determine the same minimal polynomial of $x_2$ on $K_1$ which means they have two common roots in $K_1(x_2)$. Same thing happens on $(\Tilde{g}^{(3)},\Tilde{g}^{(4)})$ and $x_4$. The name 'involutive-rational' therefore comes from here.

Example \ref{IR} showed how to construct pseudo-planar matchings $M$ in which $S_1, S_2$ admit valid components of \textbf{Case 3}. However, if both $S_1, S_2$ are rational-quadratic, we only need to adjust a few coefficients such that $\Tilde{R}^{(1)},\Tilde{R}^{(2)}$ only coincide on their linear factors so $S_i$ does not have valid components of \textbf{Case 4}.

As for matchings in general type, notice that real components of \textbf{Case 3} only exist when $F_1, F_3\in\{0, \infty\}$ (see Proposition \ref{real-sol}). Suppose $W_1, W_2$ are real valid components of $S_1, S_2$ respectively such that $W_1\times_{\{x_1, x_3\}} W_2$ is an infinite set. According to Theorem \ref{flexequi}, we are allowed to merge the extension diagrams
$$
\scalebox{0.75}{\xymatrix{
    K_1(x_2,x_3) & & K_1(x_4,x_3)\\
    K_1(x_2) \ar@{=}[u]^{\Tilde{g}^{(2)}} & K_1(y_2)=K_1(x_3)\ar@{-}[ul] \ar@{-}[ur] & K_1(x_4) \ar@{=}[u]_{\Tilde{g}^{(3)}}\\
    & K_1\ar@{-}[ul]^{\Tilde{g}^{(1)}} \ar@{=}[u] \ar@{-}[ur]_{\Tilde{g}^{(4)}}& \\
}}   
$$
According to the proof of Corollary \ref{n-is-c}, we have $y_2=kx_1$ or $y_2=k/x_1$. After substitutions $y_4=\frac{F_4+x_{1}}{1-F_4x_{1}}$ and $x_3=\frac{y_2-F_2}{1+F_2y_2}$, the relation between $(y_4,x_3)$ should keep the same format like $(y_2,x_1)$. This immediately shows the following restrictions.
\begin{equation}\label{IR-g}
    \{F_2F_4=k=\pm 1\}\cup\{-F_2/F_4=k=\pm 1\}\cup\{-F_2F_4=k=\pm 1\}\cup\{F_2/F_4=k=\pm 1\}.
\end{equation}

\subsubsection{RQ: $S_1, S_2$ admit valid components of \textbf{Case 3} and \textbf{Case 4}.}

In such a class $\Tilde{R}^{(1)},\Tilde{R}^{(2)}$ have the same factorization: an irreducible and quadratic factor times a squared linear factor. This explains the name 'ration-quadratic'. We only need to consider pseudo-planar type because of the following proposition.

\begin{proposition}\label{PQpseu}
    A matching $M$ must be pseudo-planar if, for $i=1,2$, coupling $S_i$ admits a valid component $W_i$ of \textbf{Case 4} such that 
    $$
    (W_1\cap (\rp)^5)\times_{\{x_1, x_3\}} (W_2\cap (\rp)^5)
    $$
    is an infinite set.
\end{proposition}

\begin{proof}
    Clearly, $W_1, W_2$ are real components of $S_1, S_2$ respectively, hence $F_1, F_3\in \{0, \infty\}$ by Proposition \ref{real-sol}. Moreover, due to Theorem \ref{flexequi}, we can merge the extension diagrams
    \begin{equation}\label{PQmerge}
        \scalebox{0.75}{\xymatrix{
        K_1(x_2) \ar@{=}[r] & K_1(x_3) & K_1(x_4) \ar@{=}[l]\\
        & K_1\ar@{-}[ul]^{\Tilde{g}^{(1)}} \ar@{-}[u] \ar@{-}[ur]_{\Tilde{g}^{(4)}}& \\
    }} 
    \end{equation}
    where $K_1(x_2)=K_1(x_3)=K_1(x_4)$ is based on Lemma \ref{samedsc-g}. The same lemma also tells coupling $(\Tilde{g}^{(4)},\Tilde{g}^{(1)},H^{(4)},H^{(1)})$ admits a real component of \textbf{Case 3} or \textbf{Case 4}. And again, by Proposition \ref{real-sol} we have $F_4\in \{0, \infty\}$. Finally, thanks to Theorem \ref{symext} we can use the same argument to conclude $F_2\in \{0, \infty\}$. Thus the whole matching must be pseudo-planar.
\end{proof}

The construction is easy. First, we use Example \ref{IR} to get some candidates so $\Tilde{R}^{(1)},\Tilde{R}^{(2)}$ share at least a linear factor. Then use Corollary \ref{even-term} to get the quadratic factors and make the corresponding coefficients proportional.

\subsubsection{PQ: $S_1, S_2$ only admit valid components of \textbf{Case 4}.}

In such a class we can always get a merged diagram (\ref{PQmerge}) in which
$$
K_1(x_2)=K_1(x_3)=K_1(x_4)=K_1(x_2,x_3,x_4)\text{ and } [K_1(x_2):K_1]=2.
$$
The name 'purely-quadratic' is quite straightforward.

The construction is basically as suggested in Section \ref{Generalap}. However, given $M$ is pseudo-planar (Proposition \ref{PQpseu}), and according to Corollary \ref{even-term}, the polynomial system is reduced to four equations which are derived from proportional coefficients $\alpha_{ij}$. But do not forget to remove the solutions of \textbf{RQ} class when both $S_1, S_2$ are wanted to be rational-quadratic.

\subsubsection{IQ: $S_1, S_2$ only admit valid components of \textbf{Case 5}.}

Similar to \textbf{IR} matchings, $\Tilde{g}^{(1)}, \Tilde{g}^{(2)}$ determine the same minimal polynomial of $x_2$ on $K_1(x_3)$ and $[K_1(x_3),K_1]=2$. So it is called 'involutive-quadratic'. This class belongs to the so-called \textbf{orthodiagonal involutive type} which is well studied by Aikyn in \cite{Alisher}. They gave a systematic construction for such matchings and found the equivalent condition for the valid components being real. However, we have an important result for pseudo-planar \textbf{IQ} matchings.

\begin{proposition}\label{IQpseu}
    In a pseudo-planar matching $M$ with couplings $S_1$ and $S_2$, $S_1$ is involutive-quadratic if and only if $S_2$ is involutive-quadratic.
\end{proposition}

\begin{proof}
    Suppose $S_i$ is involutive-quadratic so its valid component must be of \textbf{Case 5}. From Lemma \ref{samemini} we know that the irreducible factor of $\Tilde{R}^{(i)}$ is quadratic in both $x_1, x_3$ and only contains odd-degree terms. On the contrary, going through pseudo-planar couplings of other classes, the irreducible factors of the resultant, if quadratic, only contain even-degree terms (see Example \ref{HQ} and Corollary \ref{even-term}). Hence the other coupling must be involutive-quadratic as well.
\end{proof}

\subsubsection{Q: $S_1, S_2$ only admit valid components of \textbf{Case 6}.}

This is the irreducible class for $\Tilde{R}^{(1)},\Tilde{R}^{(2)}$, so in order to form a matching, we must have $\Tilde{R}^{(1)}=c\Tilde{R}^{(2)}$ for a constant $c$. We call it 'quartic' simply because all components of the couplings lead to $[K_1(x_3):K_1]=4$.

Again, the construction is as suggested in Section \ref{Generalap} and Example \ref{skeleton} (Reflection) shows the solution exists. 
In addition, it is easy to verify that by relabeling the quads in Example \ref{physical}, e.g. considering couplings $(\Tilde{g}^{(2)},\Tilde{g}^{(3)},H^{(2)},H^{(3)})$ and $(\Tilde{g}^{(4)},\Tilde{g}^{(1)},H^{(4)},H^{(1)})$, one may obtain a non-reflection \textbf{Q} mesh.

Hellmuth Stachel conjectured a long time ago that there do not exist irreducible matchings. In our language, in a flexible mesh which determines a matching $M$, $\text{Res}(\Tilde{g}^{(1)},\Tilde{g}^{(2)}; x_2)$ is irreducible only if $\text{Res}(\Tilde{g}^{(2)},\Tilde{g}^{(3)}; x_3)$ is reducible. Or $[K_1(x_3):K_1]=4\Rightarrow[K_2(x_4):K_2]\leq 2$ for short. The planar version\footnote{All forming panels in Fig.~\ref{mesh} left are planar.} of the conjecture has been proved by Erofeev in \cite{Erofeev}. If the conjecture is also true for non-planar type, the construction for \textbf{Q} matchings can be transformed into other classes after a rotation.

\subsection{Hybrid classes}\label{hybrid}

\subsubsection{PR + IR: $S_1, S_2$ admit valid components of \textbf{Case 1} and \textbf{Case 3} respectively, or vice versa.}

The construction for this class is just a byproduct of Example \ref{construction-PR}. Symmetrically, we assume $S_1$ admits a valid component of \textbf{Case 3} hence $S_2$ is purely-rational. In addition, we may assume $F_1\in \{0, \infty\}$ due to Proposition \ref{real-sol}.

\begin{example}\label{PR+IR}
    We construct matchings $M$ in which $F_1=0$ and the corresponding couplings $S_1, S_2$ admit valid components $W_1, W_2$ of \textbf{Case 3} and \textbf{Case 1} respectively. i.e. $\Tilde{g}^{(1)},\Tilde{g}^{(2)}$ are irreducible and $\Tilde{g}^{(3)},\Tilde{g}^{(4)}$ are reducible. According to Theorem \ref{flexequi}, the merged extension diagram of $W_1$ and $W_2$ is
    $$
    \scalebox{0.75}{\xymatrix{
         & K_1(x_2,y_2) &\\
        K_1(x_2) \ar@{=}[ur]^{g^{(2)}} & K_1(x_3)\ar@{-}[u] & K_1(x_4) \ar@{=}[l]\\
        & K_1\ar@{-}[ul]^{\Tilde{g}^{(1)}} \ar@{=}[u] \ar@{=}[ur]& \\
    }} 
    $$
    Since $F_1=0$, from Corollary \ref{n-is-c} and its proof, we know that in the diagram it is either $y_2=kx_1$ or $y_2=k/x_1$, depending on whether the coefficients of $g^{(1)}, g^{(2)}$ satisfy
    $$
    \left\{\frac{a_1}{a_2}=\frac{b_1}{c_2}=\frac{b_2}{c_1}=\frac{e_2}{e_1}=k\right\} \text{ or } \left\{\frac{a_1}{b_2}=\frac{b_1}{e_2}=\frac{a_2}{c_1}=\frac{c_2}{e_1}=\frac{1}{k}\right\}
    $$
    where $k\neq 0$ is a constant. In either of the above, the relation between $(x_1,x_3)$ is a M\"{o}bius transformation $x_3=\frac{px_1+q}{rx_1+s}$ where $p,q,r,s$ are parameterized by $k$ and $F_2$. The rest construction for $\Tilde{g}^{(3)},\Tilde{g}^{(4)}$ goes the same way as in Theorem \ref{PR} or Example \ref{construction-PR}, i.e. 
    $$
    N_4N_3\left(\begin{array}{cc}
        p & q \\
        r & s \\
        \end{array}\right)
    $$
    must be a scalar matrix.
\end{example}

\subsubsection{HQ + IQ: $S_1, S_2$ admit valid components of \textbf{Case 2} and \textbf{Case 5} respectively, or vice versa.}

The construction is as suggested in Section \ref{Generalap}. Nevertheless, we want to mention that, it is easy to realize one of $S_i$ is half-quadratic and the other is involutive-quadratic hence there is only one $\Tilde{g}^{(i)}$ reducible. We may temporarily assume it is $\Tilde{g}^{(2)}$ so there are valid components $W_1, W_2$ of $S_1, S_2$ respectively such that $W_1\times_{\{x_1, x_3\}} W_2$ is an infinite set. The following diagrams are from the same merged diagram of $W_1$ and $W_2$ but with respect to different parameters $x_1, x_4$ respectively.
$$
\scalebox{0.75}{\xymatrix{
    K_1(x_2,x_3) & & K_1(x_4,x_3)\\
    K_1(x_2) \ar@{=}[u] & K_1(x_3)\ar@{=}[ul] \ar@{=}[l] \ar@{-}[ur] & K_1(x_4) \ar@{-}[u]_{\Tilde{g}^{(3)}}\\
    & K_1\ar@{-}[ul]^{\Tilde{g}^{(1)}} \ar@{-}[u] \ar@{-}[ur]_{\Tilde{g}^{(4)}}& \\
}   \,\,\,\,\,\,\,\,\,
\xymatrix{
    K_4(x_1,x_2) & & K_4(x_2,x_3)\\
    K_4(x_1) \ar@{-}[u]^{\Tilde{g}^{(1)}} & K_4(x_2)\ar@{=}[ur] \ar@{=}[r] \ar@{-}[ul] & K_4(x_3) \ar@{=}[u]\\
    & K_4\ar@{-}[ul]^{\Tilde{g}^{(4)}} \ar@{-}[u] \ar@{-}[ur]_{\Tilde{g}^{(3)}}& \\
}}
$$
According to Definition \ref{casedefn} we have $K_1(x_2)=K_1(x_3)\neq K_1(x_4)$ and $[K_4(x_2):K_4]=[K_4(x_3):K_4]=2$. So by Lemma \ref{samedsc-g}, coupling $(\Tilde{g}^{(4)},\Tilde{g}^{(1)},H^{(4)},H^{(1)})$, like $S_2$, also admits a valid component of \textbf{Case 5}. This means no matter how we label the polynomials $\Tilde{g}^{(i)}$, an \textbf{HQ + IQ} matching is always an \textbf{HQ + IQ} matching. Therefore, without loss of generality, we can always assume $\Tilde{g}^{(2)}$ is reducible and both couplings $S_2$ and $(\Tilde{g}^{(4)},\Tilde{g}^{(1)},H^{(4)},H^{(1)})$ are involutive-quadratic. Hence by Theorem \ref{main-ps} or \ref{main-g} we have 
$$
\left\{\frac{a_4}{b_4}=\frac{c_4}{e_4}=\frac{a_1}{c_1}=\frac{b_1}{e_1}(=-1), \frac{a_3}{b_3}=\frac{c_3}{e_3}=\frac{a_4}{c_4}=\frac{b_4}{e_4}(=-1)\right\}.
$$
Adding this to the construction process will considerably reduce the number of parameters.
The solution exists according to the following example.

\begin{example}
    We first set 
    $$
    b_1=-a_1\neq 0, e_1=-c_1\neq 0, a_2=e_2=0, b_2=-c_2-1, b_2c_2\neq 0, F_1=1, F_2=0
    $$
    plus the inequalities in Definition \ref{dfnpoly}. Simple calculation shows $\Tilde{R}^{(1)}$ has an irreducible factor
    $$
    x_1x_3^2-4(a_1x_1^2+c_1)x_3^2-x_1
    $$
    which is also wanted to be a factor of $\Tilde{R}^{(2)}$. So finally, let $F_3=F_4=0$ and use Lemma \ref{samemini} to construct $g^{(3)}$ and $g^{(4)}$.
\end{example}

\subsubsection{HQ + PQ: $S_1, S_2$ admit valid components of \textbf{Case 2} and \textbf{Case 4} respectively, or vice versa.}

As we did for the last class, a similar argument can show that, no matter how we label the polynomials $\Tilde{g}^{(i)}$, an \textbf{HQ + PQ} matching is always an \textbf{HQ + PQ} matching. Without loss of generality, we can always assume $\Tilde{g}^{(2)}$ is reducible and both couplings $S_2$ and $(\Tilde{g}^{(4)},\Tilde{g}^{(1)},H^{(4)},H^{(1)})$ admit valid components of \textbf{Case 4}. And due to Proposition \ref{real-sol}, we may further assume $F_3, F_4\in \{0,\infty\}$. We first see an example of pseudo-planar type.

\begin{example}\label{HQ+PQ}
    We construct \textbf{HQ + PQ} matchings in which all $F_i=0$ and $g^{(2)}$ is reducible. Since $S_2$ admits a valid component of \textbf{Case 4}, from Theorem \ref{main-ps} $g^{(3)}$ and $g^{(4)}$ could be constructed by the first one of systems (\ref{eq-c4}). Hence by Corollary \ref{even-term} we get an irreducible factor of $\Tilde{R}^{(2)}$ explicitly
    $$
    (\alpha_{22} x_1^2 + \alpha_{02}) x_3^2 + \alpha_{11} x_1 x_3 + (\alpha_{20} x_1^2+ \alpha_{00}).
    $$
    If necessary, we need to adjust a few coefficients to make sure $\alpha_{11} \neq 0$. On the other hand, by Lemma \ref{rational}, we arbitrarily set $g^{(2)}$ reducible by letting $a_2=e_2=0$ or $b_2=c_2=0$ so $g^{(2)}$ has a factor in the form $y_2-kx_2$ or $x_2y_2-k$. We first consider the case $(y_2-kx_2)|g^{(2)}$. In the merged extension diagram of valid components of $S_1, S_2$, the minimal polynomial of $x_3$ on $K_1$ can be induced by substitutions
    $$
    g^{(1)}(x_1,y_1)=g^{(1)}(x_1,x_2)=g^{(1)}\left(x_1,\frac{y_2}{k}\right)=g^{(1)}\left(x_1,\frac{x_3}{k}\right)
    $$
    so $\Tilde{R}^{(1)}$ has an irreducible factor
    $$
    (a_1 x_1^2 + c_1 ) x_3^2 + k x_1 x_3 + k^2(b_1 x_1^2+ e_1)
    $$
    which should be the same factor of $\Tilde{R}^{(2)}$, i.e. $g^{(1)}$ is determined by 
    $$
    \frac{a_1}{\alpha_{22}}=\frac{c_1}{\alpha_{02}}=\frac{k}{\alpha_{11}}=\frac{k^2 b_1}{\alpha_{20}}=\frac{k^2e_1}{\alpha_{00}}
    $$
    where $\alpha_{11}\neq 0$, as required, is mandatory. The case for $(x_2y_2-k)|g^{(2)}$ is similar.
\end{example}

Clearly, there is no difficulty in applying the above method to construct all pseudo-planar \textbf{HQ + PQ} matchings since all polynomials remain in the same form. In fact, Example \ref{HQ+PQ} is even typical for general type. Notice that we have restricted $F_3, F_4\in \{0,\infty\}$ so for general type we only allow $F_1, F_2\notin \{0,\infty\}$. Likewise, after $g^{(1)}(x_1,y_1)$ going through three steps of substitutions
$$
\left(y_1=\frac{x_2+F_1}{1-F_1x_2}\right) \rightarrow \left(x_2=\frac{y_2}{k} \text{ or } x_2=\frac{k}{y_2} \right) \rightarrow \left(y_2=\frac{x_3+F_2}{1-F_2x_3}\right),
$$
we still get an irreducible polynomial in the form 
$$
(\alpha_{22} x_1^2 + \alpha_{02}) x_3^2 + \alpha_{11} x_1 x_3 + (\alpha_{20} x_1^2+ \alpha_{00})
$$
that can match the irreducible factor of $\Tilde{R}^{(2)}$ without odd-degree terms. This leads to the same consequence (\ref{IR-g}) as in \textbf{IR} class (replace $F_4$ by $F_1$). The rest work goes the same as in Example \ref{HQ+PQ}.

\subsubsection{PQ + IQ: $S_1, S_2$ admit valid components of \textbf{Case 4} and \textbf{Case 5} respectively, or vice versa.}

The construction is as suggested in Section \ref{Generalap}. The solution exists according to the following example.

\begin{example}
Suppose $S_1, S_2$ admit valid components of \textbf{Case 4} and \textbf{Case 5} respectively hence we may assume $F_1=0$ due to Proposition \ref{real-sol}. Let $F_2=0$ and according to Corollary \ref{even-term} we have an irreducible factor of $\Tilde{R}^{(1)}$
$$
(\alpha_{22} x_1^2 + \alpha_{02}) x_3^2 + \alpha_{11} x_1 x_3 + (\alpha_{20} x_1^2+ \alpha_{00}).
$$
On the other hand, set $F_3=\pm 1, F_4=0$ we get an irreducible factor of $\Tilde{R}^{(2)}$ by Lemma \ref{samemini-g}
$$
4a_4(a_3x_3^2 + c_3) x_1^2 +x_3 x_1 + 4b_4(a_3x_3^2+c_3).
$$
Thus the coefficients of $g^{(1)},g^{(2)}$ can be determined by equation (\ref{repara}) and
$$
\frac{4a_3a_4}{\alpha_{22}}=\frac{4a_3b_4}{\alpha_{02}}=\frac{1}{\alpha_{11}}=\frac{4c_3a_4}{\alpha_{20}}=\frac{4c_3b_4}{\alpha_{00}}.
$$
\end{example}

We call back to Example \ref{physical} for a physical model.

\section{Conclusion and future work}

In this article, we developed an algebraic approach to classify the flexible Kokotsakis polyhedra without (anti)deltoids, also as known as non-singular $3\times 3$ flexible meshes. Among all 11 different classes, most of them are provided with systematic constructions, and for the rest, partial examples are given to demonstrate no empty class. 

Obviously, the future work is to continue on singular ones which contain (anti)deltoids. Based on the results we obtained so far, our algebraic method still works but in an asymmetric way, unlike non-singular cases. On the other hand, the singularity may occur in many different ways hence a lot of reduction efforts are needed to provide a classification as concise as possible.

\section*{Acknowledgments}

The authors are grateful to Helmut Pottmann for his comprehensive introduction to the peer works and to Dmitry A.~Lyakhov for his feedback on this manuscript. This work has been supported by KAUST baseline funding.

\newpage

\bibliographystyle{unsrt}
\bibliography{arXiv}

\newpage
\appendix
\section{Definition of the dihedral angles}\label{dfnang}
Following Fig.~\ref{mesh} right, we need to give a proper definition of dihedral angles $\alpha_i'$. To do that, we are about to relocate the starting points of all vectors to the center of a unit sphere. Let us consider a planar mesh that is partially depicted in Fig.~\ref{illus} left.
\begin{figure}[t]
    \centering
    \begin{overpic}[width=0.6\textwidth]{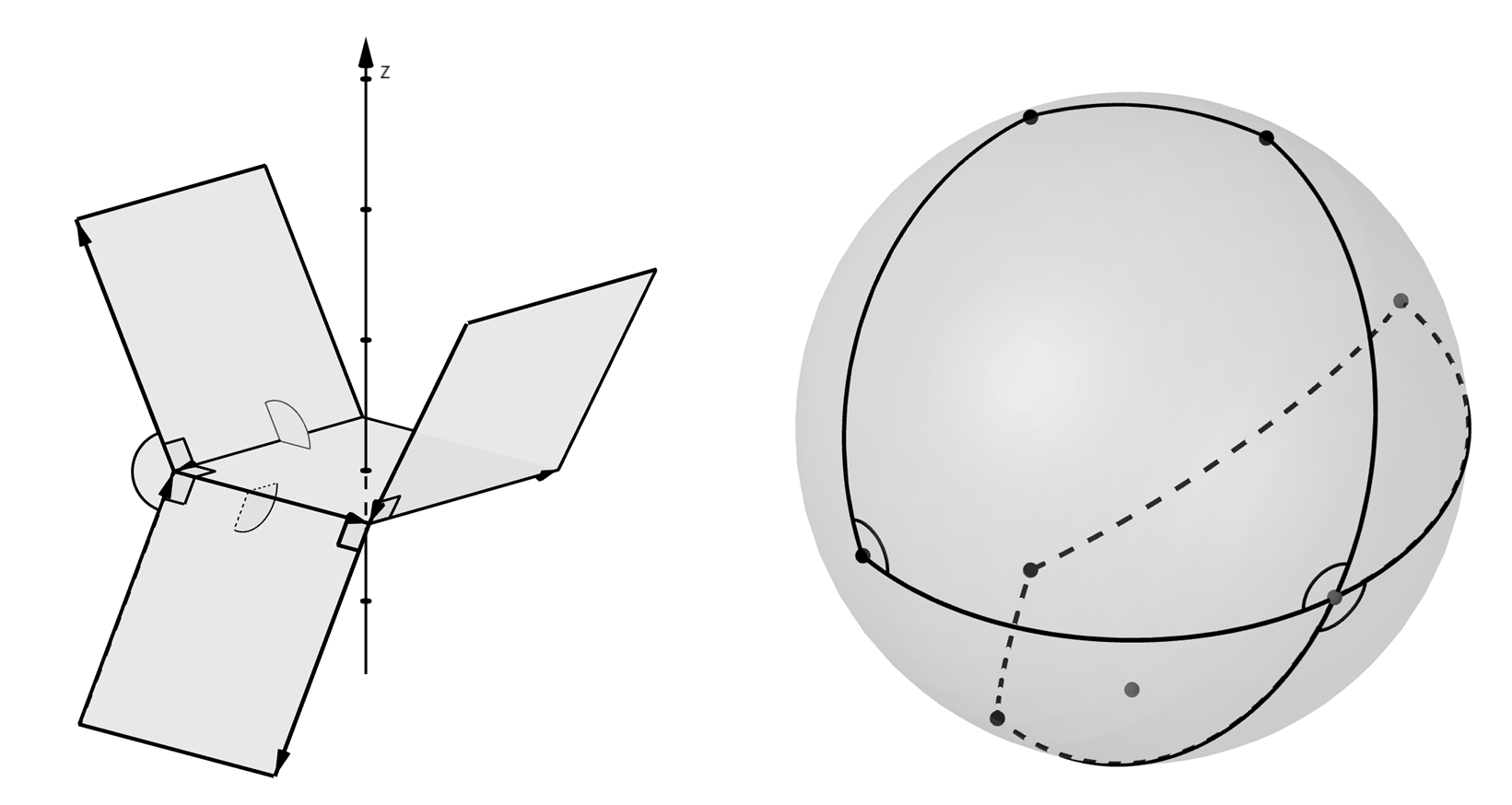}
    \put(14.5,21.2){\footnotesize\contour{white}{$-\lambda_1'$}}
    \put(19,15){\footnotesize\contour{white}{$\gamma_2'$}}
    \put(11,16.5){\footnotesize\contour{white}{$\delta_1'$}}
    \put(11,25.3){\footnotesize\contour{white}{$\gamma_1'$}}
    \put(27,21){\footnotesize\contour{white}{$\delta_2'$}} 
    \put(5,21){\footnotesize\contour{white}{$\mu_1'$}}
    \put(18.5,27){\footnotesize\contour{white}{$\alpha_1$}}
    \put(16.5,15.6){\footnotesize\contour{white}{$|$}}
    \put(16,13){\footnotesize\contour{white}{$\alpha_2$}}
    \put(75,12.5){\footnotesize\contour{white}{$\lambda_1$}}
    \put(97.5,24){\footnotesize\contour{white}{$\lambda_2$}}
    \put(76,23){\footnotesize\contour{white}{$\delta_2$}}
    \put(73,0){\footnotesize\contour{white}{$\gamma_2$}}
    \put(62,8.5){\footnotesize\contour{white}{$\mu_2$}} 
    \put(75,43){\footnotesize\contour{white}{$\mu_1$}}
    \put(60,35){\footnotesize\contour{white}{$\gamma_1$}}    
    \put(85.5,35){\footnotesize\contour{white}{$\delta_1$}}
    \put(59,18.2){\footnotesize\contour{white}{$\alpha_1$}}
    \put(90,10.4){\footnotesize\contour{white}{$\alpha_2$}}
    \put(84,16.4){\footnotesize\contour{white}{$\beta_1$}}
\end{overpic}
    \caption{Notice that $(\lambda_i, \gamma_i, \mu_i, \delta_i)$ and $(\lambda_i', \gamma_i', \mu_i', \delta_i')$ are complementary to $\pi$ respectively.}
    \label{illus}
\end{figure} 

\begin{figure}[t]
    \centering
    \begin{overpic}[width=0.5\textwidth]{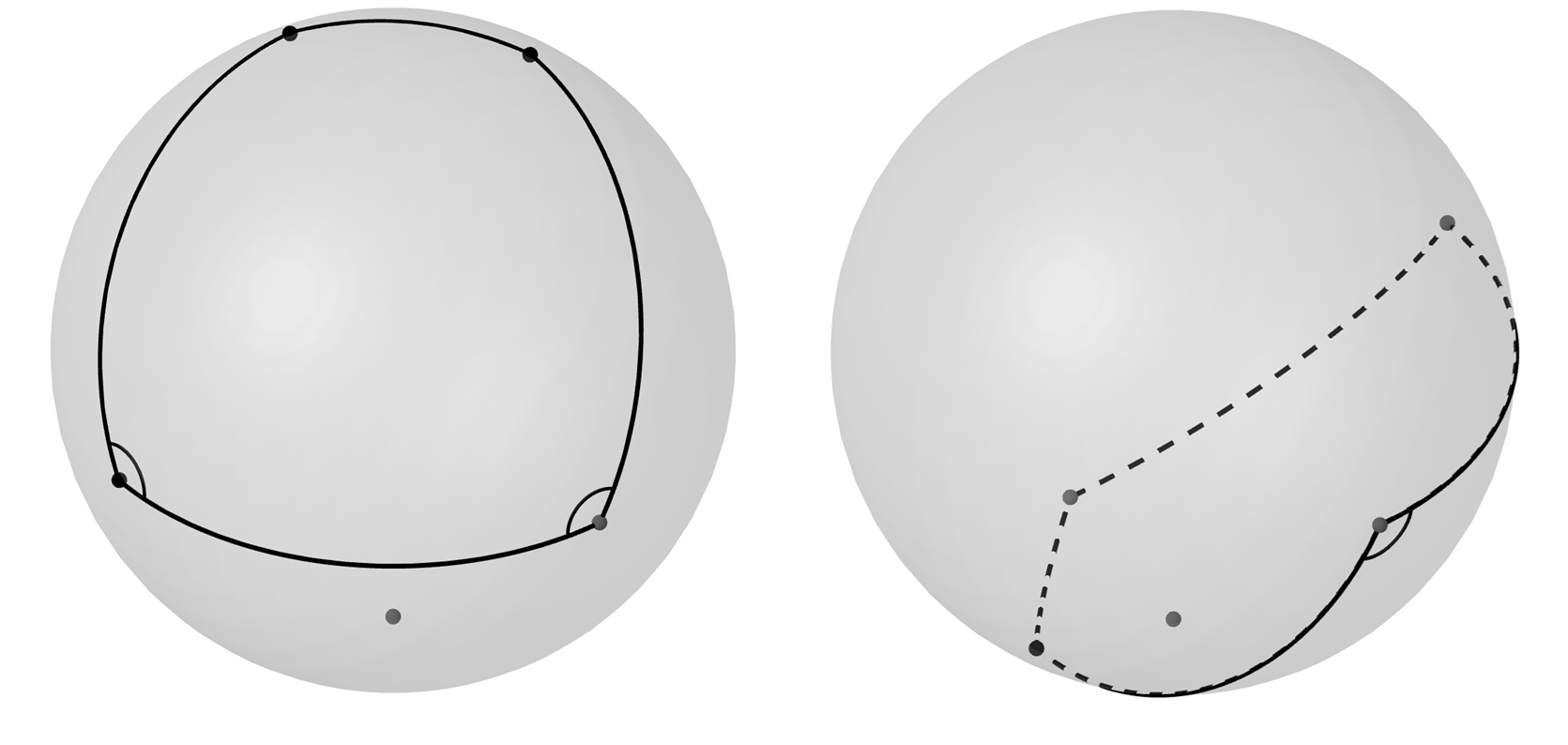}
    \put(20,13){\footnotesize\contour{white}{$\lambda_1$}}
    \put(25,43){\footnotesize\contour{white}{$\mu_1$}}
    \put(10,35){\footnotesize\contour{white}{$\gamma_1$}}    
    \put(35.5,35){\footnotesize\contour{white}{$\delta_1$}}
    \put(9,18.2){\footnotesize\contour{white}{$\alpha_1$}}
    \put(33,16.4){\footnotesize\contour{white}{$\beta_1$}}
    \put(21,7.5){\footnotesize\contour{white}{$S$}}
    \put(97.5,24){\footnotesize\contour{white}{$\lambda_2$}}
    \put(76,23){\footnotesize\contour{white}{$\delta_2$}}
    \put(73,0){\footnotesize\contour{white}{$\gamma_2$}}
    \put(62,9.5){\footnotesize\contour{white}{$\mu_2$}} 
    \put(89,11){\footnotesize\contour{white}{$\alpha_2$}}
    \put(71,7.5){\footnotesize\contour{white}{$S$}}
\end{overpic}
    \caption{This figure is the decomposition of Fig.~\ref{illus} right, $S$ is the south pole.}
    \label{split}
\end{figure}
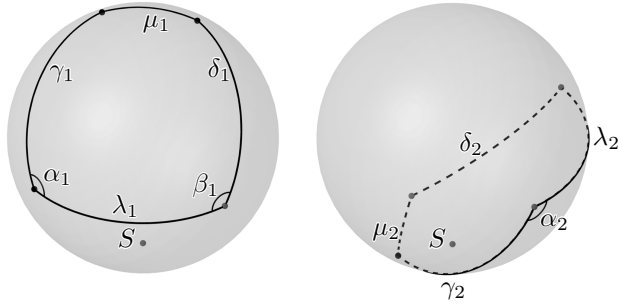 

In Fig.~\ref{illus} right, we take the common orientation of the sphere where the surface normals are pointing outwards. And in Fig.~\ref{split} left, the oriented spherical quad $Q_1: (\lambda_1, \delta_1, \mu_1, \gamma_1)$ determines its interior---the area which does not contain south pole $S$, hence the dihedral angles $\alpha_1, \beta_1 \in [0, 2\pi)$ are well defined by the interior angles of $Q_1$. However, if we adopt the same orientation on Fig.~\ref{split} right for $Q_2$, its interior would look strange---also the area which does not contain $S$. Therefore, in order to easily represent the relation between angles, we reverse the orientation and set $Q_2: (\lambda_2, \gamma_2, \mu_2, \delta_2)$, the dihedral angles $\alpha_2, \beta_2 \in [0, 2\pi)$ are defined in the same way through the interior angles of $Q_2$.

In general, once we relocated all vectors to the sphere, the spherical quads are oriented as follows:
$$
Q_i=
\left\{\begin{array}{l}
    (\lambda_i, \delta_i, \mu_i, \gamma_i),\,\, i\in \{1,3\}, \\
    (\lambda_i, \gamma_i, \mu_i, \delta_i),\,\, i\in \{2,4\}. \\
\end{array}\right.
$$
The diheral angles $\alpha_i, \beta_i$ are hence defined by the corresponding interior angles of $Q_i$ (notice that $\beta_i=\alpha_{i+1}$). 

When it comes to non-planar meshes (Fig.~\ref{sqmesh}), we keep the same orientation of the sphere and set $\boldsymbol{v_i}$ be the surface normal at the common vertex of $Q_i$ and $Q_{i+1}$. Take Fig.~\ref{squads} as an example, a neighborhood of $Q_1 \cap Q_2$ can be stereographically projected (angle preserving) to its tangent plane (with the same orientation, see Fig.~\ref{sproj}).

\begin{figure}[t]
    \centering
    \begin{overpic}[width=0.25\textwidth]{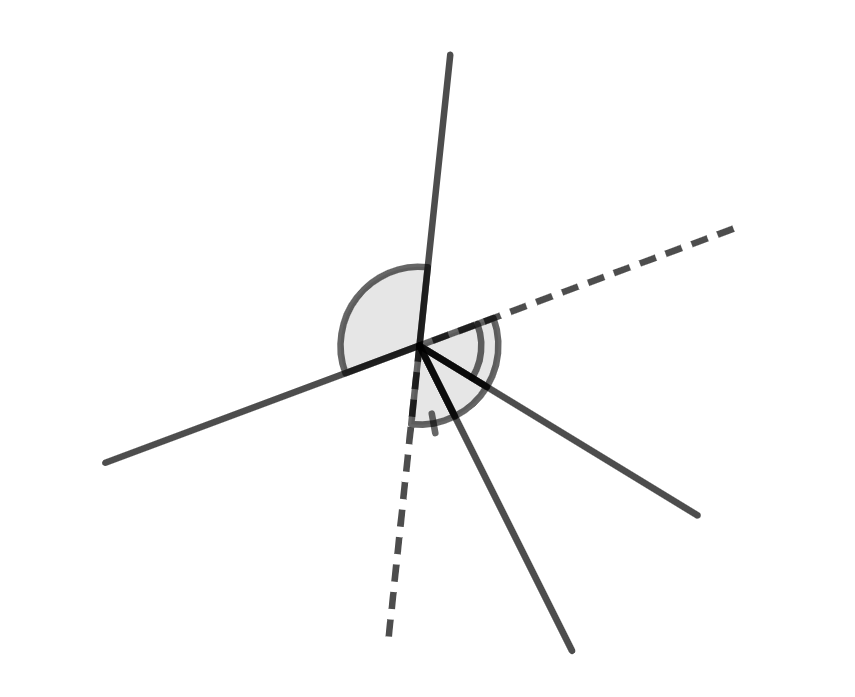}
    \put(15,22){\footnotesize\contour{white}{$\lambda_1$}}
    \put(82,22){\footnotesize\contour{white}{$\lambda_2$}}
    \put(82,48){\footnotesize\contour{white}{$-\lambda_1$}}
    \put(32,10){\footnotesize\contour{white}{$-\delta_1$}}
    \put(62,38){\footnotesize\contour{white}{$\tau_1$}}    
    \put(43,65){\footnotesize\contour{white}{$\delta_1$}}
    \put(55.5,29){\footnotesize\contour{white}{$\alpha_2$}}
    \put(32,46){\footnotesize\contour{white}{$\beta_1$}}
    \put(56,8){\footnotesize\contour{white}{$\gamma_2$}}
    \put(49,25){\footnotesize\contour{white}{$\zeta_1$}}
\end{overpic}
    \caption{Stereographic projection of the neighborhood near the '$\ast$' in Fig.~\ref{squads} left. $\boldsymbol{v_1}$ at the intersecting point is a vector perpendicular to the paper and pointing towards the reader. $-\lambda_1$ and $-\delta_1$, respectively, are the spherical reflections of $\lambda_1$ and $\delta_1$ respect to $\boldsymbol{v_1}$. }
    \label{sproj}
\end{figure}
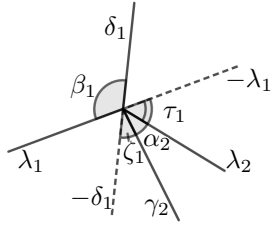 

Generally, we set $-\lambda_i$ and $-\delta_i$, respectively, are the spherical reflections of $\lambda_i$ and $\delta_i$ respect to $\boldsymbol{v_i}$. Let $T(v_i)$ be the oriented tangent plane with surface normal $\boldsymbol{v_i}$ and introduce the operator
$$
\left\{\begin{array}{l}
     <\cdot,\cdot>_{\boldsymbol{v_i}}: T(\boldsymbol{v_i})^2-\{\boldsymbol{0}\} \rightarrow [0,2\pi);  \\
     <\boldsymbol{x},\boldsymbol{y}>_{\boldsymbol{v_i}}=\text{The counterclockwise rotating angle from $\boldsymbol{x}$ to $\boldsymbol{y}$.} 
\end{array}\right.
$$
The new operator is compatible with existing angles through the following:
$$
\alpha_i=
\left\{\begin{array}{l}
    <\gamma_i,\lambda_i>_{\boldsymbol{-v_{i-1}}},\,\, i\in \{1,3\}, \\
    <\gamma_i,\lambda_i>_{\boldsymbol{v_{i-1}}},\,\, i\in \{2,4\}, \\
\end{array}\right.
\beta_i=
\left\{\begin{array}{l}
    <\delta_i,\lambda_i>_{\boldsymbol{v_i}},\,\, i\in \{1,3\}, \\
    <\delta_i,\lambda_i>_{\boldsymbol{-v_i}},\,\, i\in \{2,4\}. \\
\end{array}\right.
$$
Thereafter, by setting
$$
\tau_i=
\left\{\begin{array}{l}
    <\lambda_{i+1},-\lambda_i>_{\boldsymbol{v_i}},\,\, i\in \{1,3\}, \\
    <\lambda_{i+1},-\lambda_i>_{\boldsymbol{-v_i}},\,\, i\in \{2,4\}, \\
\end{array}\right.
\zeta_i=
\left\{\begin{array}{l}
    <-\delta_i,\gamma_{i+1}>_{\boldsymbol{v_i}},\,\, i\in \{1,3\}, \\
    <-\delta_i,\gamma_{i+1}>_{\boldsymbol{-v_i}},\,\, i\in \{2,4\}. \\
\end{array}\right.
$$
we have $\beta_i\equiv \alpha_{i+1}+\tau_i+\zeta_i \mod 2\pi$ for all $i\in \{1,2,3,4\}$. For instance,
$$
\beta_2=<\delta_2,\lambda_2>_{\boldsymbol{-v_2}}=<-\delta_2,-\lambda_2>_{\boldsymbol{-v_2}}\equiv\zeta_2+\alpha_3+\tau_2 \mod 2\pi.
$$

\section{Technical proofs}\label{techpf}

\begin{proof}(Lemma \ref{rational})
    Firstly, it is easy to check that when $a_i=e_i=0$ or $b_i=c_i=0$, $g^{(i)}$ can be factorized as stated and $kk'\neq 0$ is directly from the non-singularity of $g^{(i)}$.

    Conversely, if neither $a_i=e_i=0$ nor $b_i=c_i=0$, $g^{(i)}$ must be irreducible. We prove this by contradiction. Suppose $g^{(i)}$ is reducible. Definition \ref{dfnsingular} implies $g^{(i)}$ has no factors in $\cc[x_i]$ or $\cc[y_i]$. So any irreducible factor of $g^{(i)}$ must be in the form 
    $$
    px_iy_i+qx_i+ry_i+s,  p\neq 0 \text{ or } q\neq 0.
    $$
    This means, by regarding $x_i$ as the unknown and $y_i$ as a parameter, $g^{(i)}=0$ admits a rational solution $x_i=x_i(y_i)\in \cc(y_i)$. In other words, the discriminant
    $$
    -4a_ic_iy_i^4+(1-4a_ie_i-4b_ic_i)y_i^2-4b_ie_i
    $$
    is a square in $\cc[y_i]$. With this information, we can claim that $a_ic_i\neq 0$. Otherwise, we must have $(1-4a_ie_i-4b_ic_i)=0$ or $b_ie_i=0$. This first one fails due to the inequalities of Definition \ref{dfnpoly} and the second is also impossible given $g^{(i)}$ is non-singular and neither $a_i=e_i=0$ nor $b_i=c_i=0$. So the discriminant can be factorized into $-4a_ic_i(y_i^2-k)(y_i^2-k')$ where $k=k'$ since it is a square. Consequently, 
    $$
    -4a_ic_iy_i^2+(1-4a_ie_i-4b_ic_i)y_i-4b_ie_i=-4a_ic_i(y_i-k)(y_i-k')
    $$
    is a square in $\cc[y_i]$ hence has a zero discriminant
    $$
    (1-4a_ie_i-4b_ic_i)^2-64a_ib_ic_ie_i=0,
    $$
    contradicts to the inequalities in Definition \ref{dfnpoly}.
\end{proof}

\begin{proof}(Corollary \ref{r_ij})
    We first assume $\Tilde{g}^{(i)}$ is reducible. By Lemma \ref{rational} and Proposition \ref{2g_i} we know $\Tilde{f}^{(i)}$ is in first degree of $x_i$ and $x_{i+1}$ so $\Tilde{f}^{(i)}=0$ gives a rational relation between $x_i$ and $x_{i+1}$. In other words, the factorization of $\text{Res}(\Tilde{f}^{(i)},f^{(i+1)}; x_{i+1})$ induces the factorization of $f^{(i+1)}$ hence must be irreducible. The rest proof when $g^{(i+1)}$ is reducible will be the same.
\end{proof}

\begin{proof}(Corollary \ref{noncons})
    Suppose there is an $r(y_{i+1})\in \cc[y_{i+1}]-\cc$ which divides $R_i(x_i,y_{i+1})$. So we have $R_i(x_i,c)=0$ for some constant $c\in \cc$. According to the algebraic meaning of the resultant, $\Tilde{g}^{(i)}(x_i,x_{i+1})$ and $g^{(i+1)}(x_{i+1},c)$, as polynomials in $x_{i+1}$, share a common root almost everywhere for $x_i\in \cc$.\footnote{We need to exclude the values for $x_i$ such that the degree of $x_{i+1}$ in $\Tilde{g}^{(i)}$ degenerates.} Since $\Tilde{g}^{(i)}(x_i,x_{i+1})$ has no factors in $\cc[x_{i+1}]$ (Proposition \ref{2g_i}), by solving $\Tilde{g}^{(i)}=0$, $x_{i+1}$  takes infinitely many values while $x_i$ changes continuously. So $g^{(i+1)}(x_{i+1},c)$ must be a zero polynomial which contradicts to its non-singularity. Similarly, $R_i(x_i,y_{i+1})$ has no factors in $\cc[x_i]$ either.
\end{proof} 

\begin{proof}(Proposition \ref{mani})
    Consider the equivalent forms of $S$ in Proposition \ref{eqcoup}. We first assume none of $g^{(i)}$ is reducible. In $Z(S)$, the values of $y_i$ and $x_{i+1}$ uniquely determine each other through $H^{(i)}=0$. Given $g^{(i)}$ non-singular and $\forall c\in \cp$, $g^{(i)}(x_i,c)=g^{(i)}(c,y_i)=0$ always have finitely many solutions in $x_i$ and $y_i$ respectively. This immediately shows $Z(S)$ is locally a 1-dimensional curve except on
    $$
    W_0:=\bigcup_{i\in \{1,2\} \atop t\in \{x,y\}} Z\left(S+\left(\frac{\partial g^{(i)}}{\partial t_i}\right)\right)
    $$
    which is a finite set given that $g^{(i)}, \frac{\partial g^{(i)}}{\partial t_i}$ are coprime. In particular, $Z(S)-W_0$ is locally a function of any variable of $\{x_1,y_1,x_2,y_2,x_3\}$ since none of the partial derivatives vanishes. As for the component in Proposition \ref{eqcoup}, $W=Z(S+(r_k))$ where $r_k$ is an irreducible factor of $R_1$ and $Z(r_k)\subset (\cp)^2$ is a curve in $(x_1,y_2)$. Since $r_k|R_1$, $R_1$ vanishes on $Z(r_k)$ and therefore $\{\Tilde{g}^{(1)}=g^{(2)}=0\}$ always has a solution in $x_2$ for all $(x_1,y_2)\in Z(r_k)$. This means the projection of $W$ in $(\cp)^2$ is a curve hence, as a subset of $Z(S)-W_0$, $W-W_0$ is locally a 1-dimensional curve and a function of any variable of $\{x_1,x_2,x_3,y_1,y_2\}$. Finally, if one of $g^{(i)}$ is reducible, apply the above argument directly on $W$ (rather than $Z(S)$) to get the same result.
\end{proof}

\begin{proof}(Theorem \ref{flexequi})
    By Proposition \ref{eqcoup} we know $Z(M)\cong Z(S_1)\times_{\{x_1, x_3\}} Z(S_2)$. Clearly, if we project $Z(S_1)$ into $(x_1,x_3)$-plane we get $Z(\text{Res}(\Tilde{g}^{(1)},\Tilde{g}^{(2)};x_2))$. For each $(x_1,x_3)\in Z(\text{Res}(\Tilde{g}^{(1)},\Tilde{g}^{(2)};x_2))$, there are at most two pre-images in $Z(S_1)$ according to Proposition \ref{2g_i}.  So $Z(M)$ is an infinite set if and only if $Z(\text{Res}(\Tilde{g}^{(1)},\Tilde{g}^{(2)};x_2))\cap Z(\text{Res}(\Tilde{g}^{(3)},\Tilde{g}^{(4)};x_4))$ is an infinite set if and only if
    $$
    \gcd(\text{Res}(\Tilde{g}^{(1)},\Tilde{g}^{(2)};x_2),\text{Res}(\Tilde{g}^{(3)},\Tilde{g}^{(4)};x_4))\neq 1.
    $$
    On the other hand, $Z(M)$ is an infinite set if and only if there exist components $W_1, W_2$ of $S_1, S_2$ respectively such that $W_1\times_{\{x_1, x_3\}} W_2$ is an infinite set. In the language of field extensions, it is equivalent to saying that $x_3$ is the same algebraic element on $K_1$ in the extension diagrams of $W_1$ and $W_2$, i.e. the minimal polynomials of $x_3$ on $K_1$ are identical.
\end{proof}

\begin{proof}(Lemma \ref{n-is})
    We only prove (1) since (2) is quite similar. Symbolic computation can show
    $$
    \frac{a_1}{a_2}=\frac{b_1}{c_2}=\frac{b_2}{c_1}=\frac{e_2}{e_1}=k \Rightarrow (x_3-kx_1)^2|R(x_1,x_3) \Rightarrow (x_3-kx_1)|R(x_1,x_3).
    $$
    Conversely, when $(x_3-kx_1)|R(x_1,x_3)$, $S$ admits a component
    $$
    W=Z(g^{(1)}, g^{(2)}, x_3-kx_1, H^{(1)}, H^{(2)})
    $$
    with the extension diagram
    $$
    \scalebox{0.75}{\xymatrix{
                & K_1(x_2,x_3) & \\
                K_1(x_2) \ar@{=}[ur]^{g^{(2)}} & & K_1(x_3)\ar@{-}[ul]\\
                & K_1\ar@{-}[ul]^{g^{(1)}} \ar@{=}[ur]_{x_3-kx_1}& \\
            }}   
    $$
    in which $x_3=kx_1$ and $x_2\notin K_1$ due to the irreducibility of $g^{(1)}$.  Thus
    $$
    g^{(1)}(x_1,t)=(a_1 x_1^2 + c_1 ) t^2 + x_1 t + (b_1 x_1^2+ e_1),
    $$
    $$
    g^{(2)}(t,x_3)=(a_2 k^2 x_1^2 + b_2 ) t^2 + k x_1 t + (c_2 k^2 x_1^2+ e_2)
    $$
    both determine the minimal polynomial of $x_2$ on $K_1$ so the coefficients must be proportional:
    $$
    \frac{a_1}{a_2}=\frac{b_1}{c_2}=\frac{b_2}{c_1}=\frac{e_2}{e_1}=k.
    $$
\end{proof}

\begin{proof}(Lemma \ref{samemini})
    First, we assume system (\ref{eq-c5}) holds so we can claim that all coefficients are nonzero. This is because zeros always come in pairs in the system, which is against the non-singularity of $g^{(i)}$. Hence by Lemma \ref{rational}, $g^{(1)}, g^{(2)}$ must be irreducible. Set 
    $$
    b_1=ka_1,\,\, e_1=kc_1,\,\,c_2=ka_2,\,\, e_2=kb_2,
    $$
    symbolic computation can show
    $$
    \text{Res}(g^{(1)},g^{(2)};x_2)=k(a_2 x_1 x_3^2 - (a_1 x_1^2 + c_1) x_3 + b_2 x_1)^2.
    $$
    In addition, $a_2 x_1 x_3^2 - (a_1 x_1^2 + c_1) x_3 + b_2 x_1$ must be irreducible. Otherwise, according to Corollary \ref{noncons}, there would be an irreducible factor $r_k$ which is linear in $x_3$ and further leads to a component of \textbf{Case 3} in Definition \ref{casedefn}. Apply Corollary \ref{n-is-c} we have one of systems (\ref{eq-c3}) holds. Either of them contradicts to the inequality of system (\ref{eq-c5}). Therefore $S$ only has one component
    $$
    W=Z(g^{(1)}, g^{(2)}, a_2 x_1 x_3^2 - (a_1 x_1^2 + c_1) x_3 + b_2 x_1, H^{(1)}, H^{(2)})
    $$
    which belongs to \textbf{Case 4} or \textbf{Case 5} but not \textbf{Case 3}. Similarly, $W$ must be of \textbf{Case 5}. Otherwise, by Lemma \ref{samedsc}, we have 
    $$
    \frac{a_1 c_1}{a_2 b_2}=\frac{1- 4 a_1 e_1- 4 b_1 c_1}{1 - 4 a_2 e_2- 4 b_2 c_2}=\frac{b_1 e_1}{c_2 e_2}=\frac{1- 8k a_1 c_1}{1 - 8k a_2 b_2}=\frac{1- 8k a_1 c_1+8k a_1 c_1}{1 - 8k a_2 b_2+8k a_2 b_2}=1.
    $$
    Contradicts to our previous discussion considering Corollary \ref{n-is-c}.

    Conversely, suppose $S$ has a component $W$ of \textbf{Case 5} and consider its extension diagram. By Definition \ref{casedefn}, we have $x_3\notin K_1(x_2)$ which implies $x_2\notin K_1(x_3)$ given $[K_1(x_3):K_1]=[K_1(x_2):K_1]=2$. This means both $g^{(1)}$ and $g^{(2)}$ determine the minimal polynomial of $x_2$ on $K_1(x_3)$ hence the coefficients should be proportional, which brings us back to the left one of equation (\ref{eq-coeff}). i.e.
    \begin{equation}\label{eqsamemini}
        \left\{\begin{array}{l}
                            a_2 x_1 x_3^2 - (a_1 x_1^2 + c_1) x_3 + b_2 x_1 = 0, \\
                            c_2 x_1 x_3^2 - (b_1 x_1^2 + e_1) x_3 + e_2 x_1 = 0. \\ 
                \end{array}\right.
    \end{equation}
    Since $[K_1(x_3):K_1]=2$, the above two equations both determine the minimal polynomial of $x_3$ on $K_1$ so
    $$
    \frac{a_1}{b_1}=\frac{c_1}{e_1}=\frac{a_2}{c_2}=\frac{b_2}{e_2}.
    $$
    Finally, the irreducibility of equation (\ref{eqsamemini}) implies $\frac{a_1}{a_2}\neq\frac{b_2}{c_1}$ since
    $$
    \frac{a_1}{a_2}=\frac{b_2}{c_1}\Rightarrow a_2 x_1 x_3^2 - (a_1 x_1^2 + c_1) x_3 + b_2 x_1=\frac{1}{a_2}(a_2x_1x_3-c_1)(a_2x_3-a_1x_1).
    $$
\end{proof}

\begin{proof}(Lemma \ref{samedsc-g'})
    We first assume $S$ is equimodular. By Lemma \ref{samedsc-g}, we only need to focus on the equivalent condition for '$\exists f\in K_2$ such that $f^2=\frac{\Delta_1}{\Delta_2}$' where $\Delta_i$ are given by equation (\ref{d2}).
    
    We can claim that $a_1b_1c_1e_1a_2b_2c_2e_2\neq 0$. According to $f^2=\frac{\Delta_1}{\Delta_2}$, $\Delta_1$ and $\Delta_2$ are equal after removing their square factors. In the meanwhile, we know from the proof of Lemma \ref{samedsc} that $\Delta_2$ has 4 distinct roots if and only if $a_2b_2c_2e_2\neq 0$. Thus $a_1b_1c_1e_1=0\Rightarrow a_2b_2c_2e_2=0$. When $a_1b_1c_1e_1\neq 0$, with the same reason,
    $$
    \Delta_1'(y_1)=-4 a_1 c_1 y_1^4 + (1 - 4 a_1 e_1- 4 b_1 c_1) y_1^2 - 4 b_1 e_1=-4 a_1 c_1(y_1-\xi)(y_1+\xi)(y_1-\eta)(y_1+\eta).
    $$
    has 4 distinct roots. Notice that $\Delta_1(x_2)=(1-F_1 x_2)^4\Delta_1'\left(\frac{x_2+F_1}{1-F_1 x_2}\right)$. Therefore $\Delta_1$ has at least 3 distinct roots,\footnote{Number 3 occurs only if $\Delta_1'(\frac{-1}{F_1})=0$.} which implies the degree of $x_2$ in $\Delta_2$ is at least 3, i.e. $a_2b_2c_2e_2\neq 0$, and hence $a_1b_1c_1e_1=0 \Leftrightarrow a_2b_2c_2e_2=0$. Next we need to show that $a_1b_1c_1e_1=0$ cannot happen, otherwise $a_2b_2c_2e_2=0$ and
    $$
    \Delta_2=x_2^2(-4 a_2 b_2 x_2^2 + (1 - 4 a_2 e_2- 4 b_2 c_2)) \text{ or } (1 - 4 a_2 e_2- 4 b_2 c_2) x_2^2 - 4 c_2 e_2.
    $$
    Clearly, if we ignore the square factor, there is no linear term in $\Delta_2$. On the other hand,
    $$
    \Delta_1=\left\{\begin{array}{l}
    (1-F_1x_2)^2[(1 - 4 a_1 e_1- 4 b_1 c_1) (x_2+F_1)^2 - 4 b_1 e_1 (1-F_1x_2)^2] \text{ or}\\
    (x_2+F_1)^2[-4 a_1 c_1 (x_2+F_1)^2 + (1 - 4 a_1 e_1- 4 b_1 c_1) (1-F_1x_2)^2]. \\
    \end{array}\right.
    $$
    Likewise, there should be no linear term in
    $$
    \left\{\begin{array}{l}
    (1 - 4 a_1 e_1- 4 b_1 c_1) (x_2+F_1)^2 - 4 b_1 e_1 (1-F_1x_2)^2 \text{ or}\\
    -4 a_1 c_1 (x_2+F_1)^2 + (1 - 4 a_1 e_1- 4 b_1 c_1) (1-F_1x_2)^2 \\
    \end{array}\right.
    $$
    since the non-square factors should vanish in the fraction $\frac{\Delta_1}{\Delta_2}=f^2$. In either of the cases, we have $(a_1-b_1)(e_1-c_1)=\frac{1}{4}$ which contradicts to Definition \ref{dfnpoly}. So we can conclude $a_1b_1c_1e_1a_2b_2c_2e_2\neq 0$.
    
    From $a_1b_1c_1e_1a_2b_2c_2e_2\neq 0$ we know that both $\Delta_1$ and $\Delta_2$ are in degree 4. Since $\Delta_2$ has distinct roots, $f^2=\frac{\Delta_1}{\Delta_2}$ implies $\frac{\Delta_1}{\Delta_2}\in \cc$. So, like $\Delta_2$, there should be no linear or cubic term in $\Delta_1$, which brings us
    \begin{equation}\label{a=bb=a}
        \left\{\begin{array}{l}
                (1-4a_1e_1-4b_1c_1+8a_1c_1)=F_1^2(1-4a_1e_1-4b_1c_1+8b_1e_1), \\
                (1-4a_1e_1-4b_1c_1+8a_1c_1)F_1^2=(1-4a_1e_1-4b_1c_1+8b_1e_1). \\ 
                \end{array}\right.
    \end{equation}
    Given $F_1^2\neq -1$, we must have 
    $$
    (1-4a_1e_1-4b_1c_1+8a_1c_1)=(1-4a_1e_1-4b_1c_1+8b_1e_1),
    $$
    i.e. $\frac{a_1}{e_1}=\frac{b_1}{c_1}$. Therefore equation (\ref{a=bb=a}) is equivalent to 
    $$
    (4(a_1-b_1) (e_1-c_1)-1)=(4(a_1-b_1) (e_1-c_1)-1)F_1^2
    $$
    where $F_1=\pm 1$ follows (see the inequalities in Definition \ref{dfnpoly}). Besides, the coefficients of even-degree terms in $\Delta_1$ and $\Delta_2$ should be proportional
    \begin{equation}\label{-ratio}
        \frac{1- 4 a_1 e_1- 4 b_1 c_1-8 a_1 c_1}{-4a_2 b_2}=\frac{8 a_1 e_1+8 b_1 c_1-48 a_1 c_1-2}{1-4a_2e_2-4b_2 c_2}=\frac{1- 4 a_1 e_1- 4 b_1 c_1-8 a_1 c_1}{-4c_2 e_2}    
    \end{equation}
    hence $\frac{a_2}{e_2}=\frac{c_2}{b_2}$. Finally, the above's first equality is equivalent to the last equation of system (\ref{eq-c4-g}). 
    
    Conversely, it is easy to verify that when system (\ref{eq-c4-g}) holds, $\exists f\in \cc \subset K_2$ such that $f^2=\frac{\Delta_1}{\Delta_2}$ so $S$ is equimodular by Lemma \ref{samedsc-g}.
\end{proof}

\begin{proof}(Lemma \ref{n-is-g})
    The statement has already been proved in Theorem \ref{PR} for components of \textbf{Case 1}. Suppose the component is of \textbf{Case 3}, so $\Tilde{g}^{(1)}, g^{(2)}$ are irreducible. For pseudo-planar type, this lemma was contained in the proof of Corollary \ref{n-is-c} so we only consider the case when $F_1 \notin \{0, \infty\}$. Since $y_2\in K_1$, we may assume $y_2=\frac{f}{g}$ where $f,g\in \cc[x_1]$ and $\gcd(f,g)=1$. Due to the uniqueness of the minimal polynomial of $x_2$ on $K_1$, the coefficients in $\Tilde{g}^{(1)}$ and $g^{(2)}$ should be proportional
    \begin{equation}\label{eq-ratio-g''}
    \frac{h_2(x_1)}{a_2 y_2^2 + b_2}=\frac{h_1(x_1)}{y_2}=\frac{h_0(x_1)}{c_2 y_2^2 + e_2}\Leftrightarrow \frac{h_2(x_1)}{a_2 f^2 + b_2 g^2}=\frac{h_1(x_1)}{f g}=\frac{h_0(x_1)}{c_2 f^2 + e_2 g^2}
    \end{equation}
    where $h_i(x_1)$ are given in equation (\ref{h012}). From Corollary \ref{redu=iso} we know at most one of $\{a_2,...,e_2\}$ can be zero so $\gcd(a_2 f^2 + b_2 g^2,f g,c_2 f^2 + e_2 g^2)=1$. On the other hand, by Proposition \ref{2g_i} we also have $\gcd(h_2,h_1,h_0)=1$.
    Thus the right chain of (\ref{eq-ratio-g''}) is equal to a nonzero constant which implies the degrees of $f$ and $g$ are at most 1, i.e. $y_2=\frac{px_1+q}{rx_1+s}$.
    Finally, $y_2$ is not a constant according to the left chain of (\ref{eq-ratio-g''}) so $ps-qr\neq 0$.
\end{proof}

\begin{proof}(Proposition \ref{real-sol})
    The first step is to exclude general type. Suppose $F_1\notin \{0, \infty\}$. In the extension diagram of $W$, we take $x_2$ as the parameter and adapt the result in Lemma \ref{samedsc-g'}. Given $a_1c_1=b_1e_1$ we can conclude that 
    \begin{equation}\label{+1}
    1- 4 a_1 e_1- 4 b_1 c_1-8 a_1 c_1>0.
    \end{equation}
    In fact, recall Definition \ref{dfnpoly} we have
    $$
    1- 4 a_1 e_1- 4 b_1 c_1+8 a_1 c_1=1- 4(a_1-b_1)(e_1-c_1)>0 \Rightarrow 1- 4 a_1 e_1- 4 b_1 c_1>-8 a_1 c_1.
    $$
    The opposite of equation (\ref{+1}) leads to $1- 4 a_1 e_1- 4 b_1 c_1\leq 8 a_1 c_1$ so
    $$
    (1- 4 a_1 e_1- 4 b_1 c_1)^2 \leq (8 a_1 c_1)^2 =64a_1b_1c_1e_1,
    $$
    contradicts to Definition \ref{dfnpoly}. Similarly, $1- 4 a_2 e_2- 4 b_2 c_2-8 a_2 b_2>0$ must hold given $a_2b_2=c_2e_2$. In the meanwhile, Lemma \ref{samedsc-g'} also requires 
    $$
    (1-4a_1e_1-4b_1c_1-8a_1c_1)(1-4a_2e_2-4b_2c_2-8a_2b_2)=(16a_1c_1)(16a_2b_2)>0.
    $$
    Symmetrically, we may assume $(1-4a_1e_1-4b_1c_1-8a_1c_1)\leq |16a_1c_1|$ so 
    $$
    (1-4a_1e_1-4b_1c_1-8a_1c_1)^2-(16a_1c_1)^2=(1- 4 a_1 e_1- 4 b_1 c_1+8 a_1 c_1)(1-4a_1e_1-4b_1c_1-24a_1c_1)\leq 0.
    $$
    Again, by Definition \ref{dfnpoly}  we know $1- 4 a_1 e_1- 4 b_1 c_1+8 a_1 c_1>0$ so $1-4a_1e_1-4b_1c_1-24a_1c_1\leq 0$ which indicates $a_1c_1 >0$ and hence $a_2b_2>0$. Now we go back to the proof of Lemma \ref{samedsc-g'}, equation (\ref{-ratio}) tells us $\frac{\Delta_1}{\Delta_2}<0$ which means $x_1$ and $x_3$ cannot be both real except for finitely many values of $x_2$. Thus $F_1\in \{0, \infty\}$.
    
    The next step is to exclude the last two of systems (\ref{eq-c3or4}). We may assume $F_1=0$ since the proof for $F_1=\infty$ is quite similar. Suppose $b_1e_1=a_2b_2=0$, so
    $$
    (1-4a_1e_1-4b_1c_1)(1-4a_2e_2-4b_2c_2)=(4a_1c_1)(4c_2e_2).
    $$
    Symmetrically, we may assume $|1-4a_1e_1-4b_1c_1|\leq |4a_1c_1|$ and $|1-4a_2e_2-4b_2c_2|\geq |4c_2e_2|$. Again, recall Definition \ref{dfnpoly} we have
    $$
    1- 4 a_1 e_1- 4 b_1 c_1+4 a_1 c_1+0=1- 4(a_1-b_1)(e_1-c_1)>0 \Rightarrow 1- 4 a_1 e_1- 4 b_1 c_1>-4 a_1 c_1,
    $$
    $$
    \left.\begin{array}{l}
        1- 4 a_1 e_1- 4 b_1 c_1>-4 a_1 c_1,\\
        |1-4a_1e_1-4b_1c_1|\leq |4a_1c_1|.\\         
    \end{array}\right\} \Rightarrow 4a_1c_1>0.
    $$
    Same technique can show $1-4a_2e_2-4b_2c_2>0$ which lead to $\frac{\Delta_1}{\Delta_2}=\frac{-4a_1c_1x_2^2}{1-4a_2e_2-4b_2c_2}\leq 0$. So there are only finitely many real points in $Z(S)$ when $b_1e_1=a_2b_2=0$. For the same reason, we cannot have $a_1c_1=c_2e_2=0$ either.
    
    Finally, since $x_1, y_2$ should be real at the same time, we must have 
    $$
    \frac{\Delta_1}{\Delta_2}=\frac{a_1 c_1}{a_2 b_2}=\frac{1- 4 a_1 e_1- 4 b_1 c_1}{1 - 4 a_2 e_2- 4 b_2 c_2}=\frac{b_1 e_1}{c_2 e_2}>0.
    $$
\end{proof}

\begin{proof}(Lemma \ref{samemini-g})
    The approach is quite similar to Lemma \ref{samemini}. We first suppose $S$ only admits components of \textbf{Case 5}. Hence in every the extension diagram, we have $x_2 \notin K_1(y_2)$ given $K_1(y_2)\neq K_1(x_2,y_2)$. Since the minimal polynomial of $x_2$ on $K_1(y_2)$ is unique, it can be obtained from both equalities of the left chain of (\ref{eq-ratio-g''}), i.e.
    $$
    \left\{\begin{array}{l}
        a_2 h_1(x_1) y_2^2 - h_2(x_1) y_2 + b_2 h_1(x_1) = 0,\\
        c_2 h_1(x_1) y_2^2 - h_0(x_1) y_2 + e_2 h_1(x_1) = 0.\\         
    \end{array}\right.
    $$
    This immediately gives
    $$
    \frac{a_2}{c_2}=\frac{h_2(x_1)}{h_0(x_1)}=\frac{b_2}{e_2}\in \cc.
    $$
    According to the linear terms of $h_2$ and $h_0$, these identical ratios are simply $-1$. Similarly, from the rest coefficients of $h_2$ and $h_0$ we obtain
    $$
    \frac{a_1}{b_1}=\frac{c_1}{e_1}=\frac{a_2}{c_2}=\frac{b_2}{e_2}=-1.
    $$
    In addition, when $F_1=\pm 1$, in order to keep the component away from \textbf{Case 3} and \textbf{Case 4}, we need to break system (\ref{eq-c4-g}), i.e. $256a_1c_1a_2b_2\neq 1$.

    Conversely, if one of systems (\ref{eq-c5-g}) holds, simple calculation shows
    $$
    \text{Res}(\Tilde{g}^{(1)},g^{(2)};x_2)=-(a_2 h_1(x_1) y_2^2 - h_2(x_1) y_2 + b_2 h_1(x_1))^2.
    $$
    Gien Lemma \ref{samedsc-g'}, $S$ cannot have components of \textbf{Case 3} or \textbf{Case 4} since each one of systems (\ref{eq-c5-g}) is against systems (\ref{eq-c4-g}). As a consequence, $a_2 h_1(x_1) y_2^2 - h_2(x_1) y_2 + b_2 h_1(x_1)$ is irreducible and quadratic in $y_2$, i.e. $S$ only admits a component of \textbf{Case 5}.
\end{proof}

\begin{proof}(Theorem \ref{symext})
    Since the theorem becomes trivial when $i=j$, we first assume $i=1, j=2$. According to Lemma \ref{rational} and Proposition \ref{2g_i}, $[K_1(x_{2}):K_1]=1$ if and only if $\Tilde{g}^{(1)}$ is reducible if and only if $[K_2(x_1):K_2]=1$. Since $[K_1(x_{2}):K_1]$ and $[K_2(x_1):K_2]$ are maximum 2, we also have $[K_1(x_{2}):K_1]=2 \Leftrightarrow [K_2(x_1):K_2]=2$. So $[K_i(x_j):K_i]=[K_j(x_i):K_j]$ always holds for adjacent pairs $(i,j)$. 
    
    Now we consider the case $i=1, j=3$. Lemma \ref{symrational} tells us $[K_1(x_3):K_1]=1 \Leftrightarrow [K_3(x_1):K_3]=1$. Further, given $[K_1(x_3):K_1]\in \{1, 2, 4\}$, we only need to prove $[K_1(x_3):K_1]=2 \Leftrightarrow [K_3(x_1):K_3]=2$ so that $[K_1(x_3):K_1]=4 \Leftrightarrow [K_3(x_1):K_3]=4$ holds automatically.
    
    For \textbf{Case 2}, which is defined in a symmetric way in Definition \ref{casedefn}, we always have $[K_1(x_3):K_1]=[K_1(y_2):K_1]=[K_3(x_1):K_3]=2$. For \textbf{Case 5}, Lemma \ref{samemini} and \ref{samemini-g} provides irreducible annihilating polynomials of $(x_1,y_2)$ which are quadratic in both $x_1, y_2$ hence $[K_3(x_1):K_3]=2$. As for \textbf{Case 4}, $g^{(2)}$ is irreducible so $[K_3(x_2):K_3]=2$. By Lemma \ref{samedsc-g}, $x_1 \in K_3(x_2)$. Again, by Lemma \ref{symrational}, $[K_1(x_3):K_1]=2 \Rightarrow x_3 \notin K_1 \Leftrightarrow x_1 \notin K_3$ and $[K_3(x_1):K_3]=2$ follows.
    
    We just proved $[K_1(x_3):K_1]=2 \Rightarrow [K_3(x_1):K_3]=2$. Conversely, if we reverse all the arguments and apply them on $(\Tilde{g}^{(2)},\Tilde{g}^{(1)})$, we will obtain $[K_3(x_1): K_3]=2 \Rightarrow [K_1(x_3): K_1]=2$ and conclude that $[K_3(x_1): K_3]=2 \Leftrightarrow [K_1(x_3): K_1]=2$, hence $[K_i(x_j): K_i]=[K_j(x_i): K_j]$ for all $1\leq i, j \leq 3$.
    
    Finally, take \textbf{Case 2} of Definition \ref{casedefn} as an example. It was defined in a symmetric way based on the extension diagrams with respect to $x_1$ and $x_3$. It is not hard to realize that, for a component of any other \textbf{Case}, if we consider its extension diagram with respect to $x_3$ rather than $x_1$, the extension degrees will lead us to exactly the same \textbf{Case} as it was. This proves the last statement of the theorem.
\end{proof}

\begin{proof}(Proposition \ref{4dg})
    Rewrite $\Tilde{g}^{(1)}$ as $h_2(x_1)x_2^2+h_1(x_1)x_2+h_0(x_1)$. The term of $R_1$ in the highest degree of $y_2$ is $((c_2 h_2-a_2 h_0)^2+a_2c_2h_1^2)y_2^4$. Now suppose the coefficient vanishes, so it is either $c_2 h_2-a_2 h_0-\sqrt{-a_2c_2}h_1=0$ or $c_2 h_2-a_2 h_0+\sqrt{-a_2c_2}h_1=0$.\footnote{$\sqrt{-a_2c_2}\in\{re^{i\theta}:r\geq 0, \theta \in [0,\pi)\}$.} Clearly, $a_2, c_2$ cannot be both zero since $g^{(2)}$ is non-singular.  Also, according to Proposition \ref{2g_i}, none of $h_2, h_0$ could be zero. Therefore $a_2, c_2$ are both nonzero and hence $h_1$ is a linear combination of $h_2$ and $h_0$. This is obviously impossible when $F_1$ is 0 or $\infty$. So we only need to consider the case when $0 \neq F_1 \in \rr$ ($h_i$ goes as equation (\ref{h012})), which can be converted to an eigenvector problem
    $$
    \left (\begin{array}{cccc}
    b_1 F_1^2 +a_1 &2 F_1 (a_1 - b_1) & a_1 F_1^2 +b_1 \\
    -F_1 & 1 - F_1^2 & F_1   \\
    e_1 F_1^2 +c_1 & 2 F_1 (c_1 - e_1) & c_1 F_1^2 +e_1 \\
    \end{array}\right)
    \left (\begin{array}{cccc}
    c_2 \\
    \pm \sqrt{-a_2c_2}   \\
    -a_2 \\
    \end{array}\right)=0.
    $$
    It is not hard to get the determinant $(1+F_1^2)^3(a_1 e_1-b_1 c_1)$ which should be 0. Given $F_1^2 \neq -1$, $(a_1, b_1)$ and $(c_1, e_1)$ must be linear dependent. On the other hand, set
    $$
    \left\{\begin{array}{l}
    \Tilde{h}_2(x)=(b_1 F_1^2 +a_1) x^2 + 2 F_1 (a_1 - b_1) x + a_1 F_1^2 +b_1, \\
    \Tilde{h}_1(x)=-F_1x^2+(1 - F_1^2)x + F_1=(x+F_1)(1-F_1x). \\
    \end{array}\right.
    $$
    The eigenvector gives a common solution $x=\pm \sqrt{-c_2/a_2}$ of $\Tilde{h}_2(x)=0$ and $\Tilde{h}_1(x)=0$. According to the factors of $\Tilde{h}_1(x)$, we have either $\Tilde{h}_2(-F_1)=0$ or $\Tilde{h}_2(1/F_1)=0$, which implies $b_1=0$ or $a_1=0$. Recall $(a_1, b_1)$ and $(c_1, e_1)$ are collinear, we have either $b_1=e_1=0$ or $a_1=c_1=0$, contradicts to the non-singularity of $\Tilde{g}^{(1)}$. Therefore, $R_1(x_1,y_2)$ is in 4th degree of $y_2$. Finally, by Corollary \ref{samedg}, $R_1(x_1,y_2)$ is in 4th degree of $x_1$ as well.
\end{proof}

\begin{proof}(Theorem \ref{factor}, complementary)
    Suppose $F_1\notin \{0,\infty\}$ and $W$ is a component of \textbf{Case 3}. Lemma \ref{n-is-g} says $y_2=\frac{px_1+q}{rx_1+s}$ on $W$ which implies $(rx_1 y_2 - px_1 + sy_2 - q)|R_1$. In particular, if $S$ is rational-quadratic, $R_1$ should also have an irreducible factor which is quadratic in $x_1$ and $y_2$. Hence the only thing left is to prove that 
    $$
    (rx_1 y_2 - px_1 + sy_2 - q)|R_1\Rightarrow (rx_1 y_2 - px_1 + sy_2 - q)^2|R_1.
    $$
    Given $S$ is equimodular, we adopt system (\ref{eq-c4-g}) and equation (\ref{h012}) so
    $$
    \left\{\frac{e_1}{a_1}=\frac{c_1}{b_1}=\omega_1^2, \frac{e_2}{a_2}=\frac{b_2}{c_2}=\omega_2^2, F_1=\pm 1, h_1=2F_1(a_1-b_1)(x_1-\omega_1)(x_1+\omega_1)\right\}.
    $$
    We are using $\{a_1,b_1,\omega_1,\omega_2\}$ to reparametrize the others so the signs of $\omega_i$ are not our concern.
    Review the proof of Lemma \ref{n-is-g} we have $\frac{h_1(x_1)}{(px_1+q)(rx_1+s)}\in\cc$ hence $y_2=k\frac{x_1-\omega_1}{x_1+\omega_1}$. Plug $f=k(x_1-\omega_1), g=(x_1+\omega_1)$ into equation (\ref{eq-ratio-g''}) we obtain
    $$
    \left\{k=\pm \omega_2, a_2=\frac{2\omega_1(a_1+b_1)+1}{8k\omega_1(a_1-b_1)},c_2=\frac{2\omega_1(a_1+b_1)-1}{8k\omega_1(a_1-b_1)}\right\}.
    $$
    Now symbolic computation is able to factorize $R_1$ after all coefficients being parametrized by $\{a_1,b_1,\omega_1,\omega_2\}$.
\end{proof}

\begin{proof}(Proposition \ref{IQpseu})
    Suppose $S_i$ is involutive-quadratic so its valid component must be of \textbf{Case 5}. From Lemma \ref{samemini} we know that the irreducible factor of $\Tilde{R}^{(i)}$ is quadratic in both $x_1, x_3$ and only contains odd-degree terms. On the contrary, going through pseudo-planar couplings of other classes, the irreducible factors of the resultant, if quadratic, only contain even-degree terms (see Example \ref{HQ} and Corollary \ref{even-term}). Hence the other coupling must be involutive-quadratic as well.
\end{proof}

\end{document}